\documentclass[reqno, a4paper, 10pt, oneside]{amsart}

\usepackage{amsmath, amssymb, amsthm, mathrsfs, enumerate, scalerel, bbm}
\usepackage[utf8x]{inputenc}
\usepackage[T1]{fontenc}
\usepackage[english]{babel}

\usepackage{geometry}
\usepackage{hyperref}
\hypersetup{colorlinks,linkcolor=blue,citecolor=blue,filecolor=red,urlcolor=blue}

\usepackage[margin = 1cm]{caption}
\usepackage[margin = 0.5cm]{subcaption}
\usepackage{float} 
\usepackage{scalerel} 

\makeatletter
\renewcommand{\@biblabel}[1]{\quad#1.}
\makeatother

\usepackage{fancyhdr}
\newtheorem{theorem}{Theorem}

\newtheorem{corollary}[theorem]{Corollary}
\newtheorem{definition}[theorem]{Definition}
\newtheorem{lemma}[theorem]{Lemma}
\newtheorem{proposition}[theorem]{Proposition}

\theoremstyle{remark}
\newtheorem{remark}[theorem]{Remark}
\newtheorem{example}[theorem]{Example}

\numberwithin{equation}{section}
\numberwithin{theorem}{section}

\newcommand{\E}{\mathbb{E}}
\newcommand{\R}{\mathbb{R}}

\newcommand{\Ey}{\mathbb{E}_y} 
\newcommand{\Exy}{\mathbb{E}^x_y}

\newcommand{\1}{\mathbbm 1} 
\newcommand{\di}{\mathrm{d}}
 
\renewcommand{\epsilon}{\varepsilon}

\DeclareMathOperator{\sgn}{sgn}
\newcommand{\ind}{\text{\ensuremath{1\hspace*{-0.9ex}1}}}

\newcommand{\RMST}{\mathrm{RMST}}
\newcommand{\W}{\mathbb{W}}

\newcommand{\itembullet}{\,\begin{picture}(-1,1)(-1,-3)\circle*{2}\end{picture}\ }

\title[ ]{Delayed Switching Identities and Multi-Marginal Solutions to the Skorokhod Embedding Problem}
\author{Alexander M. G. Cox}
\address{Department of Mathematical Sciences, University of Bath}
\email{a.m.g.cox@bath.ac.uk}

\author{Annemarie M. Grass}
\address{Faculty of Mathematics, University of Vienna}
\email{annemarie.grass@univie.ac.at}
\date{\today.}

\thanks{\textit{Funding:}
A.~Grass gratefully acknowledges financial support through FWF project P 35197 \textit{(Grant DOI: 10.55776/P35197)}}
\begin{document}
\maketitle
\begin{abstract}
In this article, we consider a generalisation of the Skorokhod embedding problem (SEP) with a \emph{delayed} starting time. In the delayed SEP, we look for stopping times which embed a given measure in a stochastic process, which occur after a given delay time. Our first contribution is to show that the switching identities introduced in a recent paper of Backhoff, Cox, Grass and Huesmann extend to the case with a delay.

We then show that the delayed switching identities can be used to establish an optimal stopping representation of Root and Rost solutions to the multi-marginal Skorokhod embedding problem.  We achieve this by rephrasing the multi-period problem into a one-period framework with delay. This not only recovers the known multi-marginal representation of Root, but also establishes a previously unknown optimal stopping representation associated to the multi-marginal Rost solution. The Rost case is more complex than the Root case since it naturally requires randomisation for general initial measures, and we develop the necessary tools to develop these solutions. Our work also provides a comprehensive and complete treatment of \emph{discrete} Root and Rost solutions, embedding discrete measures into simple symmetric random walks. 

\medskip


\medskip

\end{abstract}

%
\section{Introduction} \label{sec:intro}
We consider the \emph{Skorokhod embedding problem}, that is  given a 
Brownian motion $(W_t)_{t \geq 0}$ started according to an initial distribution $W_0 \sim \lambda$, 
and a probability distribution $\mu$ on $\mathbb{R}$, 
the question is to find a 
stopping time $\tau$ such that
\begin{equation}\label{SEP} 
    W_\tau \sim \mu \text{ and } (W_{t \wedge \tau})_{t \geq 0} \text{ is uniformly integrable.} \tag{{\sf SEP}} 
\end{equation} 
While this problem was originally posed and solved by Skorokhod \cite{Sk65,Sk61} in the early 60s, 
it remains an active field of research to the present day. 
The 2004 survey paper \cite{Ob04} features more than 20 distinct solution to \eqref{SEP} 
and more recent contributions include 
\cite{ CoHo07, CoWa12, CoWa13, GaMiOb14, ObDoGa14, HeObSpTo16, ObSp17, BeCoHu14, 
BeCoHu16, GhKiPa19, CoObTo19, BaBe20, GhKiLi20, BaZh21, GaObZo21, BeNuSt19, CoKi19b} 
among others. 

Of primary interest in this article will be an extension of the classical \eqref{SEP} to \emph{multiple marginals}. 
Given a sequence of measures $\mu_1, \dots, \mu_n$ on the real line, 
the \emph{multi-marginal Skorokhod embedding problem} is to find an increasing sequence of stopping times $\tau_1 \leq \dots \leq \tau_n$ such that
\begin{equation} \label{MMSEP} \tag{{\sf MMSEP}}
        W_{\tau_k} \sim \mu_k, \, k = 1, \dots, n \text{ and } (W_{t \wedge \tau_n})_{t \geq 0} \text{ is uniformly integrable}.
\end{equation}
It is an application of Strassen's Theorem \cite{St65} and a well know fact that \eqref{MMSEP} 
will admit a solution if and only if the measures are in increasing convex order,
that is  $\lambda \leq_c \mu_1 \leq_c \cdots \leq_c \mu_n$. 
Throughout the paper we will assume this condition satisfied.

This extension to \eqref{MMSEP} is particularly motivated by applications in robust mathematical finance given the connections between the \eqref{SEP} and robust option pricing first proposed in Hobson's paper \cite{Ho98a}; see also \cite{Ho11} for an extensive survey.  These connections enable the computation of robust bounds on option prices while considering a market-given marginal distribution at the time of expiration.  Extending this theory to multiple marginals would allow the incorporation of additional market data into the robust pricing problem. We develop this approach in a companion paper, \cite{CoGr23b}.

\medskip
The focal point of this paper are two fundamental and influential solutions to \eqref{SEP} 
known as Root \cite{Ro69} and Rost \cite{Ro76} barrier solutions. 
These stopping times can be represented as hitting times of subsets $R \subseteq \R_+\times \R$,  
which are subject to some type of \emph{barrier structure}, precisely:
\begin{align}
\begin{split} \label{def:BarrierSet}
    &\itembullet \text{ If } (t,x) \in R^{Root} \text{ then } (s,x) \in R^{Root}  \text{ for all } s > t.
\\  &\itembullet \text{ If } (t,x) \in R^{Rost} \text{ then } (s,x) \in R^{Rost}  \text{ for all } s < t.
\end{split}
\end{align}
See Figure \ref{fig:intro-examples} for an illustration. 
\begin{figure}
\captionsetup{margin=2cm}
\centering
\begin{subfigure}{.44\linewidth}
\centering
\includegraphics[width = 0.49\linewidth]{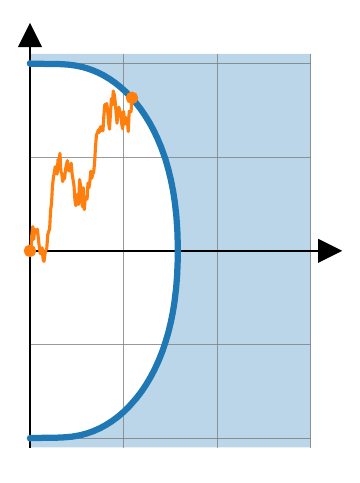}
\centering
\includegraphics[width = 0.49\linewidth]{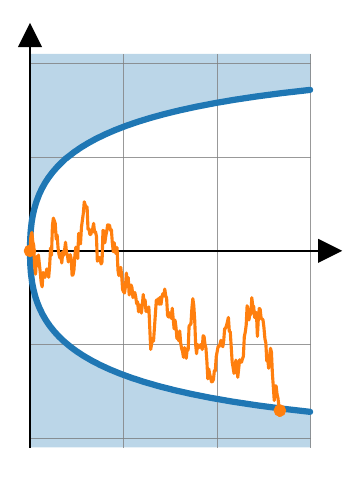}
\caption{Root (left) and Rost (right) barriers solving \eqref{SEP} for 
$\lambda = \delta_0$ and $\mu = \mathsf{U}[-2,2]$.}
\end{subfigure}
\quad
\begin{subfigure}{.44\linewidth}
\centering
\includegraphics[width = 0.49\linewidth]{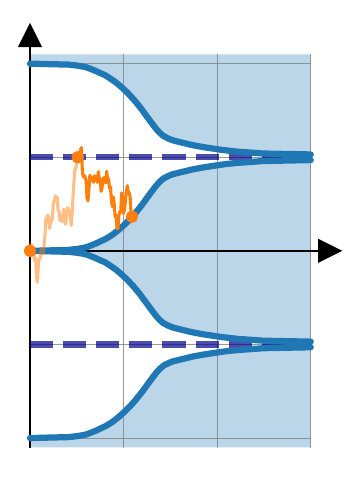}
\centering
\includegraphics[width = 0.49\linewidth]{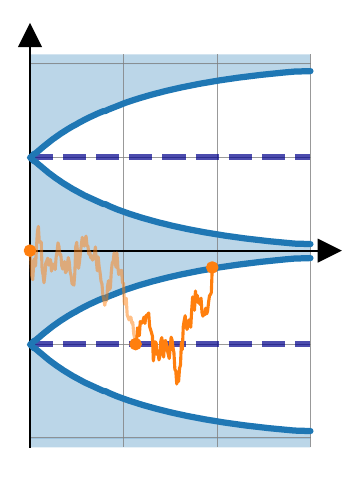}
\caption{Root (left) and Rost (right) barriers solving \eqref{MMSEP} for 
$\lambda = \delta_0$, $\mu_1 = \frac{1}{2}(\delta_{-1} + \delta_{1})$ and $\mu_2 = \mathsf{U}[-2,2]$.}
\end{subfigure}

\caption{Single-marginal (a) and multi-marginal (b) Root and Rost solutions embedding the measure $\mathsf{U}[-2,2]$, 
one of the examples specifically mentioned \cite{Ro69} for its previously unknown barrier structure.}
\label{fig:intro-examples}
\end{figure}

Our interest in Root and Rost solutions is twofold. 
First, there is a substantial interest in the probability literature on these solutions, see for example 
\cite{Lo70, Ro76, CoWa12, CoWa13, GaMiOb14, ObDoGa14, BeHuSt16, De18, CoObTo19, GaObZo21, RiTaTo20, BaBe20, BrHu21}.
Second, such solutions are related to robust bounds on variance options considering multiple market-given marginal distributions. 
In \cite{CoGr23b} we carry out an empirical exploration of such robust bounds based on the results presented here. 
See \cite{Ne94, CaLe10, CoWa12, CoWa13, HoKl12} for related work in the financial context. 

While existence results of Root and Rost barriers were shown by their originators 
and also an extension 
to multiple marginals was given in \cite{BeCoHu16}, 
as was noted in Root's original work \cite{Ro69}, concrete barriers are only known in the most trivial of situations.   
To actually construct a concrete barrier $R$ given a probability measure $\mu$ 
or even more so a sequence barriers given a sequence of probability measures $\mu_1, \dots, \mu_n$ is a highly non trivial task. 

Prior research in \cite{CoWa12, CoWa13, GaMiOb14, ObDoGa14, De18, GaObZo21} established 
that Root and Rost barriers which solve the (single-marginal) \eqref{SEP} 
can be recovered via solutions to a respective associated PDE in variational form. 
Solutions to such PDEs, in turn, are known to have a representation by means of an optimal stopping problem. 
While for a long time this optimal stopping representation of Root and Rost solutions was
exclusively accessible through this detour over PDEs in variational form, 
a direct proof utilizing time-reversal arguments was given in \cite{CoObTo19}, 
notably for the multi-marginal Root case. 
It was furthermore the aim of the article \cite{BaCoGrHu21} to give a simple probabilistic argument connecting Root and Rost solutions to their respective optimal stopping representations in the single marginal case by establishing a specific switching identity. 
To the best of our knowledge, an optimal stopping representation associated to the multi-marginal Rost problem remained unknown to this day.

In order to establish such a multi-marginal representation, 
a crucial first step is to reduce the multi-step viewpoint to a single-marginal perspective where 
the initial marginal is now a space-time law instead of a spatial law at zero.
This approach builds on an approach implicit in \cite{CoObTo19}, and which we make explicit in our approach. Our solution is computed in terms of a \emph{delay} stopping time inducing an initial space-time law after which we embed with a Root or Rost barrier respectively.  Given such a representation it becomes apparent how a multi-marginal solution can be recovered inductively.

An initial step in our process is to extend the results established in \cite{BeCoHu16} as well as \cite{BaCoGrHu21} to allow for such a delay.  The implications of this extension are twofold: First, it enables a more concise and intuitive recovery of the multi-marginal optimal stopping problem given in \cite{CoObTo19}.  Second, we can provide a novel characterisation of the multi-marginal Rost embedding in terms of the solution to a given optimal stopping (or multiple stopping) problem.  This lays the foundation for constructing Rost barriers such that their hitting times solve the \eqref{MMSEP} computable as demonstrated in Figure \ref{fig:intro-examples}.

We briefly review the known results for the single marginal problem and then proceed to state its multi-marginal extension.

%
%
%
%
%
\subsubsection*{Summary of Single Marginal Results} 
Simple probabilistic arguments justifying the optimal stopping characterization of the single marginal Root and Rost \eqref{SEP} were given in \cite{BaCoGrHu21}. 
Let us denote by $\mu^{Root}$ (resp. $\mu^{Rost}$) the law of the Brownian motion started with distribution $\lambda$ 
at the time it hits $R^{Root}$ (resp. $R^{Rost}$).
By $\mu_T^{Root}$ (resp. $\mu_T^{Rost}$) we denote the time this Brownian motion hits the barrier 
$R^{Root} \cup \left([T, \infty) \times \mathbb{R} \right)$ 
(resp. $R^{Rost} \cup \left([T, \infty) \times \mathbb{R} \right)$). 
Recall the potential of a measure $m$, denoted by 
\[U_{m}(y):=-\int|y-x|\, m(\di x).\]
The single-marginal relations of interest, as provided in \cite{CoWa12, CoWa13, BaCoGrHu21, GaObZo21} among others 
are
\begin{align} \label{eq:Root}
    U_{\mu_{T}^{Root}}(y)   &= \E_y \left [ U_{\mu^{Root}}\left(W_{\tau^{*} }\right)\1_{\tau^{*}<T} + U_{\lambda}\left(W_{\tau^{*}}\right)\1_{\tau^{*}=T}  \right ] 
\\\label{eq:RootOSP}
                            &= \sup\limits_{ \tau\leq T}\E_y \left [ U_{\mu^{Root}}(W_{\tau})\1_{\tau<T} + U_{\lambda}(W_{\tau})\1_{\tau=T}  \right ] , 
\end{align}
where the optimizer is $\tau^*:=\inf\left\{t \geq 0: (T-t,W_t) \in R^{Root} \right\}\wedge T$, \\
and
\begin{align} \label{eq:Rost}
 U_{\mu^{Rost}}(y)- U_{\mu_T^{Rost}}(y) &= \E_y \left[ \left(U_{\mu^{Rost}}-U_{\lambda}\right)(W_{\tau^*})  \right] 
 \\\label{eq:RostOSP}
                                        &= \sup\limits_{\tau\leq T}\E_y \left[ \left(U_{\mu^{Rost}}-U_{\lambda}\right)(W_{\tau})  \right],
\end{align} 
where the optimizer is $\tau^*:=\inf\left\{t \geq 0: (T-t,W_t) \in R^{Rost} \right\}\wedge T$.

%
%
%
%
%
\subsubsection*{Multi-Marginal Extensions} 
We are interested in Equations \eqref{eq:Root} - \eqref{eq:RostOSP} in a multi marginal setting.
While the multi marginal extension of the Root case \eqref{eq:Root} - \eqref{eq:RootOSP} is known due to \cite{CoObTo19}, 
the Rost equivalent appears to be novel. 
The results can be stated as follows. 

Let $R_1, R_2, \dots, R_n$ denote a sequence of Root (resp. Rost) barriers 
as defined in \eqref{def:BarrierSet}. 
Then consider the associated Root (resp. Rost) stopping times
\begin{align*}
    \rho_1 &:= \inf\left\{t \geq 0 : (t, W_t) \in R_1 \right\},
\\  \rho_k &:= \inf\left\{t \geq \rho_{k-1} : (t, W_t) \in R_k \right\}, \text{ for } k \in \{2, \dots, n\}.
\end{align*}
Define the respective laws $\mu_i := \mathcal{L}^{\lambda}(W_{\rho_i})$ and $\mu_{i,T} := \mathcal{L}^{\lambda}(W_{\rho_i \wedge T})$.
More precisely, $\mu_i$ will be the law of the Brownian motion starting with distribution $\lambda$ 
at the time it consecutively hit all of the barrier sets $R_1, \dots, R_i$ respectively, 
and $\mu_{i,T}$ will be the time it consecutively hit the barrier sets 
$R_1 \cup \left([T, \infty) \times \mathbb{R} \right), \dots,  R_i \cup \left([T, \infty) \times \mathbb{R} \right)$. 
Then these measures are in convex order, i.e. 
$\mu_1 \leq_c \mu_2 \leq_c \cdots \leq_c \mu_n$ as well as $\mu_{1,T} \leq_c \mu_{2,T} \leq_c \cdots \leq_c \mu_{n,T}$.
Moreover define the function,
\begin{equation}
    U_{\mu_{0,T}}(x) := U_{\mu_{0}}(x) := U_{\lambda}(x)  
        \quad \text{ for all } T\geq0.
\end{equation}
The multi-marginal extension of our relations of interest are given as follows. 
\begin{theorem} \label{thm:MM-Root}
If $(R_k)$ denotes a sequence of Root barriers, then 
\begin{align}
    \label{eq:MMRoot}
    U_{\mu_{k,T}}(y) &= \E_y\left[ U_{\mu_{k-1,T - \tau^*_k}} (W_{\tau^*_k}) 
        + (U_{\mu_k} - U_{\mu_{k-1}})(W_{\tau^*_k})\ind_{\{\tau^*_k  < T\}}\right]   
\\ \label{eq:MMRootOSP}
                            &= \sup_{\tau \leq T} \E_y\left[ U_{\mu_{k-1,T-\tau}} (W_{\tau}) 
        + (U_{\mu_k} - U_{\mu_{k-1}})(W_{\tau})\ind_{\{\tau < T\}}\right],
\end{align}
where the optimizer is $\tau^{*}_k :=\inf\left\{t \geq 0 : (T-t,W_t) \in R_k\right\}\wedge T$.

Moreover, given the interpolating potentials $U_{\mu_{k,T}}(y)$ for all $(T,y) \in [0, \infty) \times \mathbb{R}$ 
we can recover the Root barriers $R_1, \dots, R_n$ embedding the measures $\mu_1, \dots, \mu_n$ in the following way
\[
    R_k = \left\{(T,y) : \left(U_{\mu_{k,T}} - U_{\mu_{k-1,T}} \right)(y) = \left(U_{\mu_{k}} - U_{\mu_{k-1}}\right)(y) \right\}, \,\, k \in \{1, \dots, n\}.
\]
\end{theorem}
\begin{theorem} \label{thm:MM-Rost}
If  $(\bar R_k)$ denotes a sequence of Rost barriers, then
\begin{align}
    \label{eq:MMRost}
    U_{\mu_{k}}(y) - U_{\mu_{k,T}}(y)  &= \E_y\left[ U_{\mu_{k}} (W_{\tau^*_k})  
            - U_{\mu_{k-1,T-\tau^*_k}} (W_{\tau^*_k}) \right]   
\\ \label{eq:MMRostOSP}
        &= \sup_{\tau \leq T} \E_y\left[ U_{\mu_{k}} (W_{\tau}) 
            - U_{\mu_{k-1,T-\tau}} (W_{\tau}) \right]
\end{align}
where the optimizer is $\tau^*_{k}:=\inf\left\{t \geq 0: (T-t,W_t)\in \bar R_k\right\}\wedge T$. 

Moreover, given the interpolating potentials $U_{\mu_{k,T}}(y)$ for all $(T,y) \in [0, \infty) \times \mathbb{R}$ 
we can recover the Rost barriers $\bar R_1, \dots, \bar R_n$ embedding the measures $\mu_1, \dots, \mu_n$ in the following way
\[
    \bar R_k = \left\{(T,y) : U_{\mu_{k,T}}(y)  =  U_{\mu_{k-1,T}}(y)\right\}, \,\, k \in \{1, \dots, n\}.
\]

\end{theorem}
%
%
%
%
%
\subsubsection*{Outline of the Paper}
In Section \ref{sec:delayed} we introduce the notion of \emph{delayed} Skorokhod embeddings, 
which contain the multi-marginal Skorokhod embedding as a special case. 
We will establish a general existence of such solution following the approach in \cite{BeCoHu16} 
and present an optimal stopping representation theorem. 

Inspired by the framework proposed in \cite{BaCoGrHu21}, 
we first provide a analogous statements for simple symmetric random walks in Section \ref{sec:discrete}. 
First we show general existence of delayed Root and Rost solution in this context, 
subsequently we state and prove the corresponding respective optimal stopping representation. 

Section \ref{sec:limit} is dedicated to providing a suitable discretization of the continuous-time framework 
and of the respective optimal stopping representations such that the results of Section \ref{sec:discrete} apply. 
We then take a Donsker-type limit back to a Brownian motion, 
recovering the continuous times results stated in Section \ref{sec:delayed}.
%
%
%
%
%
%
%
%
%
%
%
%
%
%
%
%
%
%
%
%
%
%
%
%
%
%
%
%
%
%
%
%
%
%
%
%
%
%
%
%
%
%
%
%
%
%
%
%
%
%
\section{Delayed Root and Rost Embeddings} \label{sec:delayed}
For a more convenient and comprehensive treatment of \eqref{MMSEP}, 
but also present results in greater generality 
we introduce a relatively recent extension of the classical \eqref{SEP} known as \emph{delayed} embeddings. 
Although the concept of delayed embeddings has been implicitly explored in \cite{CoObTo19},
to best of our knowledge the specific name was first mentioned in \cite{BrHu21}. 

We consider a \emph{delay stopping time} $\eta$ and require our solution to \eqref{SEP} 
to wait for this delay before embedding the measure $\mu$. 

Formally, given a measure $\mu$, a Brownian motion $(W_t)_{t \geq 0)}$ started according to an initial distribution $W_0 \sim \lambda$ 
and a (possibly randomized) delay stopping time $\eta$ the \emph{delayed Skorokhod embedding problem} is to find a stopping time $\tau$ such that 
\begin{equation}\label{dSEP} 
    W_{\tau} \sim \mu \text{ where } \tau \geq \eta \text{ and } (W_{\tau \wedge t})_{t \geq 0} \text{ is uniformly integrable.} \tag{{\sf dSEP}} 
\end{equation}
We will denote a \eqref{dSEP} by the triple $(\lambda, \eta, \mu)$ and assume throughout the paper that $\eta$ is an integrable stopping time and that $\mu$ dominates $\mathcal{L}^{\lambda}(W_{\eta})$ in the convex order. 
Trivially, the \eqref{dSEP} has no solution if either of these conditions are violated. 
%
%
%
%
%
\subsection[Root and Rost solutions to \eqref{dSEP}]{Root and Rost solutions to (dSEP)}
%
We are interested in Root and Rost solutions to \eqref{dSEP}, that is solutions of the form 
\[
    \rho := \inf \left\{ t \geq \eta : (t,W_t) \in R \right\}
\]
where $R$ is a Root or Rost barrier as defined in \eqref{def:BarrierSet}. 

In \cite{BeCoHu17} the necessary theory and tools were given to identify specific solutions to \eqref{SEP} as stopping times that are optimisers over appropriate additional optimisation criteria. 
Such \eqref{SEP} with additional optimisation criteria are denoted by {\sf (OptSEP)}.
In \cite{BeCoHu16} this concept was extended to multiple marginals and denoted by {\sf (OptMSEP)}. 
Especially multi-marginal Root and Rost solutions were recovered as solutions to appropriate {\sf (OptMSEP)}.

We can further ask whether we can recover the delayed Root and 
Rost embeddings as optimisers over an appropriate optimisation criteria. 
The answer to this is positive, and is a relatively simple
application of the results of \cite{BeCoHu16}. 
We work in the class of \emph{randomised multi-stopping times} introduced in \cite{BeCoHu16}, $\RMST{}^1_2$, 
which we understand to be probability measures on the space $\R \times C_0(\R_+) \times \Xi^2$, 
where $\Xi^2 = \{(s_1, s_2) \in \R^2_+, s_1 \le s_2\}$, 
and where the measure projected on its first two coordinates is equal to the measure $\lambda \times \W$,
with $\W$ the Brownian path measure on $C_0(\R_+)$. 
The $\Xi^2$ variables represent the first and second stopping times respectively. The measure has
further constraints on its structure (see \cite{BeCoHu16}, Definition 3.9),
which ensure that the stopping time properties are respected.

We are interested in optimising a functional of the form
$\E[F(W_{\tau_2},\tau_2)]$ over the class of randomised stopping times
$(\tau_1,\tau_2)$ which satisfy the embedding constraint,
$W_{\tau_2} \sim \mu$, and for which $\tau_1 = \eta$ is the given (randomised) delay stopping
time embedding $\alpha_X = \mathcal{L}^{\lambda}(W_{\eta})$.

Note that $\eta$ can be understood in the sense of randomised multi-stopping times, 
see Lemma 3.11 of \cite{BeCoHu16}.

Moreover, we argue that for the specific form of $F(x,t) = h(t)$ for a
convex/concave function $h$, then the stopping time recovered as the optimiser is
exactly the delayed Root/Rost stopping time. We can do this using a
slight modification of the arguments in \cite{BeCoHu16}.

We proceed as follows:

Let $\RMST{}^1_2(\eta,\mu)$ denote the subset of $\RMST{}^1_2$ such
that the projection onto $\R \times C_0(\R_+) \times \Xi$ is exactly
$\eta$. Then:
\begin{enumerate}
\item $\RMST^1_2(\eta,\mu) \subseteq \RMST(\lambda,\alpha_X,\mu)$ 
as defined as in \cite[Definition 3.18]{BeCoHu16}
which is a compact set (Proposition 3.19);
\item $\RMST^1_2(\eta,\mu)$ is closed, as the projection map is
  continuous, hence it is also compact;
\item The proof of a suitably modified version of Theorem 5.2 now goes
  through essentially verbatim to give the following theorem.
\end{enumerate}

\begin{theorem}
  Assume that $\gamma : S^{\otimes 2} \to\R$ is Borel
  measurable. Assume that {\sf (OptMSEP)} is well posed and that $\xi \in
  \RMST^1_2(\eta,\mu)$ is an optimizer. Then there exists a
  $(\gamma,\xi)$-monotone Borel set $\Gamma$ such that $r_1(\xi)(\Gamma)= 1$.
\end{theorem}
Here, $\Gamma$ being $(\gamma,\xi)$-monotone is as defined in 
\cite[Definition 5.1]{BeCoHu16} , without the constraint on the projections of the
respective sets. The final conclusion of the Theorem can be applied as
in \cite[Theorem 2.7]{BeCoHu16}, \cite[Theorem 2.10]{BeCoHu16} (or more transparently, the corresponding
results in \cite{BeCoHu16}) to recover the Root/Rost forms of the
optimisers.

\medskip

While Root barrier solutions will always be adapted to the filtration generated by the underlying Brownian motion, 
it is a well known fact that Rost barrier solutions will require some form of external randomization 
in the case where the initial and the terminal distribution share some mass. 
The story is similar for delayed Root and Rost embeddings and will be summed up in the following two theorems. 
\begin{theorem} \label{thm:dRoot}
Consider a \eqref{dSEP} given by $(\lambda, \eta, \mu)$. 
Then there exists a Root barrier $R$ such that the stopping time $\rho \geq \eta$ defined via
\[
    \rho = \inf \left\{ t \geq \eta : (t,W_t) \in R \right\}
\]
embeds the measure $\mu$. 
\end{theorem}
The proof of this theorem follow analogously to the proof of \cite[Theorem 2.7]{BeCoHu16}, 
or similarly of \cite[Theorem 2.1]{BeCoHu17} heeding the modifications given above. 
Moreover, it is an easy application of the standard Loynes argument \cite{Lo70} 
to see that such a Root stopping time $\rho$ must be unique in the sense that 
for any other Root stopping time $\tilde \rho$ solving the \eqref{dSEP} we would have $\mathbb{P}^{\lambda}\left[ \rho = \tilde \rho \right] = 1$. 
More details on how to generalize the Loynes uniqueness argument can be found in \cite[Section 4]{Gr23}. 

\medskip

We proceed to Rost solutions. 
In order to give better insights on what will happen in the case of delayed Rost embeddings, 
we quickly review the case of an undelayed \eqref{SEP} determined by the initial measure $\lambda$ and terminal measure $\mu$. 
It is well known that all mass shared between the two measures 
$\lambda$ and $\mu$ must be stopped immediately in the following way. 
Let $\bar \rho$ be a Rost stopping time solving \eqref{SEP}, 
then
\(
     \mathbb{P}^{\lambda} \left[ \bar \rho = 0, W_0 \in A\right] = (\lambda \wedge \mu) (A)
\)
where the infimum of measures is defined as 
\[
    (\lambda \wedge \mu) (A) := \inf_{\scaleto{B \subseteq A,\, B \text{ measurable}}{4pt}} \big (\lambda(B) + \mu(A \setminus B) \big). 
\]
On the set $\{\bar\rho > 0\}$ the stopping time $\bar \rho$ will be  
the first hitting time of a Rost barrier $\bar R$. 

On a set where $(\lambda \wedge \mu)(A) < \lambda(A)$ we will have to randomise
our initial stopping rule in $A$ in order stop only sufficiently much mass.  If
$(\lambda \wedge \mu)(A) < \mu(A)$ then at least in some points in $A$ we need
to stop all mass and in fact we will need to extend the Rost barrier beyond 0 at
those points to stop more mass later on.  But in this case, all mass will stop
immediately in those points \emph{independent} of the choice of randomisation
rule, as a consequence of the recurrence of Brownian motion.  Thus any
randomisation rule will give the same stopping rule although an obvious
``maximal'' rule is to stop all paths starting in $A$ due to the randomisation.
If $(\lambda \wedge \mu)(A) = \mu(A) = \lambda(A)$, the randomisation rule that
stops all mass is the only viable stopping rule since there will be no further
barrier in $A$ beyond time 0.

In the delayed framework something similar needs to happen while it is apparent that 
things are complicated by the fact that not all mass will be released at time 0 at once, 
but rather over time, according to the measure $\mathcal{L}^{\lambda}(\eta)$. 
It is, however, clear that such external randomization in only feasible in the \emph{atoms} 
of the measure $\mathcal{L}^{\lambda}(\eta)$, of which there can only be countably many $\{t_0, t_1, \dots\}$. 
A \emph{delayed, randomized} Rost solution will therefore consist of 
both a Rost barrier $\bar R$ and some description of this randomization in the atoms of the delay stopping time. 
Hence, we choose to represent such a solution as follows.
\begin{definition} \label{def:dRost} 
A delay stopping time $\eta$ which distribution has atoms in $\{t_0, t_1, \dots\}$, 
a Rost barrier $\bar R$ and a family of probability densities $\{\varphi_0, \varphi_1, \dots\}$  
determine a \emph{delayed, randomized} Rost stopping time $\bar \rho^{\varphi}$ in the following way.
\[
    \bar \rho^{\varphi} := \inf \left\{ t \geq \eta : (t,W_t) \in \bar R \right\} \wedge Z
\]
where
\[
    Z = \min_{k \in \mathbb{N}} Z_k \,\,\text{ for }\,\,
    Z_k := 
        \begin{cases}
            t_k  &\text{ with probability }\, \mathbb{P}^{\lambda} 
                \left[Z_k = t_k \big{|} W_{\eta}, \eta = t_k \right] 
                    = \varphi_k\left( W_{\eta} \right)
        \\  \infty & \text{ otherwise. }
      \end{cases}
\]
\end{definition}
Let 
\begin{equation*} 
    \alpha := \mathcal{L}^{\lambda}((\eta, W_{\eta})) \in \mathcal P ([0, \infty) \times \mathbb{R})
\end{equation*}
denote the space-time distribution induced by the delay stopping time $\eta$. 
Then we abbreviate by 
\begin{equation*} 
    \alpha_{t} := \alpha\left(\{t\} \times \cdot \right)
\end{equation*}
its spatial marginal at time $t$. 

A crucial fact to note about Rost solutions is that at any time point $t \geq 0$ 
we are able to tell exactly how much mass was already embedded in a given measurable set $A \subseteq \mathbb{R}$ 
but more importantly we are furthermore able to decide 
that no more mass will be embedded in this set going forward in time. 
These properties allow us to define Rost solutions in an iterative fashion 
and for this purpose, given a Rost solution $\bar \rho$ we define the following quantity
\[
    \nu_t :=  \mu - \mathbb{P}^{\lambda} \left[ W_{\bar \rho} \in \cdot, \bar\rho < t \right],
\]
the `mass not embedded before time $t$'

Let $t_k$ denote an atom of the measure $\mathcal{L}^{\lambda}(\eta)$, 
then note that the quantity $\alpha_{t_k}$ is given 
while the quantity $\nu_{t_k}$ is fully determined by the Rost barrier \emph{before} time $t_k$, 
i.e. $\bar R \cap ([0,t_k) \times \mathbb{R})$ and the randomization densities $\{\varphi_l\}_{t_l < t_k}$.

We consider
\begin{equation} \label{def:psi}
    \psi_{k}(x) := \frac{\mathrm{d}\left(\alpha_{t_k} \wedge \nu_{t_k} \right) }{ \mathrm{d}\left( \alpha_{t_k}\right)} (x),
\end{equation}
the Radon-Nikodym density of $\alpha_{t_k} \wedge \nu_{t_k}$ with respect to $\alpha_{t_k}$ 
and state the following theorem.
\begin{theorem} \label{thm:dRost}
Consider \eqref{dSEP} given by $(\lambda, \eta, \mu)$. 
Let $\{ t_0, t_1, \dots \}$ 
denote the set of 
(at most countably many) atoms  of $\mathcal{L}^{\lambda}(\eta)$. 
Then there exists a Rost barrier $\bar R$ such that the stopping time $\rho \geq \eta$ defined via
\[
    \rho = \inf \left\{ t \geq \eta : (t,W_t) \in \bar R \right\} \wedge Z
\]
where
\[
    Z = \min_{k \in \mathbb{N}} Z_k \,\,\text{ for }\,\,
    Z_k := 
        \begin{cases}
            t_k  &\text{ with probability }\, \mathbb{P}^{\lambda} 
                \left[Z_k = t_k \big{|} W_{\eta}, \eta = t_k \right] 
                    = \psi_{k} \left( W_{\eta} \right)
        \\  \infty & \text{ otherwise }
      \end{cases}
\]
is a solution to \eqref{dSEP}.
\end{theorem}
As in the Root case, the existence of a Rost barrier will 
follow analogously to the proof of \cite[Theorem 2.10]{BeCoHu16}
or similarly of \cite[Theorem 2.4]{BeCoHu17} heeding the modifications given above. 

Outside of the atoms of the delay distribution, 
more precisely conditional on the event $\{ \eta \not \in \{ t_0, t_1, \dots \}\}$,  
two Rost stopping times $\bar \rho$ and $\tilde \rho$ solving the same \eqref{dSEP}
are given as hitting times of two Rost barriers, hence the same Loynes argument as mentioned in the Root case 
will provide a.s. equality.

The density of the random variables $Z_k$ does not necessarily need to be unique. 
There is, however, a sort of maximality about the chosen densities $\psi_0, \psi_1, \dots$ in Theorem \ref{thm:dRost} 
which will be clarified with the following lemma.
\begin{lemma} \label{lem:Rost-uniqueness}
Consider a \eqref{dSEP} given by $(\lambda, \eta, \mu)$. 
Let $\bar R$ be a Rost barrier and $\{\varphi_0, \varphi_1, \dots\}$ be a family of densities such that 
the delayed, randomized Rost stopping time $ \bar \rho^{\varphi}$ given as in Definition \ref{def:dRost} solves the 
\eqref{dSEP}. 
Then for any $k \in \mathbb{N}$ we can replace $\varphi_k$ by $\psi_k$ given by \eqref{def:psi} 
to obtain the stopping time $\bar \rho^{\psi}$ for which we will have 
$\mathbb{P}^{\lambda} \left[ \bar \rho^{\varphi} =  \bar \rho^{\psi} \right] = 1$. 
\end{lemma}
\begin{proof}
Let us consider an arbitrary atom $t_k$ of the delay distribution $\mathcal{L}^{\lambda}(\eta)$. 
First of all we clarify that in order for $\varphi_k$ to be feasible we must have $\varphi_k \leq \psi_k$, 
as for all measurable sets $A \subseteq \mathbb{R}$ we have
\[
    \mathbb{P}^{\lambda} \left[ W_{\bar \rho^{\varphi}} \in A, \bar \rho^{\varphi} = t_k, \eta = t_k\right] 
    \leq 
    \begin{cases}
         \mathbb{P}^{\lambda} \left[ W_{t_k} \in A, \eta = t_k\right] = \alpha_{t_k}(A),  
    \\   \mathbb{P}^{\lambda} \left[ W_{\bar \rho^{\varphi}} \in A, \bar \rho^{\varphi} = t_k \right] 
          \leq \mathbb{P}^{\lambda} \left[ W_{\bar \rho^{\varphi}} \in A, \bar \rho^{\varphi} \geq t_k \right] = \nu_{t_k}(A).
    \end{cases}
\]
Denote by $A_{k} := \bigcap_{\varepsilon > 0} \{x: (t_k + \varepsilon, x) \in \bar R \}$ all spatial points in the barrier $\bar R$ at time $t_k$. 
Then clearly we have
\[
      \mathbb{P}^{\lambda} \left[ \bar \rho^{\varphi} = t_k \big{|} \eta = t_k, W_{t_k} \in A_{t_k} \right] 
    = 1
    = \mathbb{P}^{\lambda} \left[ \bar \rho^{\psi} = t_k \big{|} \eta = t_k, W_{t_k} \in A_{t_k} \right]
\]
since both stopping times are hitting times of the barrier $\bar R$. 
Hence, the specific values of $\varphi$ and $\psi$ are irrelevant 
and we might replace $\varphi(x)$ with $\psi(x)$ for all $x \in A_k$. 

Due to the Rost barrier structure, 
on the set $A_{k}^c$, no mass will be embedded at a later time point,
\begin{equation} \label{eq:A-c-0}
    \mathbb{P}^{\lambda} \left[ W_{\bar \rho^{\varphi}} \in A_{t_k}^c, \bar \rho^{\varphi} > t_k \right] = 0 
    = \mathbb{P}^{\lambda} \left[ W_{\bar \rho^{\psi}} \in A_{t_k}^c, \bar \rho^{\psi} > t_k \right]. 
\end{equation}
Consider now the set $B_{k} := \{x:\varphi_k(x) < \psi_{k}(x)\} \cap A_k^c$. 
Let us assume that $\alpha_{t_k}(B_k) > 0$ as otherwise both $\varphi_k(x)$ and $\psi_k(x)$ must be 0 for all $x \in B_k$ anyway. 
Our claim is then proven if we are able to establish $\int_{B_k} \varphi(x)\mathrm{d}x = \int_{B_k} \psi(x)\mathrm{d}x$.

It is a consequence of \eqref{eq:A-c-0} that
\begin{align*}
    \mu(B_k) = \mathbb{P}^{\lambda} \left[ W_{\bar \rho^{\varphi}} \in B_k, \bar \rho^{\varphi} < t_k \right]
             + \mathbb{P}^{\lambda} \left[ W_{\bar \rho^{\varphi}} \in B_k, \bar \rho^{\varphi} = t_k \right],
\end{align*}
or equivalently that 
\begin{align*}
\nu_{t_k}(B_k) &= \mathbb{P}^{\lambda} \left[ W_{\bar \rho^{\varphi}} \in B_k, \bar \rho^{\varphi} = t_k \right]
\\        &= \mathbb{P}^{\lambda} \left[ W_{\bar \rho^{\varphi}} \in B_k, \bar \rho^{\varphi} = t_k, \eta = t_k \right].
\end{align*}
Here the second equality follows from the specific structure of Rost solutions, 
as all mass that is already released before time $t_k$ and stopped in $B_k$ 
will have done so \emph{before} time $t_k$ almost surely.  
Then 
\begin{align*}
\nu_{t_k}(B_k) = \mathbb{P}^{\lambda} \left[ W_{\bar \rho^{\varphi}} \in B_k, \bar \rho^{\varphi} = t_k, \eta = t_k\right]
    & = \mathbb{P}^{\lambda} \left[ W_{t_k} \in B_k, \bar \rho^{\varphi} = t_k, \eta = t_k\right]
\\  & = \alpha_{t_k}(B_k) \int_{B_k} \varphi(x)\mathrm{d}x
\\  & \leq   \alpha_{t_k}(B_k) \int_{B_k} \psi(x)\mathrm{d}x
    = (\alpha_{t_k} \wedge \nu_{t_k})(B_k)
\end{align*}
In particular, this implies that
\[
(\alpha_{t_k} \wedge \nu_{t_k})(B_k) = \nu_{t_k}(B_k), 
\]
hence we must have $\int_{B_k} \varphi(x)\mathrm{d}x = \int_{B_k} \psi(x)\mathrm{d}x$.
\end{proof}
Here we observe that if the set of atoms ${t_0, t_1, \dots}$ is either finite or infinite but can be ordered 
we can assume $t_0 \leq t_1 \leq \cdots$ and construct the densities $\psi_k$ consecutively. 
When this is not the case, meaning ${t_0, t_1, \dots}$ is infinite and cannot be ordered, 
while our result will still hold, 
we cannot expect to derive an explicit formula for the functions $\psi_k$ in this scenario.%

An explicit Root solution to \eqref{MMSEP} is clear given Theorem \ref{thm:dRoot} while an explicit Rost solution can be recovered from Theorem \ref{thm:dRost} in the following way.  
\begin{corollary}
Consider a starting measure $\lambda$ and a sequence of measures $\mu_1, \dots, \mu_n$ such that $\lambda \leq_c \mu_1 \leq_c \cdots \leq_c \mu_n$. 
Then there exists a sequence or Rost barriers $R_1, \dots, R_n$ such that the stopping times $\rho_1 \leq \rho_2 \leq \cdots \leq \rho_n$ defined in the following way are a solution to \eqref{MMSEP}. 
\[
    \rho_0 := 0 \, \text{ and } \,
    \rho_k := \inf \left\{ t \geq \rho_{k-1} : (t,W_t) \in R_k \right\} \wedge Z_k \,\text{ for }\, k \geq 1
\]
where
\[
    Z_k := 
        \begin{cases}
            0   &\text{ with probability }\, \mathbb{P}^{\lambda} 
                \left[Z_k = 0 \big{|} W_0, \rho_{k-1} = 0 \right] 
                    = \frac{\mathrm{d}\left( \mu^0_{k-1} \wedge \mu_k \right) }{\mathrm{d} \mu^0_{k-1}} \left(W_0 \right)
            \\ \infty  &\text{ otherwise }
        \end{cases}
\]
for $\mu^0_k := \mathbb{P}^{\lambda}\left[\rho_{k-1} = 0, W_0 \in \cdot \right]$.
\end{corollary}
\begin{proof}
Note that for $k=1$ this exact structure of the Rost solution $\rho_1$ has already been established in \cite{Ro73}. 
Observing that $\rho_1 \leq \dots \leq \rho_n$ are hitting times of inverse barriers, 
we can establish that $\mathcal{L}^{\lambda}(\rho_1)$ and consequently 
any $\mathcal{L}^{\lambda}(\rho_k)$, $k \geq 2$, do not possess atoms beyond time $0$.

Inductively recovering the multi-marginal Rost solutions from Theorem \ref{thm:dRost} 
becomes feasible by considering the single atom $t_0 = 0$. 
Thus, for $k \geq 2$, set $\eta = \rho_{k-1}$ and $\mu = \mu_k$, with $\nu_0 = \mu_k$ and $\tilde \alpha_0 = \mu^0_k$. 
Subsequently, invoking Theorem \ref{thm:dRost}, we establish the existence of a Rost barrier $R := R_k$, 
ensuring that $\rho_k \geq \rho_{k-1}$ is of the desired structure with $Z_k := Z_0$.
\end{proof}
%
%
%
%
%
\subsection[Optimal Stopping Representation of \eqref{dSEP} and \eqref{MMSEP}]{Optimal Stopping Representation of (dSE) and (MMSEP)}
%
We will derive the multi marginal Root optimal stopping problem \eqref{eq:MMRoot} - \eqref{eq:MMRootOSP} 
resp. the multi marginal Rost  optimal stopping problem \eqref{eq:MMRost} - \eqref{eq:MMRostOSP} 
as a special case of a \emph{delayed} Root resp. Rost representation.

For this purpose we consider a starting distribution $\lambda$ for a Brownian motion and an integrable delay stopping time $\eta$. 
Recall the space time distribution induced by the delay 
\begin{equation*} 
    \alpha := \mathcal{L}^{\lambda}((\eta, W_{\eta})) \in \mathcal P ([0, \infty) \times \mathbb{R}).
\end{equation*}
We denote the $\mathbb{P}^{\lambda}$-potential of $W_{\eta \wedge T}$ by
\[
    V^{\alpha}_T (x) 
    = -\E^{\lambda}\left[\left|W_{\eta \wedge T} - x \right| \right] 
\]
and the spatial marginal of $\alpha$ by
\[
    \alpha_X (A) := \alpha([0, \infty) \times A) 
        = \mathcal{L}^{\lambda}(W_{\eta})(A)
         \quad \text{ for } A \subseteq \R \text{ measurable.}
\]
Let now $R \subseteq [0, \infty) \times \mathbb{R}$ denote a Root (resp. Rost) barrier,
then we can define the (delayed) Root (resp. Rost) stopping time 
\begin{equation*}
    \rho_{\eta} := \inf \left\{t \geq \eta : (t, W_t) \in R \right\}.
\end{equation*}
These stopping times induce the following measures
\[ \beta^{\alpha} := \mathcal{L}^{\lambda} \left(W_{\rho^{\eta}}\right)
        \text{ as well as }
    \beta^{\alpha}_T := \mathcal{L}^{\lambda} \left(W_{\rho^{\eta} \wedge T}\right) \text{ for } T \geq 0
\]
and the (stopped) potential
\[
  U_{\beta^{\alpha}_T}(x) = - \E^{\lambda}\left[\left|W_{\rho^{\eta} \wedge T} - x\right| \right].
\]
The extension of the Root optimal stopping problem \eqref{eq:Root}-\eqref{eq:RootOSP} will then be given in the following theorem. 
\begin{theorem} \label{thm:delayed-Root}
Let $R$ denote a \emph{Root} barrier inducing the potentials $U_{\beta^{\alpha}}$ and $U_{\beta^{\alpha}_T}$. 
Then for every $(T, x) \in [0, \infty) \times \mathbb{R}$ we have the representation
\begin{align}
    \label{eq:delayed-Root}
     U_{\beta^{\alpha}_T}(x) &= \E^x\left[ V^{\alpha}_{T - \tau^*} (W_{\tau^*}) 
        + (U_{\beta^{\alpha}} - U_{\alpha_X})(W_{\tau^*})\ind_{\tau^*  < T}\right]   
\\ \label{eq:delayed-RootOSP}
                            &= \sup_{\tau \leq T} \E_y\left[ V^{\alpha}_{T - \tau} (W_{\tau}) 
        + (U_{\beta^{\alpha}} - U_{\alpha_X})(W_{\tau})\ind_{\tau  < T}\right].
\end{align}
where the optimizer is given by $\tau^* := \inf \{t \geq 0 : (T-t, W_t) \in R \} \wedge T$.

Moreover, given the interpolating potentials $U_{\beta^{\alpha}_T}(x)$ for all $(T,x) \in [0, \infty) \times \mathbb{R}$ 
we can recover the Root barrier $R$ embedding the measure $\beta^{\alpha}$ in the following way
\[
    R = \left\{(T,x) : \left(U_{\beta^{\alpha}_T} - V^{\alpha}_T \right)(x) 
                       = \left(U_{\beta^{\alpha}} - U_{\alpha_X}\right)(x) \right\}.
\]
\end{theorem}
Analogously we can give an optimal stopping representation for delayed Rost embeddings extending 
\eqref{eq:Rost}-\eqref{eq:RostOSP} in the following way.
\begin{theorem} \label{thm:delayed-Rost}
If $\bar R$ denotes a \emph{Rost} barrier inducing the potentials $U_{\beta^{\alpha}}$ and $U_{\beta^{\alpha}_T}$.
Then for every $(T, x) \in [0, \infty) \times \mathbb{R}$ we have the representation 
\begin{align}
    \label{eq:delayed-Rost}
    U_{\beta^{\alpha}}(x) - U_{\beta^{\alpha}_T}(x)  &= \E^x\left[ U_{\beta^{\alpha}} (W_{\sigma_*})  
            - V^{\alpha}_{T - \sigma_*} (W_{\sigma_*}) \right]   
\\ \label{eq:delayed-RostOSP}
        &= \sup_{\sigma \leq T} \E^x\left[ U_{\beta^{\alpha}} (W_{\sigma}) 
            - V^{\alpha}_{T - \sigma} (W_{\sigma}) \right],
\end{align}
where the optimizer is given by $\sigma_* := \inf \{t \geq 0 : (T-t, W_t) \in \bar D \} \wedge T$.

Moreover, given the interpolating potentials $U_{\beta^{\alpha}_T}(x)$ for all $(T,x) \in [0, \infty) \times \mathbb{R}$ 
we can recover the Rost barrier $\bar R$ embedding the measure $\beta^{\alpha}$ in the following way
\[
    \bar R = \left\{(T,x) : U_{\beta^{\alpha}_T}(x)  =  V^{\alpha}_{T}(x)\right\}.
\]
\end{theorem}
First we want to observe that for the choice of $\eta = 0$ the single marginal identities can easily be recovered. 
Indeed, 
\begin{itemize}
\item[\itembullet] $\beta^{\alpha} = \mu$ and $\beta^{\alpha}_T = \mu_T$ as in Section \ref{sec:intro}
\item[\itembullet] $\alpha = \delta_0 \times \lambda$ thus $\alpha_X = \lambda$
\item[\itembullet] $V^{\alpha}_T (y) = -\E^{\lambda} \left[\left|W_{0} - y \right| \right] = U_{\lambda}(y)$ for any $T \geq 0$
\end{itemize}
hence
\begin{align*}
    \E_y\left[ V^{\alpha}_{T - \tau} (W_{\tau}) + (U_{\mu} - U_{\alpha_X})(W_{\tau})\ind_{\tau  < T}\right]   
&=    \E_y\left[ U_{\lambda}(W_{\tau}) + (U_{\mu} - U_{\lambda})(W_{\tau})\ind_{\tau  < T}\right]   
\\&=    \E_y\left[U_{\mu}(W_{\tau})\ind_{\tau  < T} +  U_{\lambda}(W_{\tau}) \ind_{\tau  = T}\right].   
\end{align*}
in the Root case, and
\[
    \E^x\left[ U_{\beta^{\alpha}} (W_{\sigma}) - V^{\alpha}_{T - \sigma} (W_{\sigma}) \right] 
    = \E^x\left[ U_{\mu} (W_{\sigma}) - U_{\lambda} (W_{\sigma}) \right]
\]
in the Rost case respectively.

Furthermore, the multi-marginal extensions now follow as an corollary 
as they can be deduced inductively from the delayed optimal stopping problem 
\eqref{eq:delayed-Root}-\eqref{eq:delayed-RootOSP} 
(resp. \eqref{eq:delayed-Rost}-\eqref{eq:delayed-RostOSP}) in the following way. 
\begin{corollary}
Let $R_1, R_2, \dots, R_n$ denote a sequence of Root (resp. Rost) barriers inducing the potentials $U_{\mu_1}, \dots, U_{\mu_n}$ as well as $U_{\mu_1, T}, \dots, U_{\mu_n, T}$. 
Then the optimal stopping representations \eqref{eq:MMRoot}-\eqref{eq:MMRootOSP} (resp. \eqref{eq:MMRost}-\eqref{eq:MMRostOSP}) hold. 
\end{corollary}
\begin{proof}
For the first marginal induced by $R_1$ the optimal stopping representation is simply the single marginal representation 
\eqref{eq:Root}-\eqref{eq:RootOSP} (resp. \eqref{eq:Rost}-\eqref{eq:RostOSP})
already known resp. recovered above. 
Let now $k \in \{2, \dots, n\}$ and consider $\eta = \rho_{k}$ and $R = R_{k+1}$. 
Then 
\begin{itemize}
\item[\itembullet] $\alpha = \mathcal{L}^{\lambda}(\rho_k, W_{\rho_k})$, thus $\alpha_X = \mu_k$ and $U_{\alpha_X}(y) = U_{\mu_{k}} (y)$,
\item[\itembullet] $\beta^{\alpha} = \mu_{k+1}$, thus $U_{\beta^{\alpha}}(y) = U_{\mu_{k+1}} (y)$ and $U_{\beta^{\alpha}_T}(y) = U_{\mu_{k+1, T}} (y)$,
\item[\itembullet] $V^{\alpha}_T (y) = -\E^{\lambda} \left[\left|W_{\rho_k \wedge T} - y \right| \right] = U_{\mu_{k,T}}(y)$.
\end{itemize}
Hence we recover the multi marginal Root (resp. Rost) optimal stopping problems 
\eqref{eq:MMRoot}-\eqref{eq:MMRootOSP} (resp. \eqref{eq:MMRost}-\eqref{eq:MMRostOSP}).
\end{proof}%
%
%
%
%
%
%
%
%
%
%
%
%
%
%
%
%
%
%
%
%
%
%
%
%
%
%
%
%
%
%
%
%
%
%
%
%
%
%
%
%
%
%
%
%
%
%
%
%
%

\section[A Discrete \eqref{dSEP} and its Optimal Stopping Representation]{A Discrete (dSEP) and its Optimal Stopping Representation} \label{sec:discrete}                                                             
%
We take on the idea of \cite{BaCoGrHu21} to cast our problem in a discrete setting 
in order to use arguments for SSRWs to derive a discrete version of the desired result 
and then recover the continuous result through a Donsker type limiting argument. 
In this section we define a discrete version of \eqref{dSEP}. 
We furthermore define discrete Root and Rost solutions and give an explicit construction. 

Let $X$ denote a SSRWs on some probability space $(\Omega, \mathbb{P})$, started at some possibly random initial position.
Given $x \in \mathbb{Z}$ we write $\mathbb{P}^x$ for the conditional distribution when $X_0=x$. 
Hence we can consider a probability measure $\lambda$ on $\mathbb{Z}$ a ``starting'' distribution for $X$ 
by setting $\mathbb{P}^\lambda := \sum_{x\in \mathbb{Z}} \mathbb{P}^x\lambda(\{x\})$. 

%
\subsection[A Discrete \eqref{dSEP}]{A Discrete (dSEP)} \label{sec:discrete-existence}
Consider a SSRW $(X_t)_{t \in \mathbb{N}}$ started in $X_0 \sim \lambda$. 
Given a discrete measure $\mu$ the \emph{discrete Skorokhod embedding problem} is 
to find a (discrete) stopping time $\rho$ such that 
\begin{equation}\label{discrete-SEP} 
    X_{\rho} \sim \mu \text{ and } X_{t \wedge \rho} \text{ is uniformly integrable.} 
\end{equation}

Given a delay stopping time $\eta$ the \emph{delayed} discrete Skorokhod embedding problem (or \emph{delayed} \eqref{dSEP}) is to find  
a (discrete) stopping time $\rho \geq \eta$ such that \eqref{discrete-SEP} is satisfied.

Just as in the continuous case we are interested in the existence of barrier type solutions to the discrete \eqref{dSEP}. 
These barriers can be defined analogously to the continuous setting as subsets $R \subseteq \mathbb{N}\times \mathbb{Z}$ satisfying the following respective property.
\begin{itemize}
\item[\itembullet] $R$ will denote a (discrete) \emph{Root} barrier if for $(t,x)\in R$ and $s>t$ we also have $(s,x)\in R$.
\item[\itembullet] $R$ will denote a (discrete) \emph{Rost} barrier if for $(t,x)\in R$ and $s<t$ we also have $(s,x)\in R$.
\end{itemize}

While in the continuous setting it is known that no external randomization is needed in order to give a Root solution to the \eqref{SEP} 
and external randomization is needed only in $t=0$ for the Rost solution (and only if the initial measure $\lambda$ and the terminal measure $\mu$ share mass) this is no longer the case in the discrete setting. 
We will see that in order to give a discrete Root (resp. Rost) solution external randomization will be necessary in both cases also beyond time $t = 0$, however only at the boundary of the respective barrier.

\begin{example} \label{exp:easy-Rost-example}
Consider the simple example
\begin{align*}
    \lambda &= \delta_0
\\      \mu &= \frac{1}{3} \delta_{-2} + \frac{1}{3} \delta_0 + \frac{1}{3} \delta_{2}.
\end{align*}
First note that all paths of a SSRW $X$ that reach $\pm 2$ at any point in time will be absorbed there. 
As all paths originate from $(0,0)$, the probability of being in site $x = 0$ at time $t = 2$ is given by $\frac{1}{2} > \frac{1}{3} = \mu(\{0\})$. 
Consequently, we cannot stop all paths in $(2,0)$ as this would embed too much mass in $\{0\}$. 
However, if we allow all paths to proceed from $(2,0)$, then the probability of being in site $x=0$ at time $t=4$ without having hit $\pm 2$ before is given by $\frac{1}{4} < \frac{1}{3} = \mu(\{0\})$.
In other words, this would result in an insufficient mass being embedded in ${0}$.
Hence, it becomes necessary to stop \emph{some} paths in $(2,0)$, but not all of them. 
Precisely, we will need to stop $\frac{1}{3}$ of paths in $(2,0)$ while allowing the remainder to reach either of $\{-2,0,2\}$ 
at time $t=4$ in order to correctly embed the distribution $\mu$. 
Refer to Figure \ref{fig:delayed-Root-example} for an illustration on possible realizations of such paths. 
\begin{figure}[H]
\captionsetup{margin=2cm}
\centering
\begin{subfigure}{.30\linewidth}
\centering
\includegraphics[width = 1\linewidth]{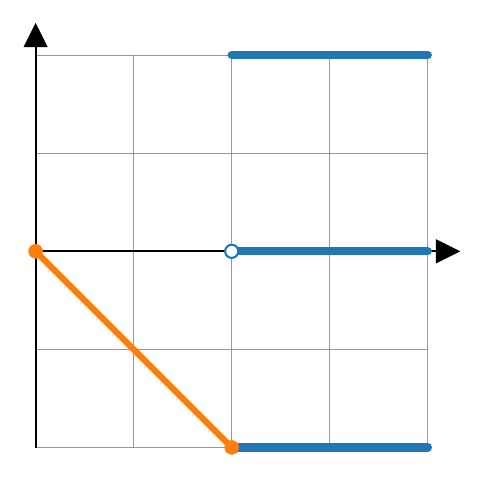}
\end{subfigure}
\begin{subfigure}{.30\linewidth}
\centering
\includegraphics[width = 1\linewidth]{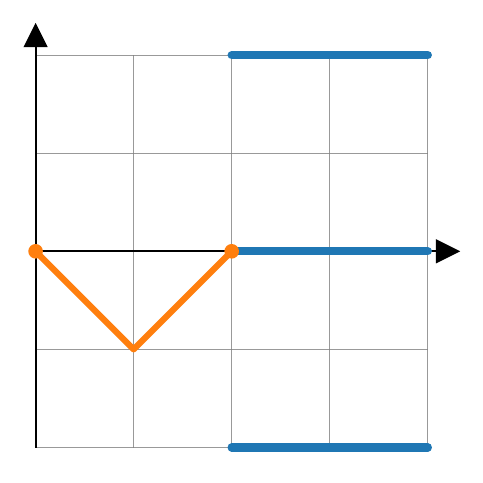}
\end{subfigure}
\begin{subfigure}{.30\linewidth}
\centering
\includegraphics[width = 1\linewidth]{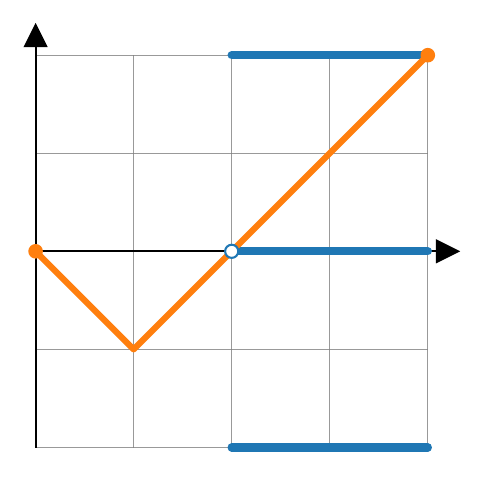}
\end{subfigure}
\caption{Possible realizations of paths hitting a Root barrier that solves the discrete \eqref{SEP} given by $\lambda = \delta_0$ and $\mu = \frac{1}{3} \delta_{-2} + \frac{1}{3} \delta_0 + \frac{1}{3} \delta_{2}$.}
\label{fig:delayed-Root-example}
\end{figure}
\end{example}
We will formalize this concept of randomized stopping in the following way. 

If $\rho$ denotes a stopping time for the process $X$ then we can for each point $(t,x) \in \mathbb{N} \times \mathbb{Z}$ consider the probability of $\rho$ stopping in site $x$ at time $t$ after $\eta$ by
\begin{equation}
    r_t(x) := \mathbb{P}^{\lambda}\left[ \rho = t | \left(\rho \wedge t, X_{\rho \wedge t}\right) = (t,x), \eta \leq t \right].
\end{equation}
If now $R$ denotes a barrier and $\rho$ its hitting time then in the classic (deterministic) case we should have
$r_t(x) \in \{0,1\}$ and $r_t(x) = 1$ if and only if $(t,x) \in R$, i.e. we stop in the barrier immediately after arrival.
However, if we allow for $r_t(x) \in (0,1)$ but agree on the convention that $(t,x) \in R$ when $r_t(x)>0$, 
then this can be understood as stopping inside of the barrier only with probability $r_t(x)$.
We can now give an equivalent definition of discrete Root (resp. Rost) barriers on the basis of these stopping probabilities.
A field of stopping probabilities denoted as $(r_t(x))_{(t,x) \in \mathbb{N} \times \mathbb{Z}}$ represents a 
\begin{align}
    &\itembullet\,\, \text{Root barrier if for all } (t,x) \in \mathbb{N} \times \mathbb{Z} \text{ such that } r_t(x)>0
        \text{ we have } r_{t+s}(x) = 1 \,\forall s \geq 1,        \label{eq:Root-field-def}
\\  &\itembullet\,\, \text{Rost barrier if for all } (t,x) \in \mathbb{N} \times \mathbb{Z} \text{ such that } r_t(x)>0
        \text{ we have } r_{t-s}(x) = 1 \,\forall s = 1, \dots, t. \label{eq:Rost-field-def}
\end{align}
This definition guarantees the respective barrier structure 
and furthermore assures that randomization only happens at the boundary points of the respective barriers. 
We will also refer to $(r_t(x))_{(t,x) \in \mathbb{N} \times \mathbb{Z}}$ as Root and Rost field of stopping probabilities, respectively.

Let $u = (u_0, u_1, \dots)$ denote a sequence of iid uniform random variable, i.e. $u_i \sim U[0,1]$. 
Then our Root (resp. Rost) stopping times subject to external randomization at the endpoints can be represented as follows
\begin{equation}
    \rho = \rho_u = \inf \left\{ t \geq \eta : r_t(X_t) > u_t \right\}.
\end{equation}
We could think of at every time point $t$ flipping a coin with success probability $r_t(X_t)$.
Was the coin flip successful the process stops, if not it moves onto the next point.
Note that if $r_t(x) \in \{0,1\}$ for all $(t,x) \in \mathbb{N} \times \mathbb{Z}$, then this definition will coincide with the usual non-randomized definition
\begin{equation}
    \rho  = \inf \left\{ t \geq \eta : (t, X_t) \in R \right\}.
\end{equation}

Our stopping time $\rho$ is now a \emph{randomized} stopping time taking the two arguments $\omega$ and $u$, thus $\rho: \Omega \times [0,1]^{\mathbb{N}} \rightarrow \mathbb{N}$.
While this is a slightly unconventional definition of an external randomization used for its more intuitive explanation in this context, 
we can also cast this randomized stopping time into the more conventional framework of using only a single variable of external randomization in the following way.
For each $\omega \in \Omega$ we consider the sequence $(\zeta_{\omega}(t))_{t \in \mathbb{N}}$ of the  probability of the process $X_{\rho \wedge t}(\omega)$ being `alive' at time $t$ given via
\begin{align*}
    \zeta_{\omega}(0) &:= (1-r_0(X_0(\omega))) \vee \ind_{\eta(\omega)=0} \quad \text{ and }
\\  \zeta_{\omega}(t) &:= \zeta_{\omega}(t-1) \cdot (1-r_t(X_t(\omega))) \vee \ind_{t < \eta(\omega)}.
\end{align*}
Note that $(\zeta_{\omega}(t))_{t \in \mathbb{N}}$ is a decreasing sequence and $\xi_{\omega}(\{0, \dots,t\}):= 1 - \zeta_{\omega}(t)$ will define a measure on $\mathbb{N}$.
Given a uniformly distributed random variable $u \sim U[0,1]$ independent of $X$ the (randomized) stopping time
\[
    \tilde \rho (\omega, u) := \inf \{ t \geq 0 : \xi_{\omega}(\{0, \dots,t\}) \geq u\}
\] 
now gives a more traditional representation of our randomized Root (resp. Rost) stopping time.

\ \\ 
In the following two subsections we will give an explicit construction of the stopping probabilities $(r_t(x))_{(t,x) \in \mathbb{N} \times \mathbb{Z}}$ solving the discrete \eqref{dSEP} both in the Root and Rost case.
\medskip
To guarantee uniform integrability of all solutions constructed we will henceforth always assume $\mu$ to denote a finite sum of atoms. 
Let $x_*$ and $x^*$ denote the smallest and largest site respectively at which we have an atom. 
Then in both the Root and the Rost case our solution should be bounded from above by the (delayed) hitting time
\[
    H_{x_*, x^*} := \inf \{t \geq \eta : X_t \in \{x_*, x^* \}\}.
\]
Since $\eta$ is a stopping time with finite expectation then so is $H_{x_*, x^*}$, 
which in turn implies uniform integrability of $(X_{t \wedge \rho})_{t \geq 0}$.
%
%
%
%
%
%
%
%
%
%
\subsubsection{Existence of Discrete Delayed Root Solutions}
%
An explicit Root solution to the discrete \eqref{dSEP} can be established with the help of expected local times. 
For a SSRW random walk $X$ the local time at site $x$ up to time $t$ is defined via
\[
    L_t^x(X) = \#\{ s : 0 \leq s < t, X_s = x \} = \sum_{s = 0}^{t-1} \ind_{X_s = x}.
\]
We note the following discrete version of Tanaka's Formula for local times.
\begin{theorem} \label{thm:discrete-Tanaka}
Consider a random walk $X$ defined by $X_t = X_0 + \sum_{s=1}^{t}\xi_{s}$ where $\xi_s$ are iid Rademacher random variables.
Then for $x \in \mathbb{Z}$ we can represent the local time $L_t^x(X)$ of $X$ until time $t$ as
\begin{equation} \label{eq:discrete-Tanaka}
    L_t^x(X) =  |X_t - x| - |X_0 - x| - \sum_{s = 0}^{t-1} \sgn(X_s - x) \xi_{s+1}.
\end{equation}
\end{theorem}
\noindent 
We refer to \cite{CsRe85} for a proof. 

Let 
\begin{equation*} 
    L_{s,t}^x(X)= \# \{u: s \leq u < t  \text{ and } X_u = x\} = \sum_{u = s}^{t-1} \ind_{X_u = x} = L_t^x(X) - L_s^x(X)
\end{equation*}
denote the local time of the process $X$ at site $x$ \emph{between} time $s$ and $t$ (for $s \leq t$).

For our delay stopping time $\eta$ write $\alpha_X = \mathcal{L}^{\lambda}(X_{\eta})$. 
Consider now $X_0 \sim \lambda$ and let $\rho \geq \eta$ be a stopping time for the random walk $X$ such that $X_{\rho} \sim \mu$. 
Then due to \eqref{eq:discrete-Tanaka} we have the following representation of the \emph{expected} local time between $\eta$ and $\rho$
\begin{equation} \label{eq:expected-lt}
    \E^{\lambda}\left[L_{\eta,\rho}^x(X)\right] = U_{\alpha_X}(x) - U_{\mu}(x) =: L(x).
\end{equation}
Conversely, for any solution $\rho \geq \eta$ to \eqref{dSEP} Equation \eqref{eq:expected-lt} must be satisfied. 

\medskip
Let $(r_t(x))_{(t,x) \in \mathbb{N} \times \mathbb{Z}}$ denote a field of stopping probabilities and $\rho$ its corresponding stopping time delayed by the stopping time $\eta$. 

Then we define 
\[
    \tilde \alpha_t(x):= \mathbb{P}^{\lambda} \left[ X_{t} = x, \eta \leq t, \rho \geq t \right],
\]
the free mass available in site $x$ at time $t$. 

Define $\alpha(t,x) := \mathbb{P}^{\lambda}\left[X_{\eta} = x, \eta = t\right]$,
then we can compute $\tilde \alpha_t$ inductively using only information up to time $t-1$ in the following way
\begin{equation*}
    \tilde \alpha_t(x) = \frac{1}{2}\left((1-r_{t-1}(x+1)) \tilde \alpha_{t-1}(x+1) + (1-r_{t-1}(x-1)) \tilde \alpha_{t-1}(x-1)\right) + \alpha(t,x).
\end{equation*}
Precisely, paths can only end up in $(t,x)$ from the neighboring points $(t-1, x+1)$ and $(t-1, x+1)$. 
They will with probability $1-r_{t-1}(x+1)$ resp. $1-r_{t-1}(x-1)$ leave the respective site and then with probability $\frac{1}{2}$ end up in $(t,x)$. 
Moreover, we only consider mass after the stopping time $\eta$, 
this is represented by $\alpha(t,x)$, the probability of particles \emph{appearing} in site $x$ at time $t$.
Note that at time $0$ we thus only have
\(
    \tilde \alpha_0(x) = \alpha(0,x).
\)

With the help of the quantities $\tilde \alpha_t(x)$ the expected number of visits in site $x$ up until time $t$ of the process $X$ can then be written as
\begin{equation}
    l_t^x := \sum_{s=0}^{t-1} (1-r_s(x))\tilde \alpha_s(x).
\end{equation}
In particular we will have $l_0^x = 0$ for all $x \in \mathbb{Z}$. 

We recall $L(x) := U_{\alpha_X}(x) - U_{\mu}(x)$ and define
\[
    \tilde L_t(x) := L(x) - l_t^x,
\]
the `missing local time' after time $t$. 
Note that $\tilde L_t(x)$ (resp. $l_t^x$) can be computed only using information up to time $t-1$. 

We are now equipped to give a precise construction of a Root solution to the discrete \eqref{dSEP}.
\begin{theorem} \label{thm:ddSEP-Root}
Consider a field of stopping probabilities $(r_t(x))_{(t,x) \in \mathbb{N} \times \mathbb{Z}}$ inductively defined via the formula
\[
  r_t(x):= 1 - \frac{ L(x) - l_t(x) }{ \tilde \alpha_t(x) } \wedge 1 
         = 1 - \frac{\left(\tilde L_t \wedge \tilde \alpha_t \right) \left( x \right) }{ \tilde \alpha_t(x) }
    \text{ when }  \tilde \alpha_t(x) > 0
\]
and chosen to comply with \eqref{eq:Root-field-def} otherwise.  

Then this field represents a Root field of stopping probabilities 
and the associated stopping time
\begin{equation*}
    \rho = \inf \left\{ t \geq \eta : r_t(X_t) > u_t \right\}.
\end{equation*}
will be a Root solution to the discrete \eqref{dSEP}. %
\end{theorem}
\begin{proof}
Note the following 
\begin{itemize}
\item[\itembullet] As both $\tilde L_t$ and $\tilde \alpha_t$ only use information that is either given or acquired up to time $t-1$, 
it is clear that $r_t$ can be computed inductively to give a stopping time. 
\item[\itembullet] When $\tilde \alpha_t(x)=0$ the point $(t,x)$ cannot be reached by the stopped random walk after time $\eta$. 
The specific value of $r_t(x)$ is hence irrelevant for all the respective quantities and can be chosen to comply with the definition of a Root field of stopping probabilities. 
\item[\itembullet] $r_t(x) = 1 \Leftrightarrow l_t^x = L(x)$. 
    In other words, we must have already accrued enough local time at site $x$. 
Moreover we will then have $l_{t+s}^x = l_{t}^x$ and $r_{t+s}(x) = 1$ for all $s \geq 0$, 
hence the Root structure of our barrier is guaranteed. 
\item[\itembullet] $r_t(x) = 0 \Leftrightarrow \tilde L_t(x) = L(x) - l_t^x \geq \tilde \alpha_t(x)$.
    So no mass will be stopped in $(t,x)$ since we still need to accrue more local time and adding all available mass to $l_t^x$ will make us remain in our local time budget, i.e. $l_{t+1}^x = l_t^x + \tilde \alpha_t(x) \leq L(x)$.
\item[\itembullet] $r_t(x) \in (0,1) \Leftrightarrow 0 < \frac{L(x) - l_t^x}{\tilde \alpha_t(x)} < 1$. 
    As letting all mass run would add too much to our expected local time only a fraction should be added and we have 
$l_{t+1}^x = l_{t}^x + \frac{L(x) - l_t^x}{\tilde \alpha_t(x)}\tilde \alpha_t(x) = L(x)$. This will imply $r_{t+s} = 1$  for all $s \geq 1$ again showing that we will recover the desired Root barrier structure. 
\end{itemize}
Thus for each time point $t$ this construction will reveal a new layer of stopping probabilities $(r_t(x))_{x \in \mathbb{Z}}$ which will be of Root structure. 
Let $\rho_{\infty}$ denote the stopping time corresponding to the field of stopping probabilities $(r_t(x))_{(t,x) \in \mathbb{N} \times \mathbb{Z}}$ constructed as above.
It remains to show that $\mathcal{L}^{\lambda}(X_{\rho_{\infty}}) = \mu$.
As $l_t^x$ is increasing and bounded from above by $L(x)$ the limit $l_{\infty}^x := \lim_{t \rightarrow \infty} l_t^x$ exists and equals the expected local time of $X_{\rho_{\infty}}$ after the stopping time $\eta$, 
\(
 \E^{\lambda}\left[L_{\eta,\rho_{\infty} }^x(X)\right] = l_{\infty}^t.
\) 
Let $U_{\mu_{\infty}}$ denote the potential of $\mathcal{L}^{\lambda}\left(X_{\rho_{\infty}}\right)$, then due to Tanaka's formula we have
\[
U_{\mu_{\infty}}(x) = U_{\alpha_x} - l_{\infty}^t,
\]
hence our stopping time will embed the right measure if and only if 
\(
 l_{\infty}^t = L(x).
\)
Consider the set 
\[
    I := \{x \in \mathbb{Z}: l_{\infty}^x < L(x)\} = \{x \in \mathbb{Z}: U_{\mu_{\infty}}(x) > U_{\mu}(x)\}.
\] 
To conclude the proof of the theorem we must show that $I$ is empty. 
A first step is to show that for every $x\in I$ we have
\begin{equation} \label{eq:P0}
    \mathbb{P}^{\lambda}\left[X_{\rho_{\infty}} = x \right] = 0.
\end{equation} 
Note that this is equivalent to $\tilde \alpha_t(x) \wedge r_t(x) = 0$ for all $t \in \mathbb{N}$. 
However, we cannot have $\tilde \alpha(x) = 0$ for all $t \in \mathbb{N}$ as this would imply that site $x$ 
is never reached by the process $X$ with positive probability after time $\eta$, 
furthermore implying $L(x) = 0$, which is a contradiction.  
Thus there must exist at least one $T \in \mathbb{N}$ such that $\tilde \alpha_T(x) > 0$. 
Assume now that also $r_T(x) > 0$. 
Then $l_{\infty}^x \geq l_{T+1}^x = l_{T}^x + \left(\frac{L(x) - l_T^x}{\tilde \alpha_T(x)}\right)\tilde \alpha_T(x) = L(x)$ 
which is, again, a contradiction, allowing us to conclude \eqref{eq:P0}.

Now we proceed to show that $I$ must be empty. 
Consider a maximal interval $J \subseteq I$, 
i.e. either $J = \{\dots, x_+ -1, x_+ \}$ and $x_+ +1 \not\in I$, 
$J = \{x_-,x_- +1,  \dots\}$ and $x_- -1 \not\in I$ 
or $J = \{x_-,x_- +1,  \dots, x_+ -1, x_+\}$ and both $x_- -1, x_+ +1 \not\in I$.
Then we must have $U_{\mu_{\infty}}(x_- -1) = U_{\mu}(x_- -1)$ resp. $U_{\mu_{\infty}}(x_+ + 1) = U_{\mu}(x_+ + 1)$.
Note that $U_{\mu_{\infty}}$ is bounded from above by the concave function $U_{\alpha_X}$ and bounded from below by the concave function $U_{\mu}$. 
Hence we cannot have that $U_{\mu_{\infty}}|_J$ is a linear function and there must exists $x \in J$ such that
\[
 U_{\mu_{\infty}}(x-1) - U_{\mu_{\infty}}(x) < U_{\mu_{\infty}}(x) - U_{\mu_{\infty}}(x+1).
\]
However, this is equivalent to 
\[
    \mathbb{P}^{\lambda}\left[X_{\rho_{\infty}} = x \right] = - \frac{1}{2} \left( U_{\mu_{\infty}}(x-1) + U_{\mu_{\infty}}(x+1) \right) + U_{\mu_{\infty}}(x) > 0
\]
which contradicts Equation \eqref{eq:P0}. 
Hence $J = \emptyset$, thus also $I = \emptyset$ and our claim is proven.
\end{proof}
%
%
%
%
%
%
\subsubsection{Existence of Discrete Delayed Rost Solutions}
%
Given a delayed stopping time $\rho$ corresponding to a field of stopping probabilities $(r_t(x))_{(t,x) \in \mathbb{N} \times \mathbb{Z}}$ 
we consider the quantity 
\[
    \tilde \mu_t (x) := \mathbb{P}^{\lambda} \left[ X_{\rho} = x, \rho \leq t \right],
\]
the `mass embedded until time $t$'.
Then
\[
    \tilde \mu_t (x) = \sum_{s=0}^{t} \tilde \alpha_{s}(x)r_s(x)
\]
for $\tilde \alpha_t$ as defined in the previous section. 
Furthermore consider
\[
    \nu_t (x) := \mu\left(\{x\}\right) - \tilde \mu_{t-1} (x),
\]
the `mass not yet embedded' and note that it can be computed inductively using information only up to time $t-1$. 
We can then also write
\[
    \tilde \mu_t (x) = \mu\left(\{x\}\right) - \nu_t (x) + \tilde \alpha_t(x)r_t(x).
\]
Note that if the field of stopping probabilities embeds the measure $\mu$, then 
\[
    \lim_{t \rightarrow \infty} \tilde \mu_t (x) = \mu\left(\{x\}\right) 
    \text{ resp. } 
    \lim_{t \rightarrow \infty} \nu_t (x) = 0.
\]
With the help of these quantities we can give an explicit construction of Rost solution to the discrete \eqref{dSEP}.
\begin{theorem} \label{thm:ddSEP-Rost}
Consider a field of stopping probabilities $(r_t(x))_{(t,x) \in \mathbb{N} \times \mathbb{Z}}$ inductively defined via the formula
\[
    r_t(x) 
           = \frac{\nu_t(x)}{ \tilde \alpha_t(x) } \wedge 1 
           = \frac{ \left(\tilde \alpha_t  \wedge \nu_t \right) \left(x\right) }{ \tilde \alpha_t\left(x\right) } 
\text{ when }  \tilde \alpha_t(x) > 0
\]
and chosen to comply with \eqref{eq:Root-field-def} otherwise.  

Then this field represents a Rost field of stopping probabilities 
and the associated stopping time
\begin{equation*}
    \rho = \inf \left\{ t \geq \eta : r_t(X_t) > u_t \right\}.
\end{equation*}
will be a Rost solution to the discrete \eqref{dSEP}. 
\end{theorem}
\begin{proof}
As both $\nu_t$ and $\tilde \alpha_t$ only use information that is either given or acquired up to time $t-1$, it is clear that $r_t$ can be computed iteratively through time, thus in each time step $t$ we identify the $t$-layer $r_t(x)_{x \in \mathbb{Z}}$ of the probability field. 
Let us justify that $(r_t(x))_{(t,x) \in \mathbb{N} \times \mathbb{Z}}$ defined this way both embeds the right measure $\mu$ and is of Rost structure. 

Firstly, if $\mu(\{x\}) = 0$ then $r_t(x) = 0$ for all $t \in \mathbb{N}$, thus consider only those $x$ such that $\mu(\{x\}) > 0$. 
Then consider the following cases.
\begin{enumerate}
\item[$t = 0:$] Only consider those $x$ such that $\tilde \alpha_0(x) = \alpha(0,x) > 0$ as otherwise the point $x$ cannot be reached by the process $X$ at time $0$ after time $\eta$.
    \begin{itemize}        
        \item[\itembullet] If $\tilde \alpha_0(x) \leq \mu(\{x\})$, then all mass available in $(0,x)$ should be stopped, hence $r_0(x) = 1$ and $\tilde \mu_0 (\{x\}) = \tilde \alpha_0(x) \leq \mu(\{x\})$. 
        \item[\itembullet] If $\tilde \alpha_0(x)  > \mu(\{x\})$, 
            then some mass needs to be stopped in $(0,x)$, precisely $r_0(x) = \frac{\mu(\{x\})}{\tilde \alpha_t(x)}$ as then we have 
            $\tilde \mu_0 (\{x\}) = \mu(\{x\})$.
    \end{itemize}
\item[$t > 0:$] Only consider those $x$ such that $\tilde \alpha_t(x) > 0$ as otherwise the point $x$ cannot be reached by the process $X$ at time $t$ after time $\eta$.
    \begin{itemize}
        \item[\itembullet] If $\nu_t(x) = 0$ then all necessary mass has already been embedded before time $t$, hence $r_t(x) = 0$. 
        \item[\itembullet] If $0 < \tilde \alpha_t(x) \leq \nu_t(x)$ then the available mass in site $x$ at time $t$ does not exceed the mass still needed at site $x$, hence $r_t(x) = 1$ as all mass can be added.
        \item[\itembullet] If $ 0 < \nu_t(x) < \tilde \alpha_t(x)$ then we have mass missing in site $x$ but the available mass at time $t$ exceeds the necessary mass, hence we should only add the fraction $r_t(x) = \frac{\nu_t(x)}{\tilde \alpha_t(x)}$ as then we have 
        $\tilde \mu_t (\{x\}) = \mu(\{x\}) - \nu_t(x) + \tilde \alpha_t(x) \frac{\nu_t(x)}{\tilde \alpha_t(x)} = \mu(\{x\})$.
        \end{itemize}
\end{enumerate}
It is crucial to note that the quantity $\nu_t(x)$ is decreasing and remains at 0 once it has reached there, hence clearly the field of stopping probabilities $(r_t(x))_{(t,x) \in \mathbb{N} \times \mathbb{Z}}$ defined as above is of Rost structure.

We are left to formally show that the stopping time
\[
   \rho_{\infty} := \inf \left\{ t \geq \eta: r_t(X_t) > u_t \right\}
\]
embeds the right measure into the SSRW, i.e. $\mathbb{P}^{\lambda}[X_{\rho_{\infty}} = x] = \mu(\{x\})$ for all $x \in \mathbb{Z}$.

We consider the measure 
\begin{align*}
    \mu_{\infty}(\{x\}) 
        := \mathbb{P}^{\lambda}[X_{\rho_{\infty}} = x] 
        = \lim_{t \rightarrow \infty} \tilde \mu_t(x)
        = \sum_{s=0}^{\infty} \tilde \alpha_s(x)p_s(x).
\end{align*}
Fist assume $\mu_{\infty}(\{x\}) > \mu(\{x\})$.
As this would be equivalent to $\lim_{t \rightarrow \infty} \nu_t(x) < 0$ there would need to exist a first $T \in \mathbb{N}$ such that $\nu_T(x) < 0$.  
However this contradicts the construction of $r_t(x)$ as described above.
\\ \
Hence we must have $\mu_{\infty}(\{x\}) \leq \mu(\{x\})$ for all $x \in \mathbb{Z}$. 

Consider the set $I:= \left\{x \in \mathbb{Z}: \mu_{\infty}(\{x\}) < \mu(\{x\})\right\}$. 
If we assume that $I$ is not empty then 
\[
    0 \leq \sum_{x \in \mathbb{Z}} \mu_{\infty}(\{x\}) < \sum_{x \in \mathbb{Z}} \mu(\{x\}) = 1.
\]
Note that $\mu_{\infty}(\{x\}) = \lim_{t \rightarrow \infty} \tilde \mu_t(x)  < \mu(\{x\})$ is equivalent to $\nu_t(x) > 0$ for every $t \in \mathbb{N}$, which in turn implies 
$r_t(x) = 1$ for all $t\in \mathbb{N}$ by the construction above. 
In particular, this implies $\rho_{\infty} \leq H_x^{\eta} := \inf\{ t \geq \eta : X_t = x \}$ which is a finite stopping time, hence also
$\mathbb{P}^{\lambda}[\rho_{\infty} < \infty] = 1$.
So $X_{\rho}$ is a well defined random variable and we must have $\sum_{x \in \mathbb{Z}} \mu_{\infty}(\{x\}) = 1$ which is a contradiction. 
This implies $I = \emptyset$ and $\mu_{\infty}(\{x\}) = \mu(\{x\})$, 
hence our constructed field of stopping probabilities and consequently the associated stopping time $\rho_{\infty}$ embeds the right measure.
\end{proof}
%
%
%
%
%
%
%
%
%
%
%
\subsection{A Discrete Optimal Stopping Representation} \label{sec:discrete-OSP}
%
In this section we will consider a second simple symmetric second random walk $Y$ on the probability space $(\Omega, \mathbb{P})$, 
mutually independent of the SSRW $X$ introduced in previous section. 
For given $x,y\in \mathbb{Z}$ we recall $\mathbb{P}^x$ as the conditional distribution given $X_0=x$ 
and analogously define $\mathbb{P}_y$ as the conditional distribution given $Y_0=y$. 
Similarly, conditioning on both events simultaneously is denoted by $\mathbb{P}^x_y$. 
We recall the starting distribution $\lambda$ for the process $X$ and introduce 
a corresponding starting distribution $\nu$ for $Y$ analogously, 
leading to $\mathbb{P}_{\nu}$ as well as $\mathbb{P}^{\lambda}_{\nu}$. 

From this point onward we envision the process $X$ evolving ``rightwards'' 
from some lattice point $(0,x)$ (where possibly $x = X_0 \sim \lambda$) 
and the process $Y$ evolving ``leftwards'' from the lattice point $(T,y)$ at time zero. 
In this context a (stopping) time $\tau$ for the ``backward'' process $Y$ will be measured as $T- \tau$ for the ``forward'' process $X$.

Let us define the discrete time equivalents of the objects introduced in Section \ref{sec:delayed}. 
We consider a discrete delay stopping time $\eta \in \mathbb{N} = \{0,1,\dots\}$ for the process $X$. 
Then analogous to the continuous case we define the space time measure 
\begin{equation*} 
    \alpha := \mathcal{L}^{\lambda}((\eta, X_{\eta})) \in \mathcal P (\mathbb{N} \times \mathbb{Z}),
\end{equation*}
its spatial marginal 
\[
    \alpha_X (A) := \alpha(\mathbb{N} \times A) 
        = \mathcal{L}^{\lambda}(X_{\eta})(A)
         \quad \text{ for } A \subseteq \mathbb{Z} \text{ measurable}
\]
and the $\mathbb{P}^{\lambda}$-potential of $X_{\eta \wedge T}$ by
\[
    V^{\alpha}_T (y) 
    := -\E^{\lambda}\left[\left|X_{\eta \wedge T} - y \right| \right]. 
\]
For a Root (resp. Rost) Barrier $R \subseteq \mathbb{Z}\times \mathbb{Z}$ consider the discrete, delayed Root (resp. Rost) stopping time 
\begin{equation*}
    \rho_{\eta} := \inf \left\{t \geq \eta : (t, X_t) \in R \right\}.
\end{equation*}
Similarly for a Root (resp. Rost) field of stopping probabilities $(r_t(x))_{(t,x) \in \mathbb{N} \times \mathbb{Z}}$ 
and a vector $U = (u_0, u_1, \dots)$ of iid $U[0,1]$-distributed random variables we consider the discrete, delayed Root (resp. Rost) stopping time 
\begin{equation*}
    \rho_{\eta} := \inf \left\{t \geq \eta   : r_t(X_t) > u_t \right\}.
\end{equation*}
We will throughout this paper assume that $\left(X_{\rho^{\eta} \wedge t}\right)_{t \in \mathbb{N}}$ is uniformly integrable. 
The stopping time $\rho_{\eta}$ induces the measures
\[ \beta^{\alpha} := \mathcal{L}^{\lambda} \left(X_{\rho_{\eta}}\right)
        \text{ as well as }
    \beta^{\alpha}_T := \mathcal{L}^{\lambda} \left(X_{\rho_{\eta} \wedge T}\right) \text{ for } T \in \mathbb{N}
\]
as well as the (stopped) potential
\[
  U_{\beta^{\alpha}_T}(y) = - \E^{\lambda}\left[\left|X_{\rho_{\eta} \wedge T} - y \right| \right].
\]
Moreover, given a field $(r_t(x))_{(t,x) \in \mathbb{N} \times \mathbb{Z}}$ of Root (resp. Rost) stopping probabilities we denote by $\mathcal{R}_r$ the set of all fields of Root (resp. Rost) stopping probabilities that agree with $r$ everywhere except on the boundary, precisely
\begin{equation*} 
  \mathcal{R}_r := \left\{ (s_t(x))_{(t,x) \in \mathbb{N} \times \mathbb{Z}}  : \parbox{8.5cm}{
    $(s_t(x))_{(t,x) \in \mathbb{N} \times \mathbb{Z}}$ is a Root (resp. Rost) field of stopping probabilities s.t. $s_t(x) = r_t(x)$ given $r_t(x) \in \{0,1\}$.
    } \right\}.
\end{equation*}
We will sometimes use the shorthand notation $s \in \mathcal{R}_r$ for $s = (s_t(x))_{(t,x) \in \mathbb{N} \times \mathbb{Z}}$ when the context makes it obvious. 
Furthermore, we define the set of all associated stopping times in the following way
\begin{equation*} 
  \mathcal{T}_r := \left\{ \inf \left\{t \geq 0 : s_{T-t}(Y_t) > v_t \right\} : \parbox{8cm}{
    $(s_t(x))_{(t,x) \in \mathbb{N} \times \mathbb{Z}} \in \mathcal{R}_r$ and $V = (v_0, \dots, v_T)$ 
    denotes a vector of iid $U[0,1]$-distributed random variables.}
\right\} .
\end{equation*}

Then we can then give the following discrete time analogue of \eqref{eq:delayed-Root}-\eqref{eq:delayed-RootOSP}, 

\begin{theorem} \label{thm:discrete-delayed-Root}
Let $(r_t(x))_{(t,x) \in \mathbb{N} \times \mathbb{Z}}$ denote a Root field of stopping probabilities, then we have the representation 
\begin{align} 
    \label{eq:discrete-delayed-Root}
     U_{\beta^{\alpha}_T}(y) &= \E_y\left[ V^{\alpha}_{T - \tau^*} (Y_{\tau^*}) 
                              + (U_{\beta^{\alpha}} - U_{\alpha_X})(Y_{\tau^*})\ind_{\tau^*  < T}\right]   
\\ \label{eq:discrete-delayed-RootOSP}
                            &= \sup_{\tau \leq T} \E_y\left[ V^{\alpha}_{T - \tau} (Y_{\tau}) 
                                              + (U_{\beta^{\alpha}} - U_{\alpha_X})(Y_{\tau})\ind_{\tau  < T}\right],
\end{align}
where all optimizers will be of the form $\tau^* := \tau \wedge T$ for $\tau \in \mathcal{T}_r$.  
\end{theorem}
When $\rho_{\eta}$ corresponds to a discrete delayed Rost stopping the discrete version of 
\eqref{eq:delayed-Rost}-\eqref{eq:delayed-RostOSP} reads the following. 
\begin{theorem} \label{thm:discrete-delayed-Rost}
Let $(r_t(x))_{(t,x) \in \mathbb{N} \times \mathbb{Z}}$ denote a Rost field of stopping probabilities, 
then we have the representation
\begin{align}
    \label{eq:discrete-delayed-Rost}
    U_{\beta^{\alpha}}(x) - U_{\beta^{\alpha}_T}(x)  &= \E^x\left[ U_{\beta^{\alpha}} (X_{\sigma_*})  
            - V^{\alpha}_{T - \sigma_*} (X_{\sigma_*}) \right]   
\\ \label{eq:discrete-delayed-RostOSP}
        &= \sup_{\sigma \leq T} \E^x\left[ U_{\beta^{\alpha}} (X_{\sigma}) 
            - V^{\alpha}_{T - \sigma} (X_{\sigma}) \right],
\end{align}
where all optimizers will be of the form $\sigma^* := \sigma \wedge T$ for $\sigma \in \mathcal{T}_r$.
\end{theorem}
It will become clear in Section \ref{sec:Rost-OSP} why the processes $X$ and $Y$ appear differently in the Root and Rost optimal stopping problem.
%
%
%
%
%
\subsubsection{The Core Argument}
%
We present the core argument repeatedly used in \cite{BaCoGrHu21} and crucial for further results, 
an equality in expectation for differences of independent random walks.
\begin{equation}\label{eq:core}
    \Exy\left [ \left| X_{T-s}-Y_{s} \right|\right ] = \Exy\left [ |X_{T-(s-1)}-Y_{s-1}|\right ].
\end{equation}
The following two remarks break down the key ingredients needed for proving the core argument.
\begin{remark} \label{rmk:independence-prep}
Let $Z$ be a Rademacher random variable. Then for any $a \in \mathbb{Z}$ we have
\begin{equation*}
    \E \left[ \left| a + Z \right| \right] = |a| + \ind_{a = 0}.
\end{equation*}
\end{remark}
\begin{remark} \label{rmk:independence}
For a sigma algebra $\mathcal{H}$ let $A$ be a $\mathcal{H}$-measurable random variable taking values in $\mathbb{Z}$ 
and let $Z$ be a Rademacher random variable independent of $\mathcal{H}$. 
Then by Remark \ref{rmk:independence-prep} and the Independence Lemma we can conclude
\begin{equation*} 
    \E \left[ |A + Z| \big{|} \mathcal{H} \right] = |A| + \ind_{A = 0}.
\end{equation*}
\end{remark}
We can now prove the core argument \eqref{eq:core}.
On the one hand we have
\begin{align*}
    \Exy\left [ |X_{T-s}-Y_{s}|\right ] 
        &= \Exy\left[\Exy\left[ |X_{T-s}-Y_{s}| \big| X_{T-s},Y_{s-1}\right]\right ] \\
        &= \Exy\left[\Exy\left[ |(X_{T-s} - Y_{s-1})-(Y_{s}-Y_{s-1})|\big| X_{T-s},Y_{s-1}\right]\right] \\
        &= \Exy\left[ |X_{T-s}-Y_{s-1}|+\1_{X_{T-s}=Y_{s-1}}\right ]
\end{align*}
where we applied Remark \ref{rmk:independence} for the sigma algebra $\mathcal{H}:= \sigma(X_{T-s},Y_{s-1})$, 
the $\mathcal{G}$-measurable random variables $A:=X_{T-s}-Y_{s-1}$ and the Rademacher random variable 
$Z:=Y_{s}-Y_{s-1}$ which is independent of $\mathcal{H}$. 
Analogously, we can show
\begin{align*}
    \Exy\left [ |X_{T-(s-1)}-Y_{s-1}|\right ]
        &= \Exy\left[\Exy\left[ |X_{T-s+1}-Y_{s-1}|\big| X_{T-s},Y_{s-1}\right ]\right ] \\
        &= \Exy\left[\Exy\left[ |(X_{T-s}-Y_{s-1})+(X_{T-s+1} -X_{T-s})|\big| X_{T-s},Y_{s-1}\right ]\right ] \\
        &= \Exy\left[ |X_{T-s}-Y_{s-1}|+\1_{X_{T-s}=Y_{s-1}}\right],
\end{align*}
so altogether we have \eqref{eq:core}. 
See also Figure \ref{fig:RootRostPlot} for an illustration of the appearance of the indicator function.
\begin{figure}
\centering
\begin{subfigure}{.49\linewidth}
\centering
\includegraphics[width = 1\linewidth]{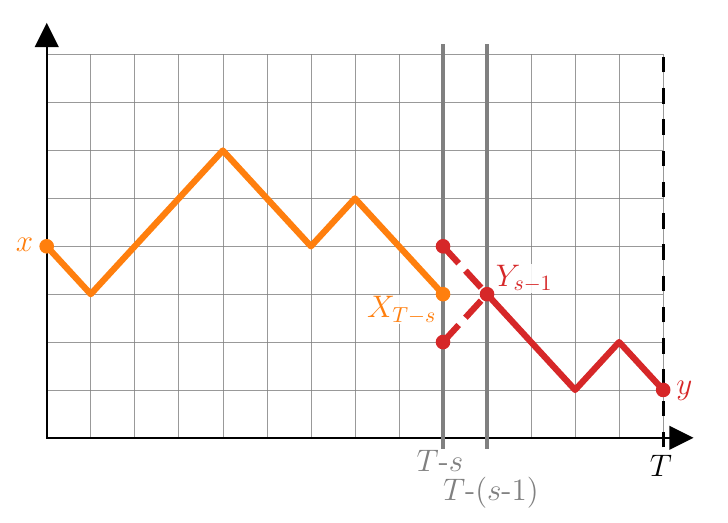}
\end{subfigure}
\begin{subfigure}{.49\linewidth}
\centering
\includegraphics[width = 1\linewidth]{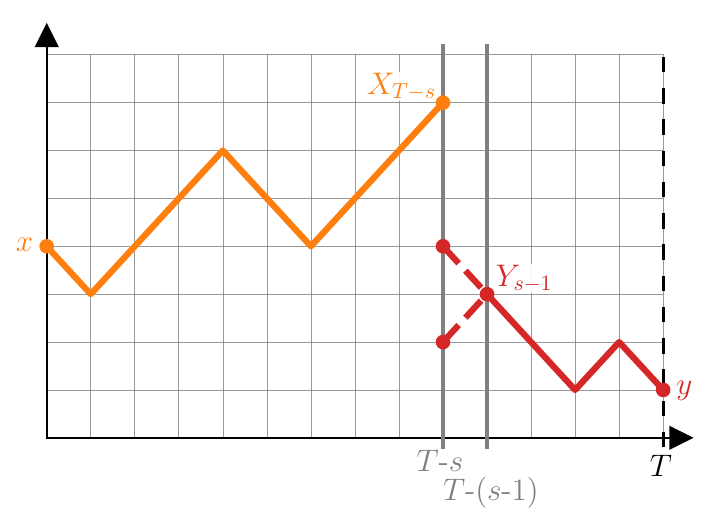}
\end{subfigure}
\caption{Illustrating the appearance of the indicator function in the core argument.}
\label{fig:RootRostPlot}
\end{figure}
In particular, Equation \eqref{eq:core} implies an independence of $s$, 
more precisely
\begin{equation*}
    \Exy\left [ |X_{T-s}-Y_{s}|\right ] = \Exy\left [ |X_{T}-Y_{0}|\right ]
    = \Exy\left [ |X_{0}-Y_{T}|\right ].
\end{equation*}
Note that we can replace $s$ by a $Y$ stopping time $\tau<T$ to receive 
\begin{equation*}
    \Exy\left [ |X_{T-\tau}-Y_{\tau}|\right ] = \Exy\left [ |X_{T}-Y_{0}|\right ].
\end{equation*}
Analogously, for an $X$ stopping time $\sigma < T$ we have 
\begin{equation*}
    \Exy\left [ |X_{\sigma}-Y_{T-\sigma}|\right ] = \Exy\left [ |X_{0}-Y_{T}|\right ],
\end{equation*}
collectively this leads to the identity 
\begin{equation} \label{eq:core-ST}
    \Exy\left [ |X_{T-\tau}-Y_{\tau}|\right ] = \Exy\left [ |X_{\sigma}-Y_{T-\sigma}|\right]
\end{equation}
which will serve as a key identity for subsequent arguments. 
%
%
%
%
%
\subsubsection{The Interpolation Function} \label{sec:interpolation-function}
%
In \cite{BaCoGrHu21} it was crucial to define an interpolating function $F(s)$ 
which connects between the LHS of Equation \eqref{eq:Root} and the RHS of Equations \eqref{eq:Root} and \eqref{eq:RootOSP}. 
With only slight modification of the proofs in \cite{BaCoGrHu21} we can
define a more general interpolation function and state the following lemma. 
\begin{lemma}\label{lem:interpolation-function}
Let $\lambda$ be a starting measure for the process $X$ and $\nu$ be a starting measure for the process $Y$. 
For an $X$ stopping time $\sigma$, a $Y$ stopping time $\tau$ and any $T \geq 0$ we consider the following interpolation function
\begin{equation}
    F^{\lambda, \nu}_{\sigma, \tau}(s) := 
        \E^{\lambda}_{\nu} \left[\left| X_{\sigma \wedge (T-\tau \wedge s)} - Y_{\tau \wedge s}\right|\right].
\end{equation}
Then we have the following:
\begin{enumerate}[(i)]
\item The function $F^{\lambda, \nu}_{\sigma, \tau}(s)$ is increasing in $s$, more precisely
\begin{align}
    F^{\lambda, \nu}_{\sigma, \tau}(s) 
        &= \E^{\lambda}_{\nu}\left[\left|X_{\sigma \wedge (T-\tau \wedge s)} - Y_{\tau \wedge (s-1)}\right| + 
            \ind_{X_{\sigma \wedge (T-s)} = Y_{s-1}, \tau \geq s} \right] \label{eq:F(s)}
\\      &\geq \E^{\lambda}_{\nu}\left[\left|X_{\sigma \wedge (T-\tau \wedge s)} - Y_{\tau \wedge (s-1)}\right| + 
            \ind_{X_{\sigma \wedge (T-s)} = Y_{s-1}, \tau \geq s, \sigma > T-s} \right]
        = F^{\lambda, \nu}_{\sigma, \tau}(s-1) \label{eq:F(s-1)}.
\end{align}
\item  If $\sigma$ is a (possibly randomized, possibly delayed) forward Root stopping time for the process $X$ 
and $\tau$ is a (possibly randomized) backwards Root stopping time for the process $Y$ 
then $F^{\lambda, \nu}_{\sigma, \tau}(s)$ is constant in $s$.
\end{enumerate}
\end{lemma}
\begin{proof}
The essence of the proof of $(i)$ is an appropriate application of the Core Argument presented above, 
and $(i)$ follows in complete analogy to the proofs of \cite[Lemma 3.3, Lemma 3.4]{BaCoGrHu21}.
For $(ii)$ we have to make some modifications. 

Let $(r_t(x))_{(t,x) \in \mathbb{N} \times \mathbb{Z}}$ denote a Root field of stopping probabilities, 
then for $(s_t(x))_{(t,x) \in \mathbb{N} \times \mathbb{Z}} \in \mathcal{R}_r$ 
consider the stopping times 
\begin{align*}
    \rho^{\eta} &:= \inf \left \{ t \geq \eta : r_t(X_t) > u_t \right\}, \,\, \text{ and }
\\  \tau^*      &:= \inf \left \{ t \geq 0 : s_{T-t}(Y_t) > v_t\right\} \wedge T. 
\end{align*}
In order to show equality in $(i)$ for these stopping times,
it suffices to show equality of the indicator functions appearing in \eqref{eq:F(s)} and \eqref{eq:F(s-1)}, 
equivalently 
\[
    \left\{\rho^{\eta} > T-s \} \supseteq \{X_{\rho^{\eta} \wedge (T-s)} = Y_{s-1}, \tau^* \geq s \right\}
\]
up to nullsets. 
Indeed, on $\{\tau^* \geq s\}$ we deduce $\tau^* > s-1$. 
Hence we must have $s_{T-(s-1)} \left(Y_{s-1}\right) < 1$ as otherwise 
$\tau^* = s-1$ almost surely. 
By the definition of the Root stopping probabilities we thus have $s_{u} \left(Y_{s-1}\right) = r_{u} \left(Y_{s-1}\right) = 0$ for all $u < T-(s-1)$. 
Moreover, on the set $\{X_{\rho^{\eta} \wedge (T-s)} = Y_{s-1}\}$ we have $r_{u} \left(X_{\rho^{\eta} \wedge (T-s)}\right) = 0$ for all $u < T-(s-1)$. 
In particular, since $\rho^{\eta} \wedge (T-s) < T-(s-1)$, on the set $\{X_{\rho^{\eta} \wedge (T-s)} = Y_{s-1}\}$ 
we have $r_{\rho^{\eta} \wedge (T-s)} \left(X_{\rho^{\eta} \wedge (T-s)}\right) = 0$. 
Hence $\rho^{\eta}> T-s$ needs to be satisfied as otherwise the Root properties of the stopping probabilities would be violated.
See Figure \ref{fig:Root-F} for an illustration of these arguments. 
\end{proof}
\begin{figure}
\centering
\begin{subfigure}{.49\linewidth}
\centering
\includegraphics[width = 1\linewidth]{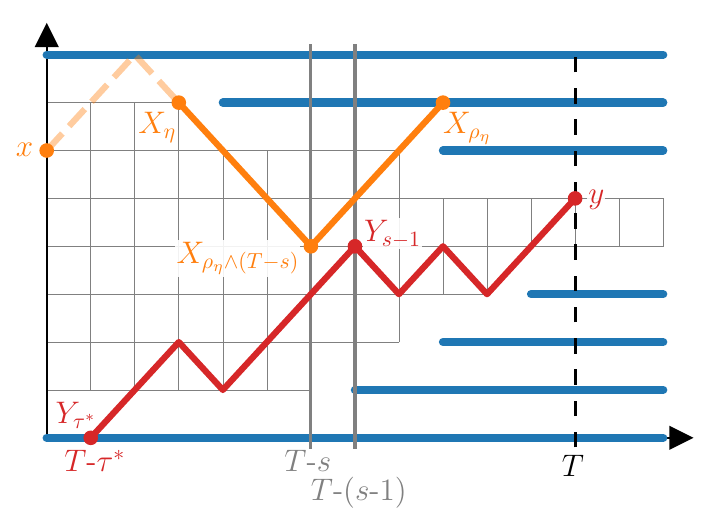}
\caption{Here $\tau^* \geq s$ and $X_{\rho_{\eta} \wedge (T-s)} = Y_{s-1}$.
We see that in this case $\rho_{\eta} \wedge (T-s) = T-s$ must hold as otherwise $Y_{\tau^*}$ would have stopped earlier.}
\end{subfigure}
\begin{subfigure}{.49\linewidth}
\centering
\includegraphics[width = 1\linewidth]{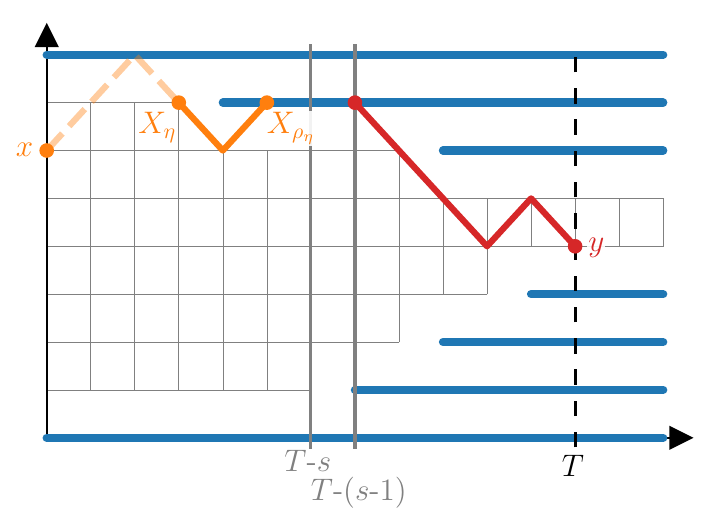}
\caption{Here $\rho_{\eta} \wedge (T-s) = \rho_{\eta}$. 
We see how $X_{\rho_{\eta} \wedge (T-s)} = Y_{s-1}$ implies that we must have $\tau^* \leq s-1$ due to the Root barrier structure.}
\end{subfigure}
\caption{An illustration of some different cases that appear when investigating equality of the interpolation function for forward and backward Root stopping times.}
\label{fig:Root-F}
\end{figure}
%
%
%
%
%
%
%
%
%
%
%
%
\subsection{The Root Optimal Stopping Representation} \label{sec:Root-OSP}
Let us introduce some frequently used notations. 
For the choice of $\nu = \delta_y$ we simplify the notation of the interpolation function 
between the forward and backward Root stopping times to
\[
    F^{\alpha}(s) := F^{\lambda, y}_{\rho_{\eta}, \tau^*}(s).
\]
Here $\alpha$ denotes the space-time law induced by the delay stopping time $\eta$ 
and $\beta^{\alpha}$ is the corresponding delayed Root law as defined in Section \ref{sec:discrete-OSP}. 
For a given stopping time $\tau$ consider the function
\begin{equation} \label{def:u(t,x)}
   u^{\alpha}_{\tau}(T,y) :=  \E_y\left[ V^{\alpha}_{T - \tau} (Y_{\tau}) 
        + (U_{\beta^{\alpha}} - U_{\alpha_X})(Y_{\tau})\ind_{\tau < T}\right]. 
\end{equation}
We also introduce  
\[
    u^{\alpha}(T,y) := u^{\alpha}_{\tau^*}(T,y)
\] 
for a backward Root stopping time 
$\tau^* = \inf \left \{ t \geq 0 : s_{T-t}(Y_t) > v_t\right\} \wedge T$. 
Considering this refined notation Equations \eqref{eq:discrete-delayed-Root}-\eqref{eq:discrete-delayed-RootOSP} 
can be written in the following way
\begin{align} \label{eq:delayed-Root-u}
     U_{\beta^{\alpha}_T}(y) &=  u^{\alpha}(T,y)
\\ \label{eq:delayed-RootOSP-u}
                            &= \sup_{\tau \leq T} u^{\alpha}_{\tau}(T,y).
\end{align}
Furthermore, it will be of importance that the function $u^{\alpha}_{\tau}$ has the following alternative representation
\begin{align}
u^{\alpha}_{\tau}(T,y) &=  \E_y\left[ V^{\alpha}_{T - \tau} (Y_{\tau}) 
                        + (U_{\beta^{\alpha}} - U_{\alpha_X})(Y_{\tau})\ind_{\tau < T}\right]
\\                     &= -\E^{\lambda}_y \left[ \left| X_{\eta \wedge (T-\tau)} - Y_{\tau} \right| 
                        + \left( \left| X_{\rho_{\eta}} - Y_{\tau}\right| - \left| X_{\eta} - Y_{\tau}\right| \right) \ind_{\tau < T} \right]. \label{eq:u-representation}
\end{align}
To further simplify and rewrite these expressions we observe the following lemma.
\begin{lemma} \label{lem:opt-stopping}
Let $\eta$ and $\theta$ be delay stopping times for the processes $X$ and $Y$ respectively. 
By $(r_t(x))_{(t,x) \in \mathbb{N} \times \mathbb{Z}}$ we denote a Root field of stopping probabilities. 
For mutually independent $U[0,1]$-distributed random variables $u_0, u_1, \dots, v_0, v_1, \dots$ we 
consider the forward Root stopping time $\rho := \inf \left \{ t \geq \eta : r_t(X_t) > u_t \right\}$ 
and for $s \in \mathcal{R}_r$ the backward Root stopping time 
$\tau := \inf \left \{ t \geq 0 : s_{T-t}(Y_t) > v_t\right\} \in \mathcal{T}_r$.

Furthermore we introduce the filtration 
$\mathcal{G}_t := \sigma \left( (X_s)_{0 \leq s \leq t}, (Y_s)_{0 \leq s \leq T}, u_1, \dots, u_t, v_1, \dots, v_T \right)$. 
Then for all $A \in \mathcal{H}:= \mathcal{G}_{\rho \wedge \left( \eta \vee (T- \tau)\right)}$ 
with $\{\tau < T\} \subseteq A$ we have 
\begin{align} \label{eq:opt-stopping}
    \E^{\lambda}_{\nu} \left[
        \left| X_{\rho} - Y_{\tau} \right| \ind_A\right] =
     \E^{\lambda}_{\nu} \left[
        \left| X_{\rho \wedge \left( \eta \vee (T- \tau)\right)} - Y_{\tau} \right| \ind_A \right].   
\end{align}
\end{lemma}
\begin{proof}
As $X$ is independent of $Y$, we have that $X$ is a $(\mathcal{G}_t)$ martingale 
and furthermore $\eta$, $\rho$ as well as $\tilde \sigma := \eta \vee (T- \tau)$ are $(\mathcal{G}_t)$ stopping times. 
Hence, by the optional sampling theorem we can conclude that for every $A \in \mathcal{H} = \mathcal{G}_{\rho \wedge \tilde \sigma}$ 
the following holds
\begin{equation} \label{eq:opt-samp}
    \E^{\lambda}_{\nu} \left[X_{\rho} \ind_A \right] = \E^{\lambda}_{\nu} \left[X_{\rho \wedge \tilde \sigma} \ind_A \right].
\end{equation}
By definition of the Root stopping probabilities and the definition of $\tau$, 
on the set $\{\tau < T\}$ we have $s_{T-\tau}\left( Y_{\tau}\right) > 0$. 
Furthermore $s_{T-\tau + u}\left( Y_{\tau}\right) = r_{T-\tau + u}\left( Y_{\tau}\right) = 1$ for all $u \geq 1$.
Moreover, on the set $\{\rho > \tilde \sigma, X_{\tilde \sigma} < Y_{\tau}\}$ we must have that $X_{\rho} \leq Y_{\tau}$ as otherwise the Root property of the stopping probabilities would be violated. 
We can argue analogously that on $\{\rho > \tilde \sigma, X_{\tilde \sigma} \geq Y_{\tau}\}$ we must have $X_{\rho} \geq Y_{\tau}$. 
Then 
\begin{align*}
    & \E^{\lambda}_{\nu} \left[\left| X_{\rho} - Y_{\tau} \right|\ind_A \right] 
\\  &= \E^{\lambda}_{\nu} \left[\left| X_{\rho} - Y_{\tau} \right|\ind_{A, \rho \leq \tilde \sigma}
        +\left| X_{\rho} - Y_{\tau} \right|\ind_{A, \rho > \tilde \sigma} \right]
\\  &= \E^{\lambda}_{\nu} \left[\left| X_{\rho} - Y_{\tau} \right|\ind_{A, \rho \leq \tilde \sigma} 
        +\left| X_{\rho} - Y_{\tau} \right| \ind_{A, \rho > \tilde \sigma, X_{\tilde \sigma} < Y_{\tau}}                               
        +\left| X_{\rho} - Y_{\tau} \right|\ind_{A, \rho > \tilde \sigma, X_{\tilde \sigma} \geq Y_{\tau}} \right]
\\  &= \E^{\lambda}_{\nu} \left[\left| X_{\rho} - Y_{\tau} \right|\ind_{A, \rho \leq \tilde \sigma} 
        -\left( X_{\rho} - Y_{\tau} \right) \ind_{A, \rho > \tilde \sigma, X_{\tilde \sigma} < Y_{\tau}}                               
        +\left( X_{\rho} - Y_{\tau} \right)\ind_{A, \rho > \tilde \sigma, X_{\tilde \sigma} \geq Y_{\tau}} \right]
\\  &= \E^{\lambda}_{\nu} \left[\left| X_{\rho} - Y_{\tau} \right|\ind_{A, \rho \leq \tilde \sigma} 
        -\left( X_{\rho \wedge \tilde \sigma} - Y_{\tau} \right) \ind_{A, \rho > \tilde \sigma, X_{\tilde \sigma} < Y_{\tau}}                               
        +\left( X_{\rho \wedge \tilde \sigma} - Y_{\tau} \right)\ind_{A, \rho > \tilde \sigma, X_{\tilde \sigma} \geq Y_{\tau}} \right]
\\      &= \E^{\lambda}_{\nu} \left[\left| X_{\rho \wedge \tilde \sigma} - Y_{\tau} \right| \ind_A \right]. 
\end{align*}
where we use Equation \eqref{eq:opt-samp} and the fact that 
$\{\rho > \tilde \sigma, X_{\tilde \sigma} < Y_{\tau} \}, \{\rho > \tilde \sigma, X_{\tilde \sigma} \geq Y_{\tau}\} \in \mathcal{H}$. 
\end{proof}
Choosing $\nu = \delta_y$, $\rho = \rho_{\eta}$ and $\theta = 0$ in the lemma above enables us to prove the following.
\begin{lemma} \label{lem:Root-prep}
The RHS of Equation \eqref{eq:discrete-delayed-Root} equals the terminal value of the interpolation function, that is
\begin{align} \label{eq:first-equality}
    -u^{\alpha}(T,y) = F^{\alpha}(T).
\end{align}
\end{lemma}
\begin{proof}
Recall the representation \eqref{eq:u-representation} and consider the following decomposition
\begin{align}
-u^{\alpha}(T,y)
    = &\,\E^{\lambda}_y \left[ \left| X_{\eta \wedge (T-\tau^*)} - Y_{\tau^*} \right| 
        + \left( \left| X_{\rho_{\eta}} - Y_{\tau^*}\right| - \left| X_{\eta} - Y_{\tau^*}\right| \right) \ind_{\tau^* < T} \right] \nonumber
\\  = &\,\E^{\lambda}_y \left[ \left| X_{\eta \wedge (T-\tau^*)} - Y_{\tau^*} \right|\ind_{\eta > T} 
        + \left( \left| X_{\rho_{\eta}} - Y_{\tau^*}\right| - \left| X_{\eta} - Y_{\tau^*}\right| \right) \ind_{\tau^* < T, \eta > T} \right] \label{eq:u1}
\\  + &\,\E^{\lambda}_y \left[ \left| X_{\eta \wedge (T-\tau^*)} - Y_{\tau^*} \right|\ind_{\eta \leq T} 
        + \left( \left| X_{\rho_{\eta}} - Y_{\tau^*}\right| - \left| X_{\eta} - Y_{\tau^*}\right| \right) \ind_{\tau^* < T, \eta \leq T} \right]. \label{eq:u2}
\end{align}
We will investigate the two lines \eqref{eq:u1} and \eqref{eq:u2} separately. 

Let us first consider \eqref{eq:u1}. 
As $\rho_{\eta} \geq \eta$, we observe that on the event $\{\eta > T\}$ 
we have $\eta \wedge (T-\tau^*) = T-\tau^* = \rho_{\eta} \wedge (T-\tau^*)$, 
hence for the first term of \eqref{eq:u1} we can conclude 
\begin{equation} \label{eq:u1-a}
    \E^{\lambda}_y \left[ \left| X_{\eta \wedge (T-\tau^*)} - Y_{\tau^*} \right|\ind_{\eta > T} \right] 
        = \E^{\lambda}_y \left[ \left| X_{\rho_{\eta} \wedge (T-\tau^*)} - Y_{\tau^*} \right|\ind_{\eta > T} \right].
\end{equation}
Furthermore, note the following:
\begin{itemize}
\item[\itembullet] On the event $\{\tau^* < T\}$ the stopping time $\tau^*$ equals a backward Root stopping time.
\item[\itembullet] On the event $\{\eta > T\}$ we have $\eta \vee (T- \tau^*) = \eta = \rho_{\eta} \wedge (\eta \vee (T- \tau^*))$.
\item[\itembullet] We have $\{\tau^* < T, \eta > T\} \in \mathcal{H}$ (for $\mathcal{H}$ defined in Lemma \ref{lem:opt-stopping}).
\end{itemize}
Therefore we can apply Lemma \ref{lem:opt-stopping} to the second term of \eqref{eq:u1} to arrive at
\begin{align}
    &\, \E^{\lambda}_y \left[ \left( \left| X_{\rho_{\eta}} - Y_{\tau^*}\right| - \left| X_{\eta} - Y_{\tau^*}\right| \right) \ind_{\tau^* < T, \eta > T} \right] \nonumber
\\  = &\, \E^{\lambda}_y \left[ \left( \left| X_{\rho_{\eta}} - Y_{\tau^*}\right| - \left| X_{\rho_{\eta} \wedge (\eta \vee (T-\tau^*))} - Y_{\tau^*}\right| \right) \ind_{\tau^* < T, \eta > T} \right] 
  = \,0. \label{eq:u1-b}
\end{align}
To summarize, line \eqref{eq:u1} reduces to 
\begin{equation} \label{eq:u1-summary}
    \E^{\lambda}_y \left[ \left| X_{\rho_{\eta} \wedge (T-\tau^*)} - Y_{\tau^*} \right|\ind_{\eta > T} \right].
\end{equation}
Next we consider \eqref{eq:u2}. 
Let us start by examining the first term of \eqref{eq:u2} which can be decomposed further into
\begin{align}
      &\, \E^{\lambda}_y \left[ \left| X_{\eta \wedge (T-\tau)} - Y_{\tau^*} \right|\ind_{\eta \leq T} \right] \nonumber
\\  = &\, \E^{\lambda}_y \left[ \left| X_{\eta \wedge (T-\tau)} - Y_{\tau^*} \right|\ind_{\eta \leq T, T- \eta \leq \tau^*} \right]
     + \E^{\lambda}_y \left[ \left| X_{\eta \wedge (T-\tau)} - Y_{\tau^*} \right|\ind_{\eta \leq T, \tau^* < T - \eta} \right] \nonumber
\\  = &\, \E^{\lambda}_y \left[ \left| X_{\rho_{\eta} \wedge (T-\tau)} - Y_{\tau^*} \right|\ind_{\eta \leq T, T- \eta \leq \tau^*} \right]
     + \E^{\lambda}_y \left[ \left| X_{\eta} - Y_{\tau^*} \right|\ind_{\eta \leq T, \tau^* < T - \eta} \right]. \label{eq:u2-a}
\end{align}
Similarly, the second term of \eqref{eq:u2} decomposes into 
\begin{align}
    & \E^{\lambda}_y \left[\left( \left| X_{\rho_{\eta}} - Y_{\tau^*}\right| - \left| X_{\eta} - Y_{\tau^*}\right| \right) \ind_{\tau^* < T, \eta \leq T} \right] \nonumber
\\  = &\, \E^{\lambda}_y \left[\left( \left| X_{\rho_{\eta}} - Y_{\tau^*}\right| - \left| X_{\eta} - Y_{\tau^*}\right| \right) \ind_{\tau^* < T, \eta \leq T, T- \eta \leq \tau^*} \right] \label{eq:u2-2a}
\\  + &\, \E^{\lambda}_y \left[\left( \left| X_{\rho_{\eta}} - Y_{\tau^*}\right| - \left| X_{\eta} - Y_{\tau^*}\right| \right) \ind_{\eta \leq T, \tau^* < T- \eta} \right]. \label{eq:u2-2b}
\end{align}
Importantly, the second term of \eqref{eq:u2-a} cancels out the second term of \eqref{eq:u2-2b}.

We will proceed to conclude that the terms depicted in \eqref{eq:u2-2a} will equate to 0. 
For this purpose observe that: 
\begin{itemize}
\item[\itembullet] On the event $\{\tau^* < T\}$ the stopping time $\tau^*$ equals a backward Root stopping time.
\item[\itembullet] On the event $\{T - \eta \leq \tau^*\}$ we have $\eta \vee (T- \tau^*) = \eta = \rho_{\eta} \wedge (\eta \vee (T- \tau^*))$.
\item[\itembullet] We have $\{\tau^* < T, \eta \leq T, T - \eta \leq \tau^*\} \in \mathcal{H}$.
\end{itemize}
Hence, we can apply Lemma \ref{lem:opt-stopping} to deduce 
\begin{align}
    &\, \E^{\lambda}_y \left[\left( \left| X_{\rho_{\eta}} - Y_{\tau^*}\right| - \left| X_{\eta} - Y_{\tau^*}\right| \right) \ind_{\tau^* < T, \eta \leq T, T- \eta \leq \tau^*} \right] \nonumber
\\  = &\, \E^{\lambda}_y \left[\left( \left| X_{\rho_{\eta}} - Y_{\tau^*}\right| - \left| X_{\rho_{\eta} \wedge (\eta \vee (T-\tau^*))} - Y_{\tau^*}\right| \right) \ind_{\tau^* < T, \eta \leq T, T- \eta \leq \tau^*} \right] \nonumber
\\  = &\,0. \label{eq:u2-b}
\end{align}
The final term left to consideration is the first term in \eqref{eq:u2-2b}. 
For this purpose, note that
\begin{itemize}
\item[\itembullet] on the event $\{\tau^* < T - \eta\}$ we have $\eta \vee (T- \tau^*) = (T- \tau^*)$, and 
\item[\itembullet] $\{\eta \leq T, \tau^* < T - \eta\} \in \mathcal{H}$.
\end{itemize}
Hence, application of Lemma \ref{lem:opt-stopping} to the first term of \eqref{eq:u2-2b} yields
\begin{align}
    &\, \E^{\lambda}_y \left[\left| X_{\rho_{\eta}} - Y_{\tau^*}\right| \ind_{\eta \leq T, \tau^* < T- \eta} \right] 
 = \E^{\lambda}_y \left[\left| X_{\rho_{\eta} \wedge (T - \tau^*)} - Y_{\tau^*}\right| \ind_{\eta \leq T, \tau^* < T- \eta} \right].
 \label{eq:u2-c}
\end{align}
To summarize, \eqref{eq:u2} reduces to 
\begin{equation} \label{eq:u2-summary}
    \E^{\lambda}_y \left[\left| X_{\rho_{\eta} \wedge (T - \tau^*)} - Y_{\tau^*}\right| \ind_{\eta \leq T} \right].
\end{equation}
Combining the observations \eqref{eq:u1-summary} and \eqref{eq:u2-summary} 
the desired identity now becomes easy to see:
\[
    -u^{\alpha}(T,y) 
    = \E^{\lambda}_y \left[\left| X_{\rho_{\eta} \wedge (T - \tau^*)} - Y_{\tau^*}\right|\right] = F^{\alpha}(T).
\] 
\end{proof}
We are ready to prove Theorem \ref{thm:discrete-delayed-Root}, the optimal stopping representation of Root solutions to \eqref{dSEP}.
\begin{proof}[Proof of Theorem \ref{thm:discrete-delayed-Root}]
We begin by proving \eqref{eq:discrete-delayed-Root} through consideration of the interpolation function 
\[
 F^{\alpha}(s) = F^{\lambda, y}_{\rho_{\eta}, \tau^*}(s) 
               = \E^{\lambda}_{y} \left[\left| X_{\rho_{\eta} \wedge (T-\tau^* \wedge s)} - Y_{\tau^* \wedge s}\right|\right].
\]
Then $-F^{\alpha}(0) = \E^{\lambda}_{y} \left[\left| X_{\rho_{\eta} \wedge T} - Y_{0}\right|\right] = U_{\beta^{\alpha}_T}(y)$ which equals the LHS of \eqref{eq:discrete-delayed-Root}, equivalently of \eqref{eq:delayed-Root-u}. 
It was the purpose of Lemma \ref{lem:Root-prep} to show that $-F^{\alpha}(T)$ equals the RHS of \eqref{eq:discrete-delayed-Root}, 
equivalently of \eqref{eq:delayed-Root-u}. 
The equality in \eqref{eq:discrete-delayed-Root} now follows by Lemma \ref{lem:interpolation-function} $(ii)$, 
which states that the interpolation function $F^{\alpha}(T)$ is constant given our consideration of Root stopping times.

Let us proceed to establish the optimal stopping problem \eqref{eq:discrete-delayed-RootOSP}. 
We keep the Root stopping time $\rho_{\eta}$ but allow for an arbitrary $Y$ stopping time $\tau \leq T$. 
It still remains true that $F^{\lambda, y}_{\rho_{\eta}, \tau}(0) = - U_{\beta^{\alpha}}(y)$, 
and due to Lemma \ref{lem:interpolation-function} $(i)$ we have the inequality 
$- U_{\beta^{\alpha}}(y) \leq F^{\lambda, y}_{\rho_{\eta}, \tau}(T)$. 
Therefore, we are left to show that 
\begin{equation*}
F^{\lambda, y}_{\rho_{\eta}, \tau}(T) 
    \leq u^{\alpha}_{\tau}(T,y)
    = -\E_y\left[ V^{\alpha}_{T - \tau} (Y_{\tau}) 
        + (U_{\beta^{\alpha}} - U_{\alpha_X})(Y_{\tau})\ind_{\tau  < T}\right],
\end{equation*}
or equivalently
\begin{equation} \label{eq:Root-inequality}
 \E^{\lambda}_{y} \left[\left| X_{\rho_{\eta} \wedge (T-\tau)} - Y_{\tau}\right|\right]
\leq 
\E^{\lambda}_{y} \left[\left| X_{\eta \wedge (T-\tau)} - Y_{\tau}\right|     
    + \left( \left| X_{\rho_{\eta}} - Y_{\tau}\right|
    - \left| X_{\eta} - Y_{\tau}\right| \right) \ind_{\tau < T} \right].
\end{equation}
We will consider a decomposition into the events $\{\eta > T-\tau\}$ and $\{\eta \leq T-\tau\}$. 

Due to Jensen's inequality and optional sampling we have the following
\[
    \E^{\lambda}_{y} \left[\left( \left| X_{\rho_{\eta}} - Y_{\tau}\right|
    - \left| X_{\eta} - Y_{\tau}\right| \right) \ind_{\tau < T, \eta > T-\tau} \right] \geq 0
\]
and can conclude
\begin{align*}
    &\E^{\lambda}_{y} \left[\left| X_{\rho_{\eta} \wedge (T-\tau)} - Y_{\tau}\right| \ind_{\eta > T-\tau}\right]
\\  &= \E^{\lambda}_{y} \left[\left| X_{\eta \wedge (T-\tau)} - Y_{\tau}\right| \ind_{\eta > T-\tau}\right]
\\  &\leq \E^{\lambda}_{y} \left[\left| X_{\eta \wedge (T-\tau)} - Y_{\tau}\right| \ind_{\eta > T-\tau}
        + \left( \left| X_{\rho_{\eta}} - Y_{\tau}\right|
    - \left| X_{\eta} - Y_{\tau}\right| \right) \ind_{\tau < T, \eta > T-\tau} \right]. 
\end{align*}
On the other hand we have 
\begin{align*}
    &\E^{\lambda}_{y} \left[\left| X_{\rho_{\eta} \wedge (T-\tau)} - Y_{\tau}\right| \ind_{\eta \leq T-\tau}\right]
\\  &=\E^{\lambda}_{y} \left[\left| X_{\rho_{\eta} \wedge (T-\tau)} - Y_{\tau}\right| \ind_{\eta \leq T-\tau, \tau < T}
                    + \left| X_{0} - Y_{\tau}\right| \ind_{\eta \leq T-\tau, \tau = T} \right]
\\ &\leq \E^{\lambda}_{y} \left[\left| X_{\rho_{\eta}} - Y_{\tau}\right| \ind_{\eta \leq T-\tau, \tau < T}
                    + \left| X_{\eta} - Y_{\tau}\right| \ind_{\eta \leq T-\tau, \tau = T} \right]
\\ &= \E^{\lambda}_{y} \left[\left| X_{\eta} - Y_{\tau}\right| \ind_{\eta \leq T-\tau}
        + \left( \left| X_{\rho_{\eta}} - Y_{\tau}\right|
    - \left| X_{\eta} - Y_{\tau}\right| \right) \ind_{\eta \leq T-\tau, \tau < T} \right]
\\ &= \E^{\lambda}_{y} \left[\left| X_{\eta \wedge (T-\tau)} - Y_{\tau}\right| \ind_{\eta \leq T-\tau}
        + \left( \left| X_{\rho_{\eta}} - Y_{\tau}\right|
    - \left| X_{\eta} - Y_{\tau}\right| \right) \ind_{\eta \leq T-\tau, \tau < T} \right]
\end{align*}
where the inequality again follows by Jensen's inequality and optional sampling.
Together this proves \eqref{eq:Root-inequality}. 
Thus we have shown that
\[
    U_{\beta^{\alpha}}(y) \geq u^{\alpha}_{\tau}(T,y)
\]
which concludes the proof of \eqref{eq:discrete-delayed-RootOSP}.
\end{proof}
%
%
%
%
%
%
%
%
%
%
%
\subsection{The Root-Rost-Symmetry} \label{sec:Rost-OSP}
%
To establish the optimal stopping representation of Rost solution to \eqref{dSEP} we could simply repeat the appropriate analogues of the arguments in the section above. 
Instead however we will derive the Rost representation from the Root one by observing a symmetry between the two similar to the approach in \cite{BaCoGrHu21}. 

We will first give a brief summary of the results in \cite{BaCoGrHu21} which were stated in non-randomized terms. 

Recall the definition of a discrete Rost barrier $\bar R \subseteq \mathbb{N} \times \mathbb{Z}$, that is 
\begin{itemize}
\item[\itembullet] If $(t,m)\in \bar R$ then for all $s < t$ also $(s,m) \in \bar R$.
\end{itemize}
Now observe that for any $T \in \mathbb{N}$ we can transform this Rost barrier into a Root Barrier $R_T$ in the following way
\begin{equation} \label{def:R_T}
    R_{T} := \left\{(T - t, x) : (t,x) \in R \right\}.
\end{equation}
To more conveniently apply results from the sections above we will reverse the roles of $X$ and $Y$ for Rost stopping times.

Consider a non-delayed (forward) Rost stopping time $\bar {\rho} = \inf \left\{t \geq 0 : (t, Y_t) \in \bar R \right\}$.
We can now observe as illustrated in Figure \ref{fig:SM-Root-Rost-symmetry} that the (forward) Rost stopping time $\bar{\rho}$ 
can be represented as a backward Root stopping time $\tau := \inf \left\{t \geq 0 : (T-t, Y_t) \in R_T \right\}$, 
and that the backwards Rost stopping time $\sigma := \inf \left\{t \geq 0 : (T-t, X_t) \in \bar R \right\}$ 
can be represented as a (non-delayed, forward) Root stopping time $\rho := \rho_0 = \inf \left\{t \geq 0 : (t, X_t) \in R_T \right\}$, 
to summarize 
\begin{align} 
\begin{split}\label{eq:SM-Root-Rost-symmetry}
    \bar{\rho} = \inf \left\{t \geq 0 : (t, Y_t) \in \bar R \right\} 
        = \inf \left\{t \geq 0 : (T-t, Y_t) \in R_T \right\} = \tau
\\  \sigma = \inf \left\{t \geq 0 : (T-t, X_t) \in \bar R \right\}
        = \inf \left\{t \geq 0 : (t, X_t) \in R_T \right\} = \rho.
\end{split}
\end{align}
\begin{figure}
\centering
\begin{subfigure}{.49\linewidth}
\centering
\includegraphics[width = 1\linewidth]{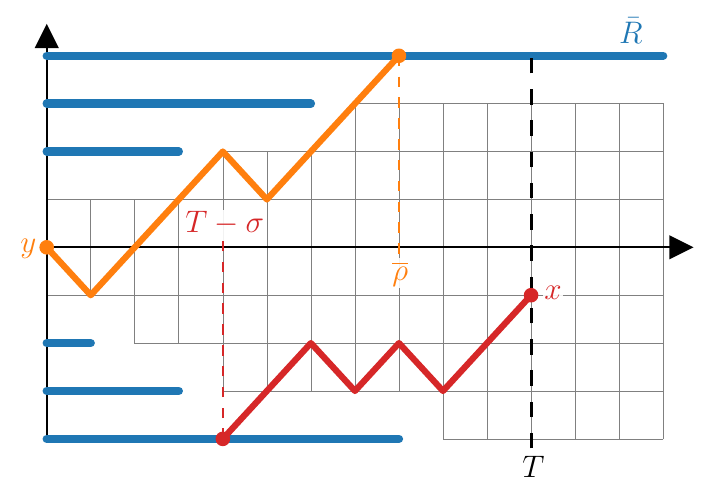}
\end{subfigure}
\begin{subfigure}{.49\linewidth}
\centering
\includegraphics[width = 1\linewidth]{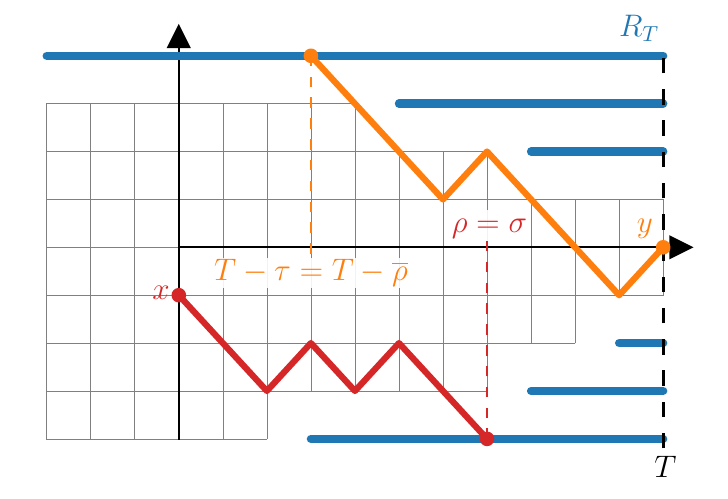}
\end{subfigure}
\caption{An illustration of the connection between $\bar{R}$ and $R_T$, resp. between $\bar{\rho}$ and $\tau$ as well as $\sigma$ and $\rho$ without delay.} 
\label{fig:SM-Root-Rost-symmetry}
\end{figure}
Using this symmetry, in \cite{BaCoGrHu21} the Rost optimal stopping problem was then derived with the methods established for the Root optimal stopping problem using the following proposition.
\begin{proposition}[{\cite[Proposition 3.7]{BaCoGrHu21}}] \label{prp:SM-symmetry}
For each $x,y,T$, any $Y$-stopping time $\tau$ such that $\Ey[|Y_{\tau}|]<\infty$, 
and every $\{0,\dots,T\}$-valued $X$-stopping time $\sigma$, we have
\begin{align} \label{eq:SM-symmetry}
    \Ey\left[|x - Y_{\tau}|-|x - Y_{\tau\wedge T}|\right]&\leq \Exy\left[|X_{\sigma} - Y_{\tau}|-|X_{\sigma}-y| \right].
\end{align}
Suppose furthermore that 
\begin{equation}
    \tau = \inf \{t\in \mathbb{N}: (T-t,Y_t) \in R\}, 
\end{equation} 
for a Root barrier $R$ and that
$\sigma=\rho^{Root}\wedge T$. Then there is equality in \eqref{eq:SM-symmetry}. 
\end{proposition} 

The proof given in \cite{BaCoGrHu21} relied heavily on the interpolation function 
$F_{\sigma, \tau}(s) = F^{x,y}_{\sigma, \tau}(s) = \E^{x}_{y} \left[\left| X_{\sigma \wedge (T - \tau \wedge s)} - Y_{\tau \wedge s} \right| \right]$ 
in the following way. 
To see inequality we observe 
\begin{align*}
    \E_y\left[|x - Y_{\tau}|-|x - Y_{\tau\wedge T}|\right] 
        &\leq \E^x_y \left[|X_{\sigma} - Y_{\tau}| \right] - F_{\sigma, \tau}(T)
\\      &\leq \E^x_y \left[|X_{\sigma} - Y_{\tau}| \right] - F_{\sigma, \tau}(0)
\\      &=    \E^x_y \left[|X_{\sigma} - Y_{\tau}|-|X_{\sigma}-y| \right]. 
\end{align*}
Equality for Root forward and backward stopping times follows from Lemma \ref{lem:interpolation-function} (ii). 
In the delayed regimes however, these arguments cannot be repeated verbatim as Lemma \ref{lem:interpolation-function} (ii) does not hold for delayed backwards Root stopping times. 

We will observe an analogous symmetry to \eqref{eq:SM-Root-Rost-symmetry} in the delayed regime and give the appropriate generalization of the symmetries \eqref{eq:SM-Root-Rost-symmetry}  as well as Proposition \ref{prp:SM-symmetry}. 

\subsubsection*{A delayed randomized Root-Rost-Symmetry}
Let $(\bar r_t(x))_{(t,x) \in \mathbb{N} \times \mathbb{Z}}$ denote a field of Rost stopping probabilities. 
Then for any $T \in \mathbb{N}$ we can transform these Rost stopping probabilities into a field $(r^T_t(x))_{(t,x) \in \mathbb{Z} \times \mathbb{Z}}$ of Root stopping probabilities in the following way
\begin{equation} \label{def:r_T}
    r_t^T(x) := \bar r_{T-t} (x).
\end{equation}
Note that the definitions of Root and Rost fields of stopping probabilities naturally extend to $(t,x) \in \mathbb{Z} \times \mathbb{Z}$. 
Consider a delayed (forward) Rost stopping time 
\[ 
\bar \rho_{\theta} = \inf \left\{t \geq \theta : \bar r_t(Y_t) > v_t \right\}.
\] 
We can now observe as illustrated in Figure \ref{fig:SM-Root-Rost-symmetry} that the (forward) Rost stopping time $\bar{\rho}$ 
can be represented as a delayed backward Root stopping time 
\[
\tau_{\theta} := \inf \left\{t \geq \theta : r_{T-t}^T(Y_t) > v_t \right\},
\] 
and that the (undelayed) backwards Rost stopping time 
\[
\sigma := \inf \left\{t \geq 0 :  \bar r_{T-t}(X_t) > u_t \right\}
\]
can be represented as a (undelayed, forward) Root stopping time 
\[
\rho := \rho_0 = \inf \left\{t \geq 0 : r_{t}^T(X_t) > u_t \right\},
\]
in sum 
\begin{align} 
\label{eq:delayed-Root-Rost-symmetry1}
\bar{\rho}_{\theta} := \inf \left\{t \geq \theta : \bar r_t(Y_t) > v_t \right\} 
                    &= \inf \left\{t \geq \theta : r_{T-t}^T(Y_t) > v_t \right\}    =: \tau_{\theta}
\\  \sigma := \inf \left\{t \geq 0 :  \bar r_{T-t}(X_t) > u_t \right\}
           &= \inf \left\{t \geq 0 : r_{t}^T(X_t) > u_t \right\}                =: \rho. \label{eq:delayed-Root-Rost-symmetry2}
\end{align}
See also Figure \ref{fig:delayed-Root-Rost-symmetry}  for an illustration of these stopping times.
\begin{figure}
\centering
\begin{subfigure}{1\linewidth}
\centering
\includegraphics[width = 0.49\linewidth]{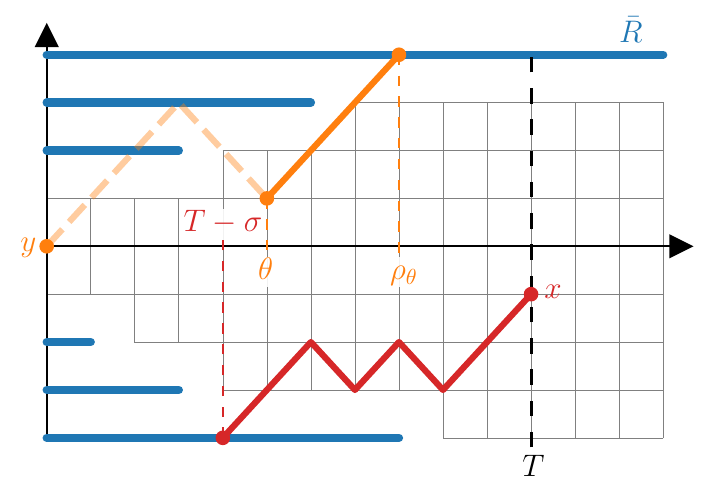}
\includegraphics[width = 0.49\linewidth]{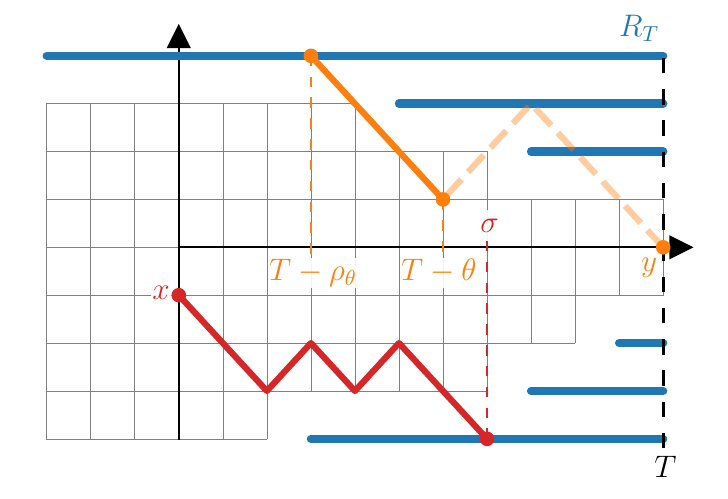}
\caption{An illustration of the connection between $\bar R$ and $R_T$ in the delayed regime for $\theta < T$.}
\end{subfigure}
\centering
\begin{subfigure}{1\linewidth}
\centering
\includegraphics[width = 0.49\linewidth]{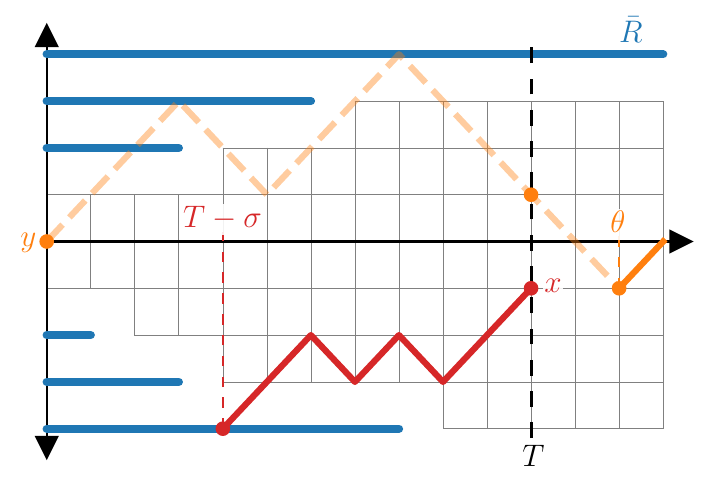}
\includegraphics[width = 0.49\linewidth]{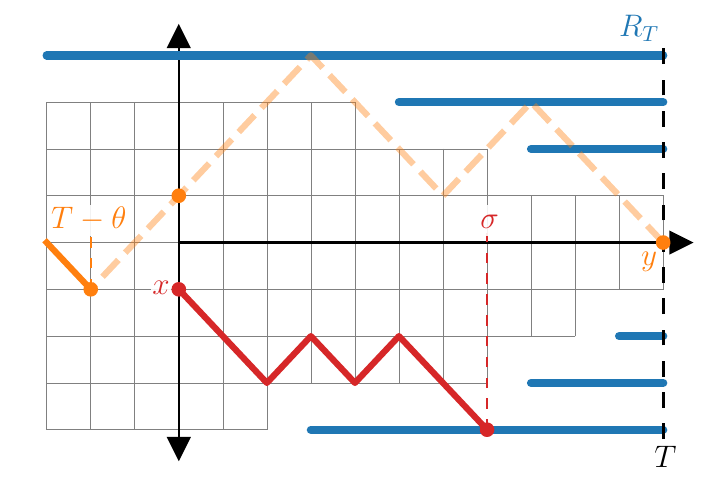}
\caption{An illustration of the connection between $\bar R$ and $R_T$ in the delayed regime for $\theta \geq T$.}
\end{subfigure}
\caption{}
\label{fig:delayed-Root-Rost-symmetry}
\end{figure}

We see that Proposition \ref{prp:SM-symmetry} is valid for \emph{any} stopping times $\sigma$ and $\tau$, 
hence also for the choices $\sigma = \rho_0$ and $\tau = \tau_{\theta}$.
However, in this case is was already observed in Section \ref{sec:interpolation-function} that the interpolation function $F_{\rho_0, \tau_{\theta}}(s)$ 
will in general no longer be constant when $\tau$ is subject to a non-zero delay $\theta$. 
Thus we will no longer recover an equality in Equation \eqref{eq:SM-symmetry} for a \emph{delayed} backward Root stopping time.
Moreover the RHS of Equation \eqref{eq:SM-symmetry} will in general also not equal the RHS of Equation \eqref{eq:discrete-delayed-Rost}, 
thus cannot be applied to prove our desired generalization of the delayed Rost optimal stopping representation.

The following generalization of Proposition \ref{prp:SM-symmetry} is needed.
\begin{proposition} \label{prp:Root-Rost-prp}
Let $\sigma \leq T$ be a stopping time for $X$ and let $\tau$ be a stopping time for $Y$ such that $\E_{\nu}\left[|Y_{\tau}| \right]$. 
Then
\begin{enumerate}[(i)]
\item 
\begin{equation} \label{eq:Root-Rost-prp1}
    \E_{\nu} \left[\left| x - Y_{\tau}\right| - \left| x - Y_{\tau \wedge T}\right|  \right]
    \leq
\E^x_{\nu} \left[\left| X_{\sigma} - Y_{\tau}\right| - \left| X_{\sigma} - Y_{\tau \wedge (T - \sigma)} \right|  \right].
\end{equation}
\item For any $Y$-stopping time $\theta$ such that $\theta \leq \tau$ we have
\begin{equation} \label{eq:Root-Rost-prp2}
    \E_{\nu} \left[\left| x - Y_{\tau}\right| - \left| x - Y_{\tau \wedge T}\right|  \right]
    \leq
\E^x_{\nu} \left[\left| X_{\sigma} - Y_{\tau}\right| - \left| X_{\sigma} - Y_{\theta \wedge (T - \sigma)} \right|  \right].
\end{equation}
\item 
Furthermore, let $(r_t(x))_{(t,x) \in \mathbb{Z} \times \mathbb{Z}}$ denotes a field of Root stopping probabilities, 
consider $(s_t(x))_{(t,x) \in \mathbb{Z} \times \mathbb{Z}} \in \mathcal{R}_r$ 
and let $\theta$ denote a delay stopping time for $Y$. 
Then there is equality in \eqref{eq:Root-Rost-prp1} and \eqref{eq:Root-Rost-prp2} when considering the stopping times 
\begin{align*}
    \sigma  &= \inf \left \{ t \geq 0 : r_t(X_t) > u_t \right\} \,\, \text{ and }
\\  \tau    &= \inf \left \{ t \geq \theta : s_{T-t}(Y_t) > v_t\right\} \wedge T.
\end{align*}
\end{enumerate}
\end{proposition}
\begin{proof}
As observed in the proof of \cite[Proposition 3.7]{BaCoGrHu21} 
we can conclude the following with the help of Jensen's inequality and optional sampling
\begin{equation*} 
    \E_{\nu} \left[\left| x - Y_{\tau}\right| - \left| x - Y_{\tau \wedge T}\right|  \right]
    \leq
    \E^x_{\nu} \left[\left| X_{\sigma} - Y_{\tau}\right| - \left| X_{\sigma \wedge (T - \tau \wedge T)} - Y_{\tau \wedge T} \right|  \right].
\end{equation*}
It is also a straightforward application of the Core Argument \eqref{eq:core} together with optional sampling to obtain 
\begin{align*}
   \E^x_{\nu} \left[\left| X_{\sigma \wedge (T- \tau \wedge T)} - Y_{\tau \wedge T} \right|  \right]
 = \E^x_{\nu} \left[\left| X_{\sigma} - Y_{\tau \wedge (T - \sigma)} \right|  \right]
\end{align*}
which concludes the proof of $(i)$.

Note that $(ii)$ easily follows from $(i)$ via Jensen's inequality and optional sampling by observing that when 
$\theta \leq \tau$ we clearly also have $\theta \wedge (T-\sigma) \leq \tau \wedge (T-\sigma)$.
 
It now remains to show $(iii)$, the equality in the Root case. 
We will apply Lemma \ref{lem:opt-stopping} 
for the choices $\lambda = \delta_{x}$, $\rho = \sigma$ and $\eta = 0$. 
We will consider a decomposition into the events $\{\tau < T\}$ and $\{\tau \geq T\}$. 
First note that on $\{\tau < T\}$ we have $\eta \vee (T-\tau) = T-\tau$, hence by Lemma \ref{lem:opt-stopping} we obtain
\[
    \E^{x}_{\nu} \left[ \left| X_{\sigma} - Y_{\tau} \right| \ind_{\tau < T}\right] =
    \E^{x}_{\nu} \left[\left| X_{\sigma \wedge (T-\tau)} - Y_{\tau} \right| \ind_{\tau < T} \right],
\]
thus
\begin{align*}
    \E_{\nu} \left[\left(\left| x - Y_{\tau}\right| - \left| x - Y_{\tau \wedge T}\right|\right) \ind_{\tau < T} \right] 
     = 0
     &= \E^{x}_{\nu} \left[ \left(\left| X_{\sigma} - Y_{\tau} \right| 
        - \left| X_{\sigma \wedge (T-\tau)} - Y_{\tau} \right| \right) \ind_{\tau < T} \right]
\\   &= \E^{x}_{\nu} \left[ \left(\left| X_{\sigma} - Y_{\tau} \right| 
        - \left| X_{\sigma \wedge (T-\tau \wedge T)} - Y_{\tau \wedge T} \right| \right) \ind_{\tau < T} \right].
\end{align*}
On the other hand, on $\{\tau \geq T\}$ we have $\eta \vee (T-\tau) = 0$, hence by Lemma \ref{lem:opt-stopping}
\begin{align*}
    \E^{x}_{\nu} \left[ \left| X_{\sigma} - Y_{\tau} \right| \ind_{\tau \geq T}\right] 
    = \E^{x}_{\nu} \left[\left| X_{\sigma \wedge 0} - Y_{\tau} \right| \ind_{\tau \geq T} \right] 
    = \E^{x}_{\nu} \left[\left| X_{0} - Y_{\tau} \right| \ind_{\tau \geq T} \right] 
     = \E_{\nu} \left[\left| x - Y_{\tau} \right| \ind_{\tau \geq T} \right]
\end{align*}
which gives
\begin{align*}
    \E_{\nu} \left[\left(\left| x - Y_{\tau}\right| - \left| x - Y_{\tau \wedge T}\right|\right) \ind_{\tau \geq T} \right] 
     &= \E^{x}_{\nu} \left[ \left(\left| X_{\sigma} - Y_{\tau} \right| 
        - \left| X_{0} - Y_{\tau \wedge T} \right| \right) \ind_{\tau \geq T} \right]
\\   &= \E^{x}_{\nu} \left[ \left(\left| X_{\sigma} - Y_{\tau} \right| 
        - \left| X_{\sigma \wedge (T - \tau \wedge T)} - Y_{\tau \wedge T} \right| \right) \ind_{\tau \geq T} \right].
\end{align*}
Altogether equality in \eqref{eq:Root-Rost-prp1} follows. 
Note that $\eta = 0$ was essential here in order to gain this equality.

It now remains to establish equality in \eqref{eq:Root-Rost-prp2}, that is we need to show
\begin{align*}
    \E^x_{\nu} \left[\left| X_{\sigma} - Y_{\tau \wedge (T - \sigma)} \right|  \right]
    = \E^x_{\nu} \left[\left| X_{\sigma} - Y_{\theta \wedge (T - \sigma)} \right|  \right].
\end{align*}
On the one hand we trivially have
\begin{align*}
    \E^x_{\nu} \left[\left| X_{\sigma} - Y_{\tau \wedge (T - \sigma)} \right| \ind_{\sigma = T} \right]
    = \E^x_{\nu} \left[\left| X_{\sigma} - Y_{0} \right| \ind_{\sigma = T}  \right]
    = \E^x_{\nu} \left[\left| X_{\sigma} - Y_{\theta \wedge (T - \sigma)} \right| \ind_{\sigma = T}  \right],
\end{align*}
whereas 
\begin{align*}
    \E^x_{\nu} \left[\left| X_{\sigma} - Y_{\tau \wedge (T - \sigma)} \right| \ind_{\sigma < T} \right]
    = \E^x_{\nu} \left[\left| X_{\sigma} - Y_{\theta \wedge (T - \sigma)} \right| \ind_{\sigma < T}  \right]
\end{align*}
can be argued verbatim as in Lemma \ref{lem:opt-stopping}, switching the roles of the two stopping times therein.
\end{proof}
\begin{remark}
Note that Proposition \ref{prp:SM-symmetry} can easily be recovered from Proposition \ref{prp:Root-Rost-prp} $(ii)$ via the choices $\nu = \delta_y$ and $\theta = 0$.
\end{remark}
Equipped with this proposition and the (delayed) Root-Rost symmetries \eqref{eq:delayed-Root-Rost-symmetry1} and \eqref{eq:delayed-Root-Rost-symmetry2} we can now give a proof of Theorem \ref{thm:discrete-delayed-Rost}, the optimal stopping representation of Rost solutions to \eqref{dSEP}

\begin{proof}[Proof of Theorem \ref{thm:discrete-delayed-Rost}]
We want to apply Proposition \ref{prp:Root-Rost-prp}. 
Consider $\tau = \tau_{\theta}$ and note that for the LHS of \eqref{eq:Root-Rost-prp1} respectively \eqref{eq:Root-Rost-prp2} with the help of the Root-Rost symmetry \eqref{eq:delayed-Root-Rost-symmetry1} we have 
\begin{align*}
    \E_{\nu} \left[ \left| x - Y_{\tau} \right| - \left| x - Y_{\tau \wedge T} \right| \right]
    =    U_{\beta^{\alpha}}(x) - U_{\beta^{\alpha}_T}(x).  
\end{align*}
Furthermore for any $X$-stopping time $\sigma$ we have
\begin{align}
    \E^x_{\nu} \left[\left| X_{\sigma} - Y_{\tau}\right| - \left| X_{\sigma} - Y_{\theta \wedge (T - \sigma)} \right|  \right]
    =   \E^x\left[ U_{\beta^{\alpha}} (X_{\sigma}) - V^{\alpha}_{T - \sigma} (X_{\sigma}) \right].
\end{align}
It now becomes obvious that Equation \eqref{eq:discrete-delayed-Rost} follows by Proposition \ref{prp:Root-Rost-prp} $(iii)$ by considering 
$\sigma = \sigma^*$ and the Root-Rost symmetry \eqref{eq:delayed-Root-Rost-symmetry2}, 
while Equation \eqref{eq:discrete-delayed-RostOSP} follows from Proposition \ref{prp:Root-Rost-prp} $(ii)$ considering general $X$- stopping times $\sigma \leq T$.
\end{proof}
%
%
%
%
%
%
%
%
%
%
%
%
%
%
%
%
%
%
%
%
%
%
%
%
%
%
%
%
%
%
%
%
%
%
%
%
%
%
%
%
%
%
%
%
%
%
%
%
%
%
\section{Recovering the Continuous Optimal Stopping Representation 
as a Limit of the Discrete Optimal Stopping Representation} \label{sec:limit}
%
The aim of this section is to recover Theorem \ref{thm:delayed-Root} (resp. Theorem \ref{thm:delayed-Rost}) 
from its discrete counterpart, Theorem \ref{thm:discrete-delayed-Root} (resp. Theorem \ref{thm:discrete-delayed-Rost}) 
through taking limits of appropriately scaled SSRWs. 

This passage to continuous time essentially relies on the application of Donsker-type results, 
similar to the limiting procedure carried out in \cite[Section 5]{BaCoGrHu21} 
which in turn heavily built on arguments established in \cite{CoKi19b}. 
For this purpose it is important to point out that the results and arguments presented in  Section \ref{sec:discrete} 
remain invariant under appropriate scaling of the space-time grid. 

In Section \ref{sec:discretization} we recall the appropriate space-time grid and an associated discretization  
of a Brownian motion into a scaled SSRW as proposed in \cite{CoKi19b}. 
We proceed to introducing a discretization of a continuous \eqref{dSEP} to a scaled discrete delayed problem 
in Section \ref{sec:(dSEP)-discretization} 
and furthermore explain how our results of Section \ref{sec:discrete} apply in this case. 

It is subject of Section \ref{sec:taking-limits} to recover continuous Root and Rost solutions to \eqref{dSEP} 
as limits of solutions to a appropriately scaled and discretized \eqref{dSEP}. 
This will give us convergence of the LHS of \ref{eq:discrete-delayed-Root} to the LHS of \ref{eq:delayed-Root} 
(resp. convergence of the LHS of \ref{eq:discrete-delayed-Rost} to the LHS of \ref{eq:delayed-Rost}). 

In Section \ref{sec:OSP-convergence} these limiting results are furthermore applied 
to the recovery of the optimal stopping problem and its value, 
namely establishing convergences of the RHS of \ref{eq:discrete-delayed-Root} to the LHS of \ref{eq:discrete-delayed-Root} 
as well as \ref{eq:discrete-delayed-RootOSP} to \ref{eq:discrete-delayed-RootOSP} (analogously for Rost). 
%
%
%
\subsection{Discretization of a Brownian motion and its stopping times}
\label{sec:discretization}
We begin by recalling the discretization of a Brownian motion $(W_t)_{t \geq 0}$ into a scaled SSRW as proposed in \cite{CoKi19b}.
For a given $N \in \mathbb{N}$ consider a discrete set of spatial points 
$\mathcal{X}^N := \{x^N_1, \dots, x^N_{L^N}\}$ such that 
\begin{itemize}
\item[\itembullet] $|x^N_{k+1} - x^N_{k}| = \frac{1}{\sqrt{N}}$ for $j = 1, \dots, L^N-1$.
\item[\itembullet] $L^N \sim \sqrt{N}$.
\end{itemize} 
We define a random walk $(Y^N_k)_{k \in \mathbb{N}}$ in the following way. 
Let $Y^N_k := W_{\tau^N_k}$ where the $\tau^N_k$ are given as those times when a Brownian Motion hits a new grid point in $\mathcal{X}^N$. 
More precisely, let $(W_t)_{t \geq 0}$ be a Brownian motion started in $W_0 \sim \lambda$. 
Then we define
\begin{itemize}
\item[\itembullet] $\tau^N_0 = \inf \{t \geq 0 : W_t \in \mathcal{X}^N\}$, and 
\item[\itembullet] given $W_{\tau^N_k} = x^N_j$ we define $\tau^N_{k+1} := \inf \{t \geq \tau^N_k : W_t \in \{x^N_{j-1}, x^N_{j+1} \} \}$.
\end{itemize} 
Note, this definition implies that the random walk $(Y^N_k)_{k \in \mathbb{N}}$ 
is started according to a discretization $\lambda^N$ of the starting law $\lambda$ given by 
\[ 
    \lambda^N := \mathcal{L}(Y^N_0) = \mathcal{L}^{\lambda}(W_{\tau^N_0}).
\]
Now let $\tau$ be a stopping time for the Brownian Motion $W_t$. 
Then we can define a stopping rule $\tilde \sigma^N$ for the discretized process $Y^N_k$ via
\[
   \tilde \sigma^N := \inf \{t \in \mathbb{N}: \tau^N_{t-1} < \tau \leq \tau^N_t  \}. 
\]
For such a discretized stopping time we have the following convergence results.
\begin{lemma}[Cf.{{\cite[Lemma 5.2]{CoKi19b}}}] \label{lem:limit}
Let $\tilde \sigma^N$ be the discretization of a stopping time $\tau$ of a Brownian motion $W_t$ as defined above. 
Then 
\begin{alignat}{3}
    Y^N_{\tilde \sigma^N}&\rightarrow W_{\tau} \quad &&\text{almost surely, and}\\
    \frac{\tilde \sigma^N}{N}&\rightarrow\tau \quad &&\text{in probability, as } N\rightarrow\infty.
\end{alignat}
In particular, this implies that for $\mu^N := \mathcal{L}(Y^N_{\tilde \sigma^N})$ 
and for $\mu := \mathcal{L}^{\lambda}(W_{\tau})$ we have that $\mu^N \rightarrow \mu$.
Moreover, for every $T \geq 0$ we have
\begin{alignat*}{3}
    Y^N_{\tilde \sigma^N \wedge NT}&\rightarrow W_{\tau \wedge T} \quad &&\text{almost surely, and}\\
    \frac{\tilde \sigma^N \wedge NT}{N}&\rightarrow \tau \wedge T \quad &&\text{in probability, as } N\rightarrow\infty.
\end{alignat*}
\end{lemma}
Note that while in \cite{CoKi19b} this lemma was stated for stopping times $\tau$ that are optimizers of a given {\sf (OptSEP)}, 
this characterization is not essential to the proof given, thus we state the lemma in more generality. 
Given the convergence of $\frac{\tilde \sigma^N}{N} \rightarrow \tau$ in probability, 
the convergence of $\frac{\tilde \sigma^N \wedge NT}{N} \rightarrow \tau \wedge T$ in probability is clear.
Regarding the convergence of the processes, it is important to observe that by construction of the discretization 
we have $\tau^N_{NT} \rightarrow T$ a.s. by the strong law of large numbers. 
Note furthermore that $Y^N_{\tilde \sigma^N \wedge NT} = W_{\tau^N_{\tilde \sigma^N} \wedge \tau^N_{NT}}$, 
hence the almost sure convergence $Y^N_{\tilde \sigma^N \wedge NT} \rightarrow W_{\tau \wedge T}$ is clear. 
%
%
%
%
%
%
%
%
%
%
\subsection[A Suitable Discretization of Root and Rost solutions to \eqref{dSEP}]{A Suitable Discretization of Root and Rost solutions to (dSEP)} \label{sec:(dSEP)-discretization}
Consider a continuous \eqref{dSEP} given by the triple $(\lambda, \eta, \mu)$. 
Theorem \ref{thm:dRoot} (resp. Theorem \ref{thm:dRost}) gives us existence of a Root (resp. Rost) barrier $R$ in $[0, \infty) \times \mathbb{R}$ 
such that
\[
    \rho = \rho_R := \inf \{ t \geq \eta : (t, W_t) \in R \}
\] 
is a stopping time embedding the given measure $\mu$. 
We propose the following discretization $(\lambda^N, \tilde\eta^N, \mu^N)$ of $(\lambda, \eta, \mu)$.
\begin{align*}
    \lambda^N       &:= \mathcal{L}(Y^N_0),
\\  \tilde \eta^N   &:= \inf \{t \in \mathbb{N}: \tau_{t-1}^N < \eta \leq \tau_{t}^N \}  \,\, \text{ and }
\\  \mu^N           &:= \mathcal{L}^{\lambda^N}(Y^N_{\tilde \sigma^N}) \,\, \text{ where }
\\  \tilde \sigma^N &:= \inf \{t \in \mathbb{N}: \tau_{t-1}^N < \rho \leq \tau_{t}^N \}. 
\end{align*}
(Note that since $\rho \geq \eta$, it is clear that we also have $\tilde{\sigma}^N \geq \tilde \eta^N$.) 
If the measure $\mu$ is supported also beyond the grid $\mathcal{X}^N$ given at the discretization step $N$ we make the boundaries of the grid, $x^N_1$ resp. $x^N_{L^N}$ absorbing barriers for the Brownian motion to ensure that the random walk only moves on the given grid. 

The following convergences are now clear by the definitions in Section \ref{sec:discretization} and Lemma \ref{lem:limit} 
\begin{align*}
    \lambda^N &\Rightarrow \lambda
\\  \eta^{(N)} := \frac{\tilde \eta^N}{N} &\xrightarrow{p} \eta
\\    \mu^N   &\Rightarrow \mu.
\end{align*}
Moreover, it is important to observe that the conditions $\lambda^N \leq_c \mathcal{L}^{\lambda^N}\left( Y^N_{\tilde \eta^N} \right) \leq_c \mu^N$ and $\E \left[\tilde \eta^N \right] < \infty$ are satisfied. 
Therefore, by Theorem \ref{thm:ddSEP-Root} (resp. Theorem \ref{thm:ddSEP-Rost}) 
and rescaling there exists a field of Root (resp. Rost) stopping probabilities $(r^N_t(x))_{(t,x) \in \mathbb{N} \times \mathcal{X}^N}$ 
where $\mathcal{X}^N \subseteq \frac{1}{\sqrt{N}} \mathbb{Z}$ such that 
\[
    \tilde \rho^N := \inf\left\{ t \geq \tilde \eta^N : r^N_t(Y^N_{t}) > u_t \right\}
\]
defines a discrete Root (resp. Rost) stopping time embedding the measure $\mu^N$ into the SSRW $Y^N$, 
solving the discrete (and rescaled) \eqref{dSEP} given by $(\lambda^N, \tilde\eta^N, \mu^N)$. 

By $\left(W_t^{(N)}\right)_{t \geq 0}$ we denote the rescaled continuous version of the random walk $Y^N$ defined via 
\[
    W_t^{(N)} :=Y_{\lfloor Nt \rfloor}^N + \left(Nt - \lfloor Nt \rfloor \right) \left(Y_{\lfloor Nt \rfloor + 1}^N - Y_{\lfloor Nt \rfloor}^N \right).
\]
Note that we have $\left(\frac{\tilde \eta^N}{N}, Y^{N}_{\tilde \eta^N} \right) = \left(\eta^{(N)}, W_{\eta^{(N)}}^{(N)} \right)$ almost surely. 

We now want to define Root and Rost stopping times for $\left(W_t^{(N)}\right)_{t \geq 0}$.
By definition of the discretization we have $W^{(N)}_t \in \mathcal{X}^N$ if and only if $Nt \in \mathbb{N}$. 
Hence we can extend a Root field of stopping probabilities $(r^N_t(x))_{(t,x) \in \mathbb{N} \times \mathcal{X}^N}$ to $[0, \infty) \times \mathbb{R}$ by setting
\[
    r^N_t(x) := \begin{cases}
        r^N_{\lfloor t \rfloor}(x) &\text{ for } x \in \mathcal{X}^N, 
\\      0 &\text{ for all } t \text{ when } x \in \mathbb{R} \setminus \mathcal{X}^N.
\end{cases}
\]
Then we  define
\begin{equation} \label{eq:rho-N-def}
    \rho^{(N)} := \inf\left\{t \geq \eta^{(N)} : r^N_{ \lfloor Nt \rfloor}\left(W^{(N)}_t \right) > u_{\lfloor Nt \rfloor } \right\}
\end{equation}
in order to have
\[
    \left(\tilde \rho^N, Y^N_{\tilde \rho^N} \right) = \left(\rho^{(N)} , W^{(N)}_{\rho^{(N)}}\right) \, \text{ almost surely}.
\]
Furthermore $(r^N_t(x))_{(t,x) \in \mathbb{N} \times \mathcal{X}^N}$ induce a Root barrier $\tilde R^N \subseteq \mathbb{N} \times \frac{1}{\sqrt{N}} \mathbb{Z}$ via
\[
    \tilde R^N := \left\{(t,x) : r^N_t(x) > 0 \right\},
\]
and a corresponding Root barrier in $[0, \infty) \times \mathbb{R}$ can be defined in the following way
\begin{equation} \label{eq:R^N-def}
  R^{(N)} := \left\{ (t,x) \in [0, \infty) \times \mathbb{R} :   \left( \lfloor Nt \rfloor,x \right) \in \tilde R^N \right\}.
\end{equation}
We will sometimes use the notation $\tilde R^N =: \tilde R^N_+$ (resp. $R^{(N)} =: R^{(N)}_+$) and 
analogously define $\tilde R^N_-$ (resp. $R^{(N)}_-$) via
\[
    \tilde R^N_- := \left\{(t,x) : r^N_t(x) = 1 \right\}.
\]
Note that $\tilde R^N_- \subseteq \tilde R^N_+$ and $R^{(N)}_- \subseteq R^{(N)}_+$.
We make the same definitions for Rost fields of stopping probabilities replacing the floor function by a ceiling function. 
%
%
%
%
%
%
%
%
%
%
\subsection{Taking Limits} \label{sec:taking-limits}
The main objective of this section is to show the following convergence of the discrete time objects defined above to their continuous conterparts. 
This will be carried out individually for Root and Rost respectively. 

We follow \cite{Ro69} for the definition of a metric on the space of barriers. 
Let $\mathscr{R}$ denote the space of all Root and Rost barriers in $[0, \infty] \times [-\infty, \infty]$. 
Then a complete metric on this space can be defined in the following way. 
Consider the map $f: [0, \infty] \times [-\infty, \infty]  \rightarrow [0,1] \times [-1,1]$, $f(t,x) := \left(\frac{t}{1+t}, \frac{x}{1+|x|} \right)$. 
Let $d$ denote the ordinary Euclidean metric on $\mathbb{R}$. 
Then for two barriers $R,S \subseteq [0, \infty] \times [-\infty, \infty]$ \emph{Root's metric} $d_R$ on the space of all barriers $\mathscr{R}$ is defined as follows
\[
    d_R(R,S) := \max \left\{ \sup_{(t,x) \in R} d\left(f(t,x), f(S) \right), \sup_{(s,y) \in S} d\left(f(R), f(s,y) \right) \right\}.
\]
%
%
%
%
\subsubsection{A Limit of Root Stopping Times}

\begin{lemma} \label{lem:main-convergence-Root}
Let $\left(\rho^{(N)}\right)_{n \in \mathbb{N}}$ be a sequence of discrete Root stopping times as defined in \eqref{eq:rho-N-def}. 
Then there exists a Root barrier $R \subseteq [0, \infty) \times \mathbb{R}$ such that we have the following convergence
\begin{equation} \label{eq:main-convergence-Root}
     \left( \rho^{(N)}, W^{(N)}_{\rho^{(N)}} \right) \xrightarrow{d} \left( \rho, W_{\rho} \right)
\end{equation}
where
\[
    \rho := \inf\left\{t \geq \eta : (t, W_t)   \in R \right\}.
\]
\end{lemma}
\begin{proof}
The Root barriers $(R^{(N)})_{N \in \mathbb{N}}$ converge (possibly along a subsequence) to a Root barrier $R$ in Root's metric, 
details can be found in \cite{Ro69}.

We define the following auxiliary hitting time
\begin{align*}
    \rho^N     &:= \inf\left\{t \geq \eta : (t, W_t) \in R^{(N)}\right\}
\end{align*}
and show convergence in two steps.
\begin{enumerate}[(i)]
\item We have $\rho^N \xrightarrow{p} \rho$. 
This is basically a consequence of a delayed version of Root's original convergence lemma, \cite[Lemma 2.4]{Ro69}, which we provide in Lemma \ref{lem:delayed-original-Root}.
\item We furthermore have
\[
    \left| \left(\rho^{(N)}, W^{(N)}_{\rho^{(N)}}\right) - \left(\rho^N, W_{\rho^N}\right)\right| \xrightarrow{p} 0,
\]
which is a delayed version of \cite[Lemma 5.6]{CoKi19b} which we give in Lemma \ref{lem:5.6-Root}. 
\end{enumerate}
\end{proof}
%
%
%
%
%
%
%
We formulate a delayed version of Root's convergence lemma, \cite[Lemma 2.4]{Ro69}
\begin{lemma} \label{lem:delayed-original-Root}
Let $\eta$ be a delay stopping time and $R$ be a Root barrier such that for the corresponding delayed hitting time  
\[
    \tau = \inf \left\{t \geq \eta : (t, W_t) \in R \right\}
\]
we have $\E[\tau]<\infty$.

Then for any $\varepsilon > 0$ there exist $\delta, \tilde \delta > 0$ 
such that if for another Root barrier $\bar R$ we have $d_R(R, \bar R)< \delta$ 
and for another delay stopping time $\bar \eta$ we have $\mathbb{P}[|\eta - \bar{\eta}|>\tilde \delta] < \tilde \delta$, 
then for the corresponding delayed hitting time $\bar \tau$ defined as
\[
    \bar \tau = \inf \left\{t \geq \bar{\eta} : (t, W_t) \in \bar R \right\},
\]
it holds that
\[
    \mathbb{P}[\bar \tau > \tau + \varepsilon] < \varepsilon.
\]
\end{lemma}
\begin{proof}
Similar to \cite{Ro69} we make the following choices for $\epsilon > 0$.
\begin{itemize}
\item[\itembullet] Choose $\tilde \varepsilon > 0$ such that $\tilde \varepsilon < \frac{\epsilon}{4}$ and
\[
    \mathbb{P}\left[\sup_{\tilde\varepsilon < t < \varepsilon} W_t > \tilde \varepsilon \text{ and }  
    \inf_{\tilde\varepsilon < t < \varepsilon} W_t < - \tilde \varepsilon
\right] > 1 -\frac{\varepsilon}{4}.
\]
\item[\itembullet] Choose $T > \frac{3 \E[\tau]}{\varepsilon}$, then $\mathbb{P}[\tau \geq T] < \frac{\varepsilon}{4}$ (due to Markov's inequality).
\item[\itembullet] Choose $M$ and $\delta >0$ such that if $(t,x) \in R$, $t \leq T$, $|x| \leq M$ and $d_R(R, \bar R)< \delta$, then $d((t,x), \bar R) \leq \tilde\varepsilon$. 
\item[\itembullet] Choose $\tilde \delta = \tilde \varepsilon$, then $\mathbb{P}[|\eta - \bar{\eta}|>\tilde \varepsilon] < \tilde \varepsilon$.
\end{itemize}
Then we consider the set
\begin{align*}
A := \{ \omega \in \Omega : 
     (i)   &    \sup_{\tau(\omega) + \tilde\varepsilon < t < \tau(\omega) + \varepsilon} (W_t(\omega) - W_{\tau}(\omega)) > \tilde \varepsilon \text{ and }  
\\         & \quad \inf_{\tau(\omega) + \tilde\varepsilon < t < \tau(\omega) + \varepsilon} (W_t(\omega) - W_{\tau}(\omega)) < - \tilde \varepsilon,
\\   (ii)  & \;\; \tau(\omega) < T,
\\   (iii) & \;\; |W_{\tau}(\omega)| < M,
\\   (vi)  & \;\; \tilde \eta < \eta + \tilde \varepsilon. \}
\end{align*}
We then see from the definitions made above that for $\omega \in A$ follows $\bar \tau(\omega) < \tau(\omega) + \epsilon$. 
Moreover, these choices and the strong Markov property give us $\mathbb{P}[A] > 1-\varepsilon$.
\end{proof}
%
%
%
%
%
%
In order to complete step (ii) in the proof of Lemma \ref{lem:main-convergence-Root}, 
we need to provide the following auxiliary convergence result, 
which is basically a delayed version of \cite[Lemma 5.7]{CoKi19b}.
\begin{lemma} \label{lem:5.7}
Consider the stopping times
\begin{align*}
    \rho^{(M,N)} &:= \inf \left\{t \geq \eta^{(N)} : \left(t, W^{(N)}_t\right) \in R^{(M)} \right\},
\\  \rho^{M}     &:= \inf \left\{t \geq \eta : \left(t, W_t \right) \in R^{(M)} \right\},
\end{align*}
where $\eta^{(N)} \rightarrow \eta$ in probability.
Then we have the following convergence
\begin{equation} \label{eq:lem:5.7-conv}
    \left(W^{(N)}_{\rho^{(M,N)}}, \rho^{(M,N)} \right) \xrightarrow[N \rightarrow \infty]{d} \left(W_{\rho^M}, \rho^M\right).
\end{equation}
\end{lemma}

\begin{proof}
We aim to apply Donsker's theorem to the continuous scaled symmetric random walk $\left(W_t^{(N)}\right)_{t \in \left[\eta^{(N)}, T \right]}$. 
Let $\left(B_t\right)_{t \in \left[\eta, T \right]}$ denote a Brownian motion, 
then for each $\omega \in \Omega$ we have 
$\left(W_t^{(N)}\right)_{t \in \left[\eta^{(N)}, T \right]}, \left(B_t\right)_{t \in \left[\eta, T \right]} \in \mathcal{D}_T$ where 
$\mathcal{D}_T := \bigcup_{a \leq T} C([a,T])$.

We define an appropriate metric on $\mathcal{D}_T$. 
More generally, consider the function space $\mathcal{D}:= \bigcup_{a \leq b} C([a,b])$. 
Thus for each $f \in \mathcal{D}$ there exist $l_f, r_f \in \mathbb{R}$, $l_f \leq r_f$ such that $f \in C([l_f,r_f])$. 
For $f \in \mathcal{D}$ and $t \geq 0$ we consider $\overline f (t) := f(l_f \vee (t \wedge r_f)) \in C([0, \infty))$ and define the following distance on $\mathcal{D}$.
\begin{align*}
    d_{\mathcal{D}}(f,g) 
            &:=  \left( \sup_{t \in [0, \infty)} |\overline f(t) -  \overline g(t)| \right) \vee |l_f - l_g| \vee |r_f - r_g|.
\end{align*}
Considering this, metric we can apply Donsker's theorem to get
\[
\left(W_t^{(N)}\right)_{t \in \left[\eta^{(N)}, T \right]} \Rightarrow \left(B_t\right)_{t \in \left[\eta, T \right]}
\]
and we are able to use the Portmanteau Theorem in the following way
\begin{align*}
      &\lim_{N \rightarrow \infty} \mathbb{P}\left[\left(W_t^{(N)}\right)_{t \in \left[\eta^{(N)}, T \right]} \in \mathcal{K} \right] 
        =  \mathbb{P}\left[\left(B_t\right)_{t \in \left[\eta, T \right]} \in \mathcal{K} \right] 
\\    & \quad \text{ for all } \mathcal{K} \subseteq \mathcal{D} \text{ Borel}
      \text{ with } \mathbb{P}^{\lambda}\left[\left(B_t\right)_{t \in \left[\eta, T \right]} \in \partial \mathcal{K} \right] = 0.
\end{align*}
It is now left to consider cleverly chosen sets $\mathcal{K} \in \mathcal{B}(\mathcal{D})$ 
to support our convergence claim \eqref{eq:lem:5.7-conv}.
For $M$ fixed consider $W^{(N)}_{\rho^{(M,N)}}$ and $B_{\rho^M}$ 
as well as the spatial grid $\mathcal{X} := \{x^M_1, \dots, x^M_{L^M}\}$. 

Choose an arbitrary point $x_i^M \in \mathcal{X}$, $t \in (0,T]$ such that
$(t, x_i^M)$ lies somewhere in the interior of $R^{(M)}$ and some $\gamma < \left(t - \inf\left\{s: \left(s, x_i^M\right) \in R^{(M)}\right\}\right) \wedge (T-t)$ 
to consider the set
\[
    \bar{R}(\gamma) := [t - \gamma, t + \gamma] \times \left\{x_i^M \right\} \subseteq R^{(M)}.
\]
Since the barrier $R^{(M)}$ will be a subset of $[0,\infty) \times \mathcal{X}$ we can identify a smallest $z \in \mathcal{X}$ such that 
$z > x_i^M$ and $\left([t - \gamma, t + \gamma]  \times \{z\}  \right) \subseteq R^{(M)}$. 
Analogously, we identify a largest $y \in \mathcal{X}$, $y < x_i^M$. 

Moreover, let $y < w_1 < \cdots < w_n < z$,  $w_i \in \mathcal{X} \setminus \left\{ x^M_i \right\}$ 
denote all those points at which a piece of the barrier $R^{(M)}$ sticks into our area of interest,
precisely $\left([t - \gamma, t + \gamma]  \times \{w_i\}  \right) \cap R^{(M)} \neq \emptyset$. 
It is clear that there can only be finitely may such points. 
Consider $\bar{R}_{w_i} := \left( [t - \gamma, t + \gamma] \times \left\{w_i\right\} \right) \cap R^{(M)}$ for $i = 1, \dots, n$. 

We proceed to define the following subsets of $\mathcal{D}_T$.
\begin{align*}
    \hat{\mathcal{K}}
    &:= \bigcup_{\substack{\hat \gamma > 0 \\ \hat \gamma \in \mathbb{Q} }} 
    \left\{f \in \mathcal{D}_T: \| (s, f(s)) - R^{(M)}\| > \hat \gamma \, \, \forall \, s \in [l_f, (t - \gamma) \vee l_f] \right\}
\\  \hat{\mathcal{K}}^{w}
        &:= \bigcap_{i = 1, \dots ,n}
            \bigcup_{\substack{\hat \gamma > 0 \\ \hat \gamma \in \mathbb{Q} }}
            \left\{f \in \mathcal{D}_T: \| (s, f(s)) - \bar R_{w_i}\| > \hat \gamma \, \, \forall \, s \in [(t - \gamma) \vee l_f, t + \gamma] \right\}
\\  \hat{\mathcal{K}}_z &:= \{f \in \hat{\mathcal{K}} \cap \hat{\mathcal{K}}^{w}: 
        f((t - \gamma) \vee l_f) \in (x_i^M,z) \}
\\  \mathcal{K}^{\epsilon}_q &:= \{ f \in \mathcal{D}_T: f(s) < z - \epsilon \,\, \forall \, s \in [(t - \gamma) \vee l_f, q] \cap \mathbb{Q} \} 
\\  \mathcal{K}^{\epsilon, \delta}_q &:= \{ f \in \mathcal{D}_T: f(q) < x_i^M - \delta \} \cap \mathcal{K}^{\epsilon}_q 
\\  \mathcal{K}^{\epsilon, \delta} &:= \bigcup_{\substack{ q \in [t - \gamma, t + \gamma] \\ q \in \mathbb{Q} }} \mathcal{K}^{\epsilon, \delta}_q
\\  \mathcal{K}^{z} &:= \bigcap_{\substack{\delta > 0 \\ \delta \in \mathbb{Q} }}
                        \bigcup_{\substack{\epsilon > 0 \\ \epsilon \in \mathbb{Q} }} \mathcal{K}^{\epsilon, \delta} \cap \hat{\mathcal{K}}_z
\end{align*}
Then
\begin{itemize}
\item[{$\hat{\mathcal{K}}$}] is the set of all paths 
which stay away from the barrier $R^{(M)}$ at any time point before $(t - \gamma) \vee l_f$.
Since $R^{(M)}$ is a finite barrier, $\hat{\mathcal{K}}$ is a Borel set. 
\item[{$\hat{\mathcal{K}}^w$}] is the set of all paths 
which stay away from the barrier pieces $\bar{R}_{w_1}, \dots, \bar{R}_{w_n}$ before time $t+\gamma$.
\item[{$\hat{\mathcal{K}}_z$}] is the set of all paths in $\hat{\mathcal{K}} \cap \hat{\mathcal{K}}^w$ 
that are between $x_i^M$ and $z$ at time $(t - \gamma) \vee l_f$. 
\item[{$\mathcal{K}^{\epsilon}_q$}] is the set of all paths that stay below $z - \epsilon$ after time $(t - \gamma) \vee l_f$ and before time $q$.
\item[{$\mathcal{K}^{\epsilon, \delta}_q$}] is the set of all paths that are below $x_i^M - \delta$ at time $q$ 
and stay below $z - \epsilon$ after time $(t - \gamma) \vee l_f$ and before time $q$. 
\item[{$\mathcal{K}^{\epsilon, \delta}$}] is the set of all paths that 
such that there exist a $q \in [t - \gamma, t + \gamma] \cap \mathbb{Q}$  fulfilling the properties of set $\mathcal{K}^{\epsilon, \delta}_q$.
\item[{$\mathcal{K}^{z}$}] is the set of all paths such that for all $\delta > 0$ there exists an $\epsilon > 0$ such that the properties of set $\mathcal{K}^{\epsilon, \delta}$ are fulfilled.
\end{itemize}
Analogously, we define $\mathcal{K}^{y}$ with the opposite inequalities.

Note that $\mathcal{K}^{z}$ will be a Borel set consisting of all those paths 
that start somewhere away from the barrier $R^{(M)}$ 
but hit the barrier for the first time in $\bar{R}(\gamma)$ coming from above, 
while $\mathcal{K}^{y}$ will be a Borel set of those paths 
that hit the barrier $R^{(M)}$ for the first time in $\bar{R}(\gamma)$ coming from below. 
Altogether, $\mathcal{K} := \mathcal{K}^{y} \cup \mathcal{K}^{z}$ will be a Borel set consisting 
of exactly those paths that hit the barrier $R^{(M)}$ for the first time in $\bar{R}(\gamma)$ 
after starting in some initial time $l_f$ away from the barrier $R^{(M)}$.

We need to identify $\partial\mathcal{K}$ and to argue that $\mathbb{P}^{\lambda}\left[\left(B_t\right)_{t \in \left[\eta, T \right]} \in \partial \mathcal{K} \right] = 0$.
The set $\partial\mathcal{K}$ will consist of exactly those paths that start somewhere away from the boundary
and either
\begin{itemize}
\item[\itembullet] hit $\bar R(\gamma)$ exactly at time $t \pm \gamma$, or
\item[\itembullet] hit $\bar R(\gamma)$ elsewhere, but also \emph{touch} the barrier $R^{(M)}$ before, without passing \emph{through} the barrier. 
\item[\itembullet] are born somewhere inside the barrier $R^{(M)}$ but immediately leave it without ever touching the barrier again.
\item[\itembullet] are born somewhere inside the barrier and remain flat. 
\end{itemize}
By standard properties of a Brownian motion it is clear that these events must have zero probability. 

With appropriate choices of $x_i^N, t$ and $\gamma$ it now becomes clear 
that we can cover the whole barrier $R^{(M)}$ with such sets $\bar R (\gamma)$. 

We are left to consider the possibility of \emph{appearing} exactly \emph{on} the barrier. 
As a combination of the convergence  of $\eta^{N}$ to $\eta$ and Donsker's theorem gives (should give?)
\[
    \left(\eta^{(N)},  W^{(N)}_{\eta^{(N)}}\right) \xrightarrow[N \rightarrow \infty]{d} \left(\eta,  B_{\eta} \right).
\]
Now note that similarly to the arguments given in the proof of Lemma \ref{lem:delayed-original-Root}, 
a Brownian motion that is very close to a Root barrier will proceed to hit this barrier almost immediately.  
It is then a consequence of this regularity of Root barriers which gives us 
\[
    \mathbb{P}^{\lambda^N} \left[ \left(\eta^{(N)},  W^{(N)}_{\eta^{(N)}}\right) \in R^{(M)} \right] 
        \xrightarrow[N \rightarrow \infty]{} 
    \mathbb{P}^{\lambda} \left[ \left(\eta,  B_{\eta} \right) \in R^{(M)} \right],
\]
concluding the proof of the convergence \eqref{eq:lem:5.7-conv}. 
\end{proof}
%
%
%
%
%
%
%
We are now ready to finish step (ii) in the proof of Lemma \ref{lem:main-convergence-Root}.
\begin{lemma}[Cf.{{\cite[Lemma 5.6]{CoKi19b}}}] \label{lem:5.6-Root}
We have the following convergence
\[
    \left| \left(\rho^{(N)}, W^{(N)}_{\rho^{(N)}}\right) - \left(\rho^N, W_{\rho^N}\right)\right|  
        \xrightarrow{p} 0 \text{ for } N \rightarrow \infty.
\]
\end{lemma}
\begin{proof}
Define the following hitting times for the scaled continuous random walk $W^{(N)}$ 
\begin{align*}
    \tau^{(M,N)} &:= \inf\left\{t \geq \eta^{(N)} : \left(t, W^{(N)}_t \right) \in R^{(M)} \right\}, \,\, \text{ and }
\\  \tau^{(N)}   &:= \inf\left\{t \geq \eta^{(N)} : \left(t, W^{(N)}_t \right) \in R^{(N)} \right\}. 
\end{align*}
Then analogous to the proof of \cite[Lemma 5.6]{CoKi19b}, 
but replacing Root's original convergence lemma with its delayed version Lemma \ref{lem:delayed-original-Root}, 
we can deduce the convergence
\[
    \left| \left(\tau^{(M,N)}, W^{(N)}_{\tau^{(M,N)}}\right) - \left(\tau^{(N)}, W^{(N)}_{\tau^{(N)}}\right) \right|  
        \xrightarrow[M,N \rightarrow \infty]{p} 0.
\]
It is a consequence of Lemma \ref{lem:5.7} that
\[
    \left(\tau^{(M,N)}, W^{(N)}_{\tau^{(M,N)}}\right) \xrightarrow[N \rightarrow \infty]{d} 
    \left(\rho^M, W_{\rho^M}\right).
\]
Thus it remains to show that 
\begin{equation} \label{eq:Root-conv-z}
    \left| \left(\rho^{(N)}, W^{(N)}_{\rho^{(N)}}\right) - \left(\tau^{(N)}, W^{(N)}_{\tau^{(N)}}\right) \right|  
        \xrightarrow[N \rightarrow \infty]{p} 0.
\end{equation}
For this purpose recall the Root barriers $R^{(N)}_+$ and $R^{(N)}_-$ defined above and consider 
\begin{align*}
    \tau^{(N)}_+ &:= \inf\left\{t \geq \eta : \left(t, W^{(N)}_t \right) \in R^{(N)}_+ \right\} = \tau^{(N)},
\\  \tau^{(N)}_- &:= \inf\left\{t \geq \eta : \left(t, W^{(N)}_t \right) \in R^{(N)}_- \right\}.
\end{align*}
Note that by Root's barrier structure and the definition of the discretization we have 
\begin{equation*} 
    d_R\left( R^{(N)}_+ , R^{(N)}_- \right) \leq \frac{1}{N} \rightarrow 0 \,\,\text{ for } N \rightarrow \infty,
\end{equation*}
hence both $R^{(N)}_+$ and $R^{(N)}_-$ converge to the same Root barrier $R$ in Root's metric. 
Moreover, it is a consequence of Lemma \ref{lem:delayed-original-Root} together with Donsker's Theorem that
\[
    \left| \tau^{(N)}_+ - \tau^{(N)}_- \right|  \xrightarrow[N \rightarrow \infty]{p} 0.
\]
Now since $\tau^{(N)}_+ \leq \rho^{(N)} \leq \tau^{(N)}_-$ the convergence \eqref{eq:Root-conv-z} follows, concluding the proof.
\end{proof}
%
%
%
%
%
%
%
%
%
%
%
\subsubsection{A Limit of Rost Stopping Times}
%
We are left to show convergence of discrete Rost solutions of \eqref{dSEP} to continuous ones. 
For a measurable set $A$ and $t \geq 0$ recall the following measures defined in Section \ref{sec:delayed}
\begin{align*}
    \nu_t(A)      &= \mu(A) - \mathbb{P}^{\lambda} \left[ W_{\rho} \in A, \rho < t \right] \,\, \text{ and }
\\  \alpha_{t}(A) &= \alpha\left(\{t\} \times A \right) = \mathbb{P}^{\lambda} \left[ W_{t} \in A, \eta = t \right].
\end{align*}
\begin{lemma} \label{lem:main-convergence-Rost}
Let $\left(\rho^{(N)}\right)_{n \in \mathbb{N}}$ be a sequence of discrete Rost stopping times. 
Then there exists a Rost barrier $\bar R \subseteq [0, \infty) \times \mathbb{R}$ such that we have the following convergence
\begin{equation} \label{eq:main-convergence-Rost}
     \left( \rho^{(N)}, W^{(N)}_{\rho^{(N)}} \right) \xrightarrow{d} \left( \rho, W_{\rho} \right)
\end{equation}
where
\[
    \rho := \inf\left\{t \geq \eta : (t, W_t)   \in \bar R \right\} \wedge Z 
\]
for 
\[
    Z = \min_{k \in \mathbb{N}} Z_k , \quad
    Z_k := 
        \begin{cases}
            t_k  &\text{ with probability }\, \mathbb{P}^{\lambda} 
                \left[Z = t_k \big{|} W_{\eta}, \eta = t_k \right] 
                    = \frac{\mathrm{d}\left( \alpha_{t_k} \wedge \nu_{t_k} \right) }{ \mathrm{d}\left( \alpha_{t_k} \right) } \left( W_{\eta} \right)
        \\  \infty & \text{ else. }
      \end{cases}
\]
(Here $\{ t_0, t_1, \dots \}$ denote the set of (at most countably many) atoms  of $\mathcal{L}^{\lambda}(\eta)$.)
\end{lemma}
\begin{proof}
Let $\bar R^{(N)}$ denote the Rost barrier associated to $\rho^{(N)}$ via Definition \eqref{eq:R^N-def}. 
The the sequence $(\bar R^{(N)})_{N \in \mathbb{N}}$ of these Rost barriers converges 
(possibly along a subsequence) to a Rost barrier $\bar R$ in Root's metric. 
 
Consider the following decomposition
\begin{align*}
    \left| \left(\rho^{(N)}, W^{(N)}_{\rho^{(N)}} \right) - \left(\rho, W_{\rho} \right) \right|
        &\leq \left| \left(\rho^{(N)}, W^{(N)}_{\rho^{(N)}} \right) \ind_{\rho^{(N)} > \eta^{(N)} + \varepsilon} - 
           \left(\rho, W_{\rho} \right) \ind_{\rho > \eta} \right| 
\\      &+ \left| \left(\rho^{(N)}, W^{(N)}_{\rho^{(N)}} \right) \ind_{\rho^{(N)} \leq \eta^{(N)} + \varepsilon} - 
           \left(\rho, W_{\rho} \right) \ind_{\rho = \eta} \right|, 
\end{align*}
hence \eqref{eq:main-convergence-Rost} is equivalent to both 
\begin{align}
\left| \left(\rho^{(N)}, W^{(N)}_{\rho^{(N)}} \right) \ind_{\rho^{(N)} > \eta^{(N)} + \varepsilon} - 
           \left(\rho, W_{\rho} \right) \ind_{\rho > \eta} \right| 
     & \xrightarrow[\varepsilon \searrow 0, N \rightarrow \infty]{p} 0 \quad \text{ and }                \label{eq:conv-A}
\\   \left| \left(\rho^{(N)}, W^{(N)}_{\rho^{(N)}} \right) \ind_{\rho^{(N)} \leq \eta^{(N)} + \varepsilon} - 
           \left(\rho, W_{\rho} \right) \ind_{\rho = \eta} \right| 
    & \xrightarrow[\varepsilon \searrow 0, N \rightarrow \infty]{p} 0.                                   \label{eq:conv-B}
\end{align}
Let us first consider \eqref{eq:conv-A}. 
Define the hitting time
\[
    \tau := \inf\left\{t \geq \eta : (t, W_t)   \in R \right\},
\]
thus $\rho = \tau \wedge Z$ and on $\{\rho > \eta\}$ we have $\rho = \tau$. 
Furthermore, for $\bar R^{(N)}$ 
consider the following (auxiliary) hitting times
\begin{align*}
    \tau^N     &:= \inf\left\{t \geq \eta       : (t, W_t) \in \bar R^{(N)}\right\} \, \text{ and }
\\  \tau^{(N)} &:= \inf\left\{t \geq \eta^{(N)} : (t, W^{(N)}_t) \in \bar R^{(N)} \right\}.
\end{align*}
We will then show convergence in three steps. 
\begin{enumerate}[(i)]
\item $\tau^N \xrightarrow{p} \tau$, see Lemma \ref{lem:5.5-Rost}.
\item Convergence for non-randomized hitting times, precisely
\[
    \left| \left(\tau^{(N)}, W^{(N)}_{\tau^{(N)}}\right) - \left(\tau^N, W_{\tau^N}\right)\right| \xrightarrow{d} 0,
\]
see Lemma \ref{lem:5.6}. 
\item Convergence for randomized Rost stopping times after $\eta^{(N)}$, precisely 
\[
 \left|\left(\rho^{(N)}, W^{(N)}_{\rho^{(N)}} \right)\ind_{\rho^{(N)} > \eta^{(N)} + \varepsilon} 
    - \left( \tau^N, W_{\tau^N}\right) \ind_{\tau^N > \eta} \right|
    \xrightarrow[\varepsilon \searrow 0, N \rightarrow \infty]{p} 0,
\]
see Lemma \ref{lem:5.6-2}.
\end{enumerate}
It remains to show \eqref{eq:conv-B}.
For $\varepsilon > 0$ define 
\begin{align*}
    \hat \mu^N_{\varepsilon} &:= \mathcal{L}^{\lambda^N} \left(W^{(N)}_{\rho^{(N)}} ; \rho^{(N)} > \eta^{(N)} + \varepsilon \right)
\\  \overline \mu^N_{\varepsilon} &:= \mu^N - \hat \mu^N_{\varepsilon}
\\ \text{ as well as} 
\\    \hat \mu      &:= \mathcal{L}^{\lambda} \left(W_{\rho} ; \rho > \eta\right) \,\,\text{ and}
\\  \overline \mu &:= \mu - \hat \mu.
\end{align*}

Then by \eqref{eq:conv-A} we have 
\[
\hat \mu^N_{\varepsilon} = \mathcal{L}^{\lambda^N} \left(W^{(N)}_{\rho^{(N)}} ; \rho^{(N)} > \eta^{(N)} + \varepsilon \right)
    \Rightarrow 
\hat \mu 
\,\,\text{ for  }  N \rightarrow \infty \text{, then } \varepsilon \searrow 0.
\]
Thus recall $\mu^N \Rightarrow \mu$ for $N \rightarrow \infty$ to conclude that we must have
\[
    \overline \mu^N_{\varepsilon} \Rightarrow \overline \mu := \mu - \hat \mu 
    \,\,\text{ for  }  N \rightarrow \infty \text{, then } \varepsilon \searrow 0
\]
where $\overline \mu$ corresponds to the mass that is acquired by stopping instantly. 
In other words
\begin{equation} \label{eq:Rost-0-conv}
    \mathcal{L}^{\lambda^N} \left(W^{(N)}_{\rho^{(N)}} ; \rho^{(N)} \leq \eta^{(N)} + \varepsilon \right)
    \Rightarrow 
    \mathcal{L}^{\lambda} \left(W_{\rho} ; \rho = \eta \right) 
    \,\,\text{ for  }  N \rightarrow \infty \text{, then } \varepsilon \searrow 0.
\end{equation}
The convergence \eqref{eq:Rost-0-conv} is equivalent to
\begin{equation} \label{eq:Rost-0-conv-2} 
    \mathbb{P}^{\lambda^N} \left[W^{(N)}_{\rho^{(N)}} \in A , \rho^{(N)} \leq \eta^{(N)} + \varepsilon \right]
    \xrightarrow[\varepsilon \searrow 0, N \rightarrow \infty]{} \mathbb{P}^{\lambda} \left[ W_{\rho} \in A, \rho = \eta \right]
\end{equation}
for all measurable $A$ such that $\mathbb{P}^{\lambda} \left[ W_{\rho} \in \partial A, \rho = \eta \right] = 0$.
As we have the convergence $\eta^{(N)} \rightarrow \eta$ in probability we can furthermore conclude
\begin{equation} \label{eq:Rost-0-conv-t} 
    \mathbb{P}^{\lambda^N} \left[W^{(N)}_{\rho^{(N)}} \in A , \rho^{(N)} \leq \eta^{(N)} + \varepsilon, |\eta^{(N)} - t| < \varepsilon \right]
    \xrightarrow[\varepsilon \searrow 0, N \rightarrow \infty]{} \mathbb{P}^{\lambda} \left[ W_{\rho} \in A, \rho = \eta , \eta = t\right]
\end{equation}
for all measurable $A$ such that $\mathbb{P}^{\lambda} \left[ W_{\rho} \in \partial A, \rho = \eta, \eta = t \right] = 0$. 

That $\rho$ must now share the desired representation on $\{\rho = \eta\}$ is a consequence of Lemma \ref{lem:Rost-uniqueness} and concludes the proof.
\end{proof}
%
%
%
%
%
%
In order to conclude the proof above we show a delayed Version of \cite[Lemma 5.5]{CoKi19b}.
\begin{lemma} \label{lem:5.5-Rost}
The Rost barriers $\bar R^{(N)}$ converge (possibly along a subsequence) to another Rost barrier $\bar R^{\infty}$ and 
\[
    \rho^N := \inf \left\{t \geq \eta: (t, W_t) \in \bar R^{(N)} \right\} \xrightarrow{p} 
              \inf \left\{t \geq \eta: (t, W_t) \in \bar R^{\infty}\right\} =: \rho^{\infty} 
\]
\end{lemma}
\begin{proof}
Note that both stopping times are subject to the \emph{same} delay $\eta$, 
hence we can repeat the proof of Lemma 5.5 in \cite{CoKi19b} with minimal adaptations. 
The arguments used therein rely on the existence of boundary functions 
$b:(0, \infty) \rightarrow \mathbb{R} \cup \{\infty\}$ and 
$c:(0, \infty) \rightarrow \mathbb{R} \cup \{-\infty\}$ 
such that the Rost hitting time can be represented as the first time the Brownian motions $(W_t)$ either rises above $b(t)$ or falls below $c(t)$. 

While in \cite[Lemma 5.5]{CoKi19b} the Brownian motion was assumed to start in $(0, W_0)$ it was possible to give this description using only a single upper boundary function $b$ and lower boundary function $c$. 
Due to the possibly delayed starting in $(\eta, W_{\eta})$ however, the processes might see different parts of the barrier that remain unseen for other starting positions. 
Hence we will no longer be able to describe the Rost hitting time with a single boundary function but as at most countably many such areas of the boundary can exists we can instead consider a family of boundary function 
$\{ b_n:(0, \infty) \rightarrow \mathbb{R} \cup \{\infty\}, n \in \mathbb{N} \}$ and 
$\{ c_n:(0, \infty) \rightarrow \mathbb{R} \cup \{-\infty\}, n \in \mathbb{N} \}$. 
Then for each starting position $(\eta, W_{\eta})$ we can identify the \emph{closest} boundary functions $b_{\eta}$ and $c_{\eta}$. 
Precisely, for each realization of $\eta$ we can define $b_{\eta}$ as the boundary function $b_n$ such that $W_{\eta} \leq b_n(\eta)$ and 
$\left|(\eta, W_{\eta}) - (\eta,b_n(\eta)) \right|$ is minimal.
There might be several boundary functions fulfilling these properties, however as they all have to coincide for $t\geq \eta$ the specific choice is irrelevant. 
Analogously we define $c_{\eta}$ to conclude the proof.
\end{proof}
%
%
%
%
%
%
We now give a delayed Version of \cite[Lemma 5.6]{CoKi19b}.
\begin{lemma}[Cf.{{\cite[Lemma 5.6]{CoKi19b}}}] \label{lem:5.6}
For $N \in \mathbb{N}$ consider a Rost barrier $\bar R^{(N)} \subseteq [0, \infty) \times \frac{1}{\sqrt{N}}\mathbb{Z}$ and  
let $\eta^{(N)} \in \frac{1}{N}\mathbb{N}$ and $\eta \in [0, \infty)$ be delay stopping times such that $\eta^{(N)} \xrightarrow{p} \eta$. 
Then for the stopping times
\begin{align*}
    \tau^{(N)} &:= \inf\left\{t \geq \eta^{(N)} : (t, W^{(N)}_t) \in \bar R^{(N)} \right\} \, \text{ and }
\\  \tau^N     &:= \inf\left\{t \geq \eta       : (t, W_t) \in \bar R^{(N)}\right\}
\end{align*}
we have the following convergence
\[
    \left| \left( \tau^{(N)}, W^{(N)}_{\tau^{(N)}} \right) - \left( \tau^{N}, W_{\tau^{N}}\right) \right| 
        \xrightarrow{p} 0 \text{ for } N \rightarrow \infty.
\]
\end{lemma}
\begin{proof}
First we give a slight refinement of the proof in \cite[Lemma 5.6]{CoKi19b} for $\eta = 0$.

As in \cite[Lemma 5.6]{CoKi19b} consider the Brownian motions $W^{\pm \varepsilon} = W_t \pm \varepsilon t$ with drift. 
Then we define the following stopping times 
\begin{align*}
    \rho^{+, N, \pm \varepsilon} &:= \inf \left\{ t \geq 0: W_t \geq W_0, \left(t, W^{\pm \varepsilon}_t \right)   \in \bar R^{(N)} \right\},
\\  \rho^{-, N, \pm \varepsilon} &:= \inf \left\{ t \geq 0: W_t \leq W_0, \left(t, W^{\pm \varepsilon}_t \right)   \in \bar R^{(N)} \right\},
\end{align*}
and 
\begin{align*}
    \rho^{+, N} &:= \inf \left\{ t \geq 0: W_t \geq W_0, \left(t, W_t \right)   \in \bar R^{(N)} \right\},
\\  \rho^{-, N} &:= \inf \left\{ t \geq 0: W_t \leq W_0, \left(t, W_t \right)   \in \bar R^{(N)} \right\},
\end{align*}
as well as
\begin{align*}
    \rho^{+, (N)} &:= \inf \left\{ t \geq 0: W^{(N)}_t \geq W^{(N)}_0, \left(t, W^{(N)}_t \right)   \in \bar R^{(N)} \right\},
\\  \rho^{-, (N)} &:= \inf \left\{ t \geq 0: W^{(N)}_t \leq W^{(N)}_0, \left(t, W^{(N)}_t \right)   \in \bar R^{(N)} \right\}.
\end{align*}
Furthermore for $T > 0$ and $\varepsilon > 0$ consider the set
\[
    A^{N, \varepsilon} = \left\{ \omega \in \Omega : \sup_{0 \leq t \leq T} \left|W_t - W^{(N)}_t \right| \leq \varepsilon \right\}
\]
and note that by Donsker's Theorem the probability of $\left( A^{N, \varepsilon} \right)^c$ goes to zero for $N \rightarrow \infty$. 

We have the following properties of the stopping times defined above. 
\begin{itemize}
\item[\itembullet] On the set $A^{N, \varepsilon}$ we have that  $|\rho^{+,(N)} - \rho^{+,N} | \leq |\rho^{+, N, +\varepsilon} - \rho^{+, N, - \varepsilon}|$ 
  as well as $|\rho^{-,(N)} - \rho^{-,N} | \leq |\rho^{-, N, +\varepsilon} - \rho^{-, N, - \varepsilon}|$.
\item[\itembullet] As demonstrated in the proof of \cite[Lemma 5.6]{CoKi19b} it follows from Girsanov's Theorem that 
    $|\rho^{+, N, +\varepsilon} - \rho^{+, N, - \varepsilon}| \xrightarrow{p} 0$ 
and 
    $|\rho^{-, N, +\varepsilon} - \rho^{-, N, - \varepsilon}| \xrightarrow{p} 0$ for $\varepsilon \searrow 0$.
\item[\itembullet] Now note that $\rho^{(N)} = \rho^{+,(N)} \wedge \rho^{-,(N)}$ and $\rho^{N} = \rho^{+,N} \wedge \rho^{-,N}$, 
thus on the set $A^{N, \varepsilon}$ we have
\begin{align*}
    |\rho^{(N)} - \rho^{N}| 
        &= |\rho^{+,N} \wedge \rho^{-,N} - \rho^{+,(N)} \wedge \rho^{-,(N)}|
\\      &\leq |\rho^{+,N} - \rho^{+,(N)}| + |\rho^{-,N} - \rho^{-,(N)}|    
\\      &\leq  |\rho^{+, N, +\varepsilon} - \rho^{+, N, - \varepsilon}| + |\rho^{-, N, +\varepsilon} - \rho^{-, N, - \varepsilon}|
\end{align*}
\item[\itembullet] Combining all these ingredients we can conclude that $|\rho^{(N)} - \rho^{N}| \xrightarrow{p} 0$ just as in \cite[Lemma 5.6]{CoKi19b}.
\end{itemize}

We now add a delay. 

Since $\eta^{(N)} \xrightarrow{p} \eta$ we can for ever $\varepsilon, \delta > 0$ find an $N_0 \in \mathbb{N}$ such that for
\[
    B^{N, \varepsilon} := \left\{\left|\eta^{(N)} - \eta \right| \leq \varepsilon \right\}
\]
we have $\mathbb{P}\left[ \left( B^{N, \varepsilon} \right)^c \right] \leq \delta$ for all $N \geq N_0$.
Define the set $C^{N, \varepsilon} := A^{N, \varepsilon} \cap B^{N, \varepsilon}$ for $A^{N, \varepsilon}$ defined above, 
then similarly $\mathbb{P}\left[ \left( C^{N, \varepsilon} \right)^c \right] \leq \tilde \delta$ for any $\tilde \delta > 0$ in $N$ is chosen big enough.

Define the following stopping times.
\begin{align*}
    \rho^{+, N, \pm \epsilon} &:= \inf \left\{ t \geq \eta: W_t \geq W_0, \left(t, W^{\pm \varepsilon}_t \right)   \in \bar R^{(N)} \right\},
\\  \rho^{-, N, \pm \epsilon} &:= \inf \left\{ t \geq \eta: W_t \leq W_0, \left(t, W^{\pm \varepsilon}_t \right)   \in \bar R^{(N)} \right\},
\end{align*}
and 
\begin{align*}
    \rho^{+, N} &:= \inf \left\{ t \geq \eta: W_t \geq W_0, \left(t, W_t \right)   \in \bar R^{(N)} \right\},
\\  \rho^{-, N} &:= \inf \left\{ t \geq \eta: W_t \leq W_0, \left(t, W_t \right)   \in \bar R^{(N)} \right\},
\end{align*}
as well as
\begin{align*}
    \rho^{+, (N)} &:= \inf \left\{ t \geq \eta^{(N)}: W^{(N)}_t \geq W^{(N)}_0, \left(t, W^{(N)}_t \right)   \in \bar R^{(N)} \right\},
\\  \rho^{-, (N)} &:= \inf \left\{ t \geq \eta^{(N)}: W^{(N)}_t \leq W^{(N)}_0, \left(t, W^{(N)}_t \right)   \in \bar R^{(N)} \right\}.
\end{align*}
Replacing the stopping times as well as $A^{N, \varepsilon}$ with $C^{N, \varepsilon}$ in the arguments above resp. in \cite[Lemma 5.6]{CoKi19b} we can conclude the desired convergence. 
\end{proof}
%
%
%
%
%
%
We show the final convergence via a sandwiching argument. 

For a given field of stopping probabilities $(r^N_t(x))_{(t,x) \in \mathbb{N} \times \frac{1}{\sqrt{N}} \mathbb{Z}}$ define
\begin{align*}
    \tilde R^{N}_+ &:= \left\{(t,x): r^N_t(x) > 0 \right\}, 
\\  \tilde R^{N}_- &:= \left\{(t,x): r^N_t(x) = 1 \right\}  =: \tilde R^N
\end{align*}
and let $\bar R^{(N)}_+$ and $\bar R^{(N)}_-$ respectively denote their (Rost) continuifications as defined in \eqref{eq:R^N-def}.  
Note that $\bar R^{(N)} = \bar R^{(N)}_- \subseteq \bar R^{(N)}_+$ and the two barriers coincide on all points $(t,x)$ such that  $r^N_t(x) \in \{0, 1\}$.
Furthermore define $\bar S^{(N)}_+ := \bar R^{(N)}_+ \setminus \left( \{0\} \times \mathbb{R} \right)$ 
as well as $\bar S^{(N)}_- := \bar R^{(N)}_- \setminus \left(\{0\} \times \mathbb{R}\right)$ and recall $d_R$, Roots metric on barriers. 
While $\bar R^{(N)}_+$ and $\bar R^{(N)}_-$ might differ substantially on $\{0\} \times \mathbb{R}$, 
for $\bar S^{(N)}_+$ and $\bar S^{(N)}_-$ by Rost's barrier structure and definition of the discretization we have 
\begin{equation} \label{eq:R+-conv}
    d_R\left( \bar S^{(N)}_+ , \bar S^{(N)}_- \right) \leq \frac{1}{N} \rightarrow 0 \,\,\text{ for } N \rightarrow \infty. 
\end{equation}
Hence both $\bar S^{(N)}_+$ and $\bar S^{(N)}_-$ converge to the same object $\bar S$ in Roots metric. 

Define the hitting times
\begin{align*}
    \tau^{(N)}_+ &:= \inf\left\{t \geq \eta^{(N)} : \left(t, W^{(N)}_t \right) \in \bar R^{(N)}_+ \right\},
\\  \tau^{(N)}_- &:= \inf\left\{t \geq \eta^{(N)} : \left(t, W^{(N)}_t \right) \in \bar R^{(N)}_- \right\}.
\end{align*}
Then both $\tau^{(N)}_+$ and $\tau^{(N)}_-$ are non-randomized hitting times and we will almost surely have
\[
    \tau^{(N)}_+ \leq \rho^{(N)} \leq \tau^{(N)}_-.
\]
We will also consider the hitting times 
\begin{align*}
    \tau^{N}_+ &:= \inf\left\{t \geq \eta : \left(t, W_t \right) \in \bar R^{(N)}_+ \right\},
\\  \tau^{N}_- &:= \inf\left\{t \geq \eta : \left(t, W_t \right) \in \bar R^{(N)}_- \right\}.
\end{align*}
These auxiliary stopping times will help us prove the following lemma.
\begin{lemma} \label{lem:5.6-2}
We have the convergence
\begin{equation} 
 \left|\left(\rho^{(N)}, W^{(N)}_{\rho^{(N)}} \right)\ind_{\rho^{(N)} > \eta^{(N)} + \varepsilon} 
    - \left( \tau^N, W_{\tau^N}\right) \ind_{\tau^N > \eta} \right|
    \xrightarrow[\varepsilon \searrow 0, N \rightarrow \infty]{p} 0. 
\end{equation}
\end{lemma}
\begin{proof}
Since 
\(
    \tau^{(N)}_+ \leq \rho^{(N)} \leq \tau^{(N)}_-
\) 
almost surely it suffices to show the two convergences
\begin{align}
    \left| \left( \tau^{(N)}_+, W^{(N)}_{\tau^{(N)}_+} \right) \ind_{\rho^{(N)} > \eta^{(N)} + \varepsilon}  
        - \left( \tau^{N}, W_{\tau^{N}}\right)  \ind_{\tau^N > \eta} \right| 
        &\xrightarrow[\varepsilon \searrow 0, N \rightarrow \infty]{p} 0  \,\,\text{ and}       \label{eq:conv11}
\\     \left| \left( \tau^{(N)}_-, W^{(N)}_{\tau^{(N)}_-} \right) \ind_{\rho^{(N)} > \eta^{(N)} + \varepsilon} 
        - \left( \tau^{N}, W_{\tau^{N}}\right) \ind_{\tau^N > \eta} \right| 
        &\xrightarrow[\varepsilon \searrow 0, N \rightarrow \infty]{p} 0. \label{eq:conv12}
\end{align}
Lemma \ref{lem:5.6} gives us the following convergence
\begin{align}
    \left| \left( \tau^{(N)}_+, W^{(N)}_{\tau^{(N)}_+} \right) - \left( \tau^{N}_+, W_{\tau^{N}_+}\right) \right| 
        \xrightarrow[N \rightarrow \infty]{d} 0 \label{eq:conv21}
\end{align}
as well as 
\begin{align}
    \left| \left( \tau^{(N)}_-, W^{(N)}_{\tau^{(N)}_-} \right) - \left( \tau^{N}_-, W_{\tau^{N}_-}\right) \right| 
        \xrightarrow[N \rightarrow \infty]{d} 0. \label{eq:conv22}
\end{align}
Note that \eqref{eq:conv21} already implies \eqref{eq:conv11} as $\tau^N = \tau^N_+$.

To see the convergence \eqref{eq:conv12}, observe that on $\{\rho^{(N)} > \eta^{(N)} + \varepsilon\}$ we will also have $\tau^{(N)}_- > \eta^{(N)} + \varepsilon$. 
Furthermore, since $\tau^N =\tau^N_+ \leq \tau^N_-$ on $\{\tau^N > \eta \}$ we thus also have $\tau^N_- > \eta$.
Hence it is a consequence of \eqref{eq:R+-conv} and Lemma \ref{lem:5.5-Rost} that
\[
    \left|\tau^{N}_+ \ind_{\tau^{N} > \eta} - \tau^{N}_- \ind_{\tau^{N} > \eta} \right| \xrightarrow{p} 0.
\]
This together with  \eqref{eq:conv22} implies convergence \eqref{eq:conv12}, concluding the proof.
\end{proof}
Lemma \ref{lem:main-convergence-Root} (resp. \ref{lem:main-convergence-Rost}) now enables the first step of recovering the continuous optimal stopping representation of Theorem \ref{thm:delayed-Root} (resp. \ref{thm:delayed-Rost}) 
from the discrete counterpart Theorem \ref{thm:discrete-delayed-Root} (resp. \ref{thm:discrete-delayed-Rost}). 
\begin{corollary} \label{cor:LHS-conv}
Let $(\lambda^N, \tilde \eta ^N, \mu^N)$ denote a (rescaled) discrete \eqref{dSEP} approximating the continuous \eqref{dSEP}  given by $(\lambda, \eta, \mu)$ 
as defined in Section \ref{sec:(dSEP)-discretization}.  
Let $\rho^{(N)}$ denote a Root (resp. Rost) solution to $(\lambda^N, \tilde \eta ^N, \mu^N)$ while $\rho$ denotes a Root (resp. Rost) solution to $(\lambda, \eta, \mu)$. 
Then for every $T \in [0, \infty)$ we have the convergence
\begin{equation} \label{eq:mu_T-conv}
    \mu^N_T = \mathcal{L}^{\lambda^N}\left(W^{(N)}_{\rho^{(N)} \wedge T}\right) 
        \Rightarrow \mathcal{L}^{\lambda} \left(W_{\rho \wedge T}\right) = \mu_T.
\end{equation}
Moreoever this gives convergence of the LHS of \ref{eq:discrete-delayed-Root} (resp. \ref{eq:discrete-delayed-Rost})  to the LHS of \ref{eq:delayed-Root} (resp. \ref{eq:delayed-Rost}).
\end{corollary}
%
%
%
%
%
%
%
%
%
%
%
\subsection{A Discretization of the Optimal Stopping Problem and its Convergence} \label{sec:OSP-convergence}
It remains to show convergence of the discrete optimal stopping problem 
to the continuous one. 

Les us consider a continous \eqref{dSEP} determined by  $(\lambda, \eta, \mu)$.
We will first prove convergence of the Root case, the Rost case can then be argued analogously. 

To avoid heavy usage of floor functions, we will assume from now on for the cutoff time $T \in I:=\left\{ \frac{m}{2^n}:m,n\in\mathbb{N} \right\}$.
Then taking limits along the subsequence $\left(Y^{2^n}\right)_{n\in\mathbb{N}}$ 
(resp. $\left(W^{(2^n)}\right)_{n\in\mathbb{N}}$) ensures that there exists an $N_0 \in \mathbb{N}$ such that 
$T$ will always be a multiple of the step size $\frac{1}{2^n}$ for all $n \geq N_0$.
For arbitrary $T>0$ the results can be recovered via density arguments.

With the help of the function
\begin{align*}
    G_T(x,t)   &:= V^{\alpha}_{T-t}(x) + \left(U_{\mu}(x) - U_{\alpha_X}(x)\right) \ind_{t < T}
\end{align*}
we can recall the continuous identities of interest
\begin{align}
    \label{eq:delayed-Root-G}
     U_{\mu_T}(x) &= \E^x\left[G_T\left(W_{\tau^*},\tau^*\right) \right]   \tag{\ref{eq:delayed-Root}}
\\ \label{eq:delayed-RootOSP-G}
                            &= \sup_{\tau \leq T} \E^x\left[G_T\left(W_{\tau},\tau\right) \right]  \tag{\ref{eq:delayed-RootOSP}}
\end{align}
where the optimizer is given by 
\[
\tau^* := \tau_T \wedge T \text{ for } \tau_T :=\inf \{t \geq 0 : (T-t, W_t) \not\in R \}
\]
for a Root barrier $R$.

Let now $(\lambda^N, \tilde \eta ^N, \mu^N)$ denote a (rescaled) discrete \eqref{dSEP} approximating the continuous \eqref{dSEP} 
as defined in Section \ref{sec:(dSEP)-discretization}. 

Note that Lemma \ref{lem:limit} gives us the following additional convergence results
\begin{align}
    \alpha^N_X := &\mathcal{L}^{\lambda^N}(Y^N_{\tilde \eta^N}) \rightarrow \alpha_X \,\,\text{ and }
\\                &\mathcal{L}^{\lambda^N}(Y^N_{\tilde \eta^N \wedge TN}) \rightarrow \mathcal{L}^{\lambda}(W_{\eta \wedge T}). \label{eq:V^N-conv}
\end{align}
Thus we are able to define a discrete version of $G_T$ in the following way 
\begin{align*}
  G^N_T(x,t) &:= V^{N}_{T-t}(x) + \left(U_{\mu^N}(x) - U_{\alpha^N_X}(x)\right) \ind_{t < T}
\end{align*}
where
\begin{align*}
    V^{N}_{T}(x) &:= \E^{\lambda^N} \left[ \left| Y^N_{\tilde \eta ^N \wedge NT} - x \right| \right].
\end{align*}
Here $V^{N}_{T}$ (resp. $V^{\alpha}_{T}$) is the potential function with respect to the LHS (resp. RHS) of \eqref{eq:V^N-conv}, 
hence we have uniform convergence of $G^N_T$ to $G_T$ due to the uniform convergence of all the potential functions involved.

Theorem \ref{thm:ddSEP-Root} gives existence of a Root field of stopping probabilities $(r^N_t(x))_{(t,x) \in \mathbb{N} \times \mathcal{X}^N}$ 
such that the corresponding discrete delayed Root stopping time $\tilde \rho^{N}$ embeds the measure $\mu^N$.
Let 
\[
    \tilde R^N := \left\{(t,x) \in \mathbb{N} \times \mathcal{X}^N : r^N_t(x) > 0 \right\}
\]
denote a Root barrier induced by the field of stopping probabilities 
and let $R^{(N)}$  denote its continuification as defined in \eqref{eq:R^N-def}.

The rescaled version of the results in Theorem \ref{thm:discrete-delayed-Root} then read
\begin{align}
    \label{eq:discrete-delayed-Root-G}
     U_{\mu^{N}_T}(x) &= \E^x\left[G^N_T\left(Y^N_{\tilde\tau^{N*}}, \frac{\tilde\tau^{N*}}{N}\right) \right]   \tag{\ref{eq:discrete-delayed-Root}*}
\\ \label{eq:discrete-delayed-RootOSP-G}
                            &= \sup_{\frac{\tau}{N} \leq T} 
                           \E^x\left[G^N_T\left(Y^N_{\tau}, \frac{\tau}{N}\right) \right],  \tag{\ref{eq:discrete-delayed-RootOSP}*}
\end{align}
and an optimizer is given by
\[
\tilde{\tau}^{N*} := \tilde{\tau}^N_T \wedge NT \,\,\text{ for }\,\, 
    \tilde{\tau}^N_T :=\inf \{t \in \mathbb{N} : (NT-t, Y^N_t) \in \tilde{R}^N \}.
\]
To conclude the proof of Theorem \ref{thm:delayed-Root}, 
it remains to establish the following two convergence results. 
\begin{lemma} 
\hfill
\begin{enumerate}[(i)]
\item
We have convergence of the RHS of \eqref{eq:discrete-delayed-Root-G} to the RHS of \ref{eq:delayed-Root-G}, 
precisely
\begin{equation*}
     \E^x\left[G^N_T\left(Y^N_{\tilde\tau^{N*}}, \frac{\tilde\tau^{N*}}{N}\right) \right] \xrightarrow[]{N \rightarrow \infty}
    \E^x\left[G_T\left(W_{\tau^*},\tau^*\right) \right]. 
\end{equation*}
\item We furthermore have convergence of the RHS of \eqref{eq:discrete-delayed-RootOSP-G} to the RHS of \ref{eq:delayed-RootOSP-G}, 
precisely
\begin{equation*}
    \sup_{\frac{\tau}{N} \leq T} \E^x\left[G^N_T\left(Y^N_{\tau}, \frac{\tau}{N}\right) \right] 
        \xrightarrow[]{N \rightarrow \infty}
    \sup_{\tau \leq T} \E^x\left[G_T\left(W_{\tau},\tau\right) \right].
\end{equation*}
\end{enumerate}
\end{lemma}
\begin{proof}
So see \emph{(i)}, recall the continuous rescaled version $W^{(N)}$ of $Y^N$.
Then due to \cite{CoKi19b}, Section 5, alternatively 
also due to Lemma \ref{lem:main-convergence-Rost} for the trivial delay $\eta^{(N)} = \eta = 0$
and additionally considering the Root-Rost symmetry 
we have the following Donsker type convergence 
\begin{equation}
    \left(\tau^{(N)*}, W^{(N)}_{\tau^{(N)*}} \right) \xrightarrow{d} \left(\tau^{*}, W_{\tau^{*}}\right)  \text{  as } N \rightarrow \infty,
\end{equation}
where
\[
    \tau^{(N)*} := \inf \left\{ t \geq 0 : \left(T-t, W^{(N)}_t\right) \in R^{(N)} \right\} \wedge T
\]
is the backwards Root stopping time of $W^{(N)}$ and
\[
    \tau^{*} := \inf \left\{ t \geq 0 : \left(T-t, W_t \right) \in R \right\} \wedge T
\]
is the backwards Root stopping time of the Brownian motion. 
Furthermore, note that by the definition of the discretization we have
\begin{equation*}
    \left(\tau^{(N)*}, W^{(N)}_{\tau^{(N)*}} \right) = \left( \frac{\tilde \tau^{N*}}{N}, Y^N_{\tilde \tau^{N*}}\right).
\end{equation*}
Since all other crucial properties - that $G_T$ is a lower semi-continuous function and that $G_T^N$ converges uniformly to $G_T$ - 
are also given here, we can repeat the convergence proof given in \cite{BaCoGrHu21}, Section 5 verbatim.

It is left to show \emph{(ii)}, convergence of the optimal stopping problems. 
Trivially, we have
\[
    \E^x \left[G^T(W_{\tau^*}, \tau^*)\right] \leq \sup_{\tau \leq T} \E^x \left[ G^T(W_{\tau}, \tau) \right].
\]
Hence it remains to show that 
\[
    \sup_{\tau \leq T} \E^x \left[ G^T(W_{\tau}, \tau) \right] \leq \E^x \left[G^T(W_{\tau^*}, \tau^*)\right].
\]
Let $\bar{\tau}$ be an optimizer of the continuous optimal stopping problem \eqref{eq:delayed-RootOSP}. 
Consider the auxiliary functions 
\begin{align*}
    \bar{G}^{\varepsilon}_T(x,t)     &:= V^{\alpha}_{T-t}(x) + \left(U_{\mu}(x) - U_{\alpha_X}(x)\right) \ind_{t \leq T - \varepsilon} \,\, \text{ and }
\\  \tilde{G}^{N,\varepsilon}_T(x,t) &:= V^{N}_{T-t}(x) + \left(U_{\mu^N}(x) - U_{\alpha^N_X}(x)\right) \ind_{t \leq T - \varepsilon}.
\end{align*}
Then for any $\delta > 0$ we have the following
\begin{align}
    \sup_{\tau \leq T} \E^x \left[ G_T(W_{\tau}, \tau) \right] - \delta
        &< \E^x \left[G_T  \left(W_{\bar\tau}, \bar\tau\right) \right]      \label{1.}              
\\      &= \lim_{\varepsilon \searrow 0} \E^x \left[\bar{G}^{\varepsilon}_T (W_{\bar\tau},\bar\tau)\right]     \label{2.}
\\      &\leq \lim_{\varepsilon \searrow 0} \liminf_{N \rightarrow \infty} 
            \E^x \left[\tilde{G}^{N,\varepsilon}_T\left(Y^N_{\tilde \sigma^N}, \frac{\tilde \sigma^N}{N} \right) \right]     \label{3.}
\\      &\leq \lim_{\varepsilon \searrow 0} \liminf_{N \rightarrow \infty} 
            \sup_{\frac{\tau}{N} \leq T} \E^x \left[ \tilde{G}^{N,\varepsilon}_T \left(Y^N_{\tau}, \frac{\tau}{N}\right) \right]     \label{4.}
\\      &\leq \lim_{\varepsilon \searrow 0} \liminf_{N \rightarrow \infty} \sup_{\frac{\tau}{N} \leq T- \varepsilon } 
            \E^x \left[ G^{N}_{T-\varepsilon} \left(Y^N_{\tau}, \frac{\tau}{N}\right) \right]     \label{5.}
\\      &= \lim_{\varepsilon \searrow 0} \liminf_{N \rightarrow \infty} 
            \E^x \left[ G^{N}_{T-\varepsilon} \left(Y^N_{\tilde\tau^{N*}_{T-\varepsilon}}, \frac{\tilde\tau^{N*}_{T-\varepsilon}}{N}\right) \right]      \label{6.}
\\      &= \lim_{\varepsilon \searrow 0} 
            \E^x \left[ G_{T-\varepsilon} \left(W_{\tau^{*}_{T-\varepsilon}}, \tau^{*}_{T-\varepsilon}\right) \right]     \label{7.}
\\      &= \lim_{\varepsilon \searrow 0} U_{\mu_{T-\varepsilon}}(x)       \label{8.}
\\      &= U_{\mu_{T}}(x) = \E^x \left[ G \left(W_{\tau^*}, \tau^*\right) \right].       \label{9.}
\end{align}
We justify this chain of inequalities and equalities step by step. 

The inequality in \eqref{1.} is clear as $\bar{\tau}$ was assumed to be an optimizer.
%
For the difference between $G_T$ and $\bar{G}^{\varepsilon}_T$ we have
\[
    |G_T(x,t) - \bar{G}^{\varepsilon}_T(x,t)| = \underbrace{|U_{\mu}(x) - U_{\alpha_X}(x)|}_{\leq c < \infty} \1_{T-\varepsilon < t < T},
\]
hence for a random variable $X$ and stopping time $\tau$ we have
\begin{equation*} 
 \E^x \left[|G_T(X,\tau) - \bar{G}^{\varepsilon}_T(X,\tau)|\right] \leq c \cdot \mathbb{P}\left[ \tau \in (T-\varepsilon, T)\right].   
\end{equation*}
Now since $\lim_{\varepsilon \searrow 0}\mathbb{P}\left[ \tau \in (T-\varepsilon, T)\right] = 0$ we can conclude
\[
 \lim_{\varepsilon \searrow 0} \E^x \left[|G_T(W_{\bar\tau},\bar\tau) - \bar{G}^{\varepsilon}_T(W_{\bar\tau},\bar\tau)|\right] = 0,
\]
which gives \eqref{2.}. 
%
The inequality in \eqref{3.} follows from uniform convergence of $\tilde{G}^{N, \varepsilon}_T$ to $\bar{G}^{\varepsilon}_T$ and the fact that $\tilde{G}^{N, \varepsilon}_T$ is lower semi-continuous.
%
In \eqref{4.} we simply take a supremum. 
%
To see \eqref{5.}, let us extend the function $V^N_{T} (x)$ to $T<0$ in the following way. 
\begin{equation} \label{eq:V-extension}
    V^{N}_{T}(x) := \begin{cases} 
                - \E^{\lambda^N} \left[ \left| Y^N_{\tilde \eta ^N \wedge NT} - x \right| \right] & \text{ for } T \geq 0 \\
                U_{\lambda^N}(x) & \text{ for } T < 0.
                     \end{cases}
\end{equation}
Consider $t \leq u$, then $T-u \leq T-t$, thus also $\tilde\eta^N \wedge (T-u) \leq \tilde\eta^N \wedge (T-t)$. 
Due to Jensen and Optional Sampling we then have
\[
   - V^N_{T-u} (x) 
    =    \E^{\lambda^N} \left[\left| Y^N_{\tilde \eta^N \wedge (T-u)} - x  \right| \right] 
    \leq \E^{\lambda^N} \left[\left| Y^N_{\tilde \eta^N \wedge (T-t)} - x  \right| \right] = - V^N_{T-t} (x),
\]
or equivalently
\[
   V^N_{T-t} (x) \leq V^N_{T-u} (x).
\]
Especially, we have for $\varepsilon > 0$ that
\[
   V^N_{T-t} (x) \leq V^N_{T-\varepsilon-t} (x),
\]
hence the function $v(t,x):= V^N_{T-t} (x)$ is \emph{increasing} in $t$ and convex in $x$.
Now 
\begin{align}
    \sup_{\frac{\tau}{N} \leq T} \E^x \left[ \tilde{G}^{N,\varepsilon}_T \left(Y^N_{\tau}, \frac{\tau}{N}\right) \right] 
        &= \sup_{\frac{\tau}{N} \leq T} 
            \E^x \Big[ 
                \underbrace{ V^{N}_{T - \tau}\left(Y^N_{\tau}\right) }_{\leq V^{N}_{T - \varepsilon - \tau}\left(Y^N_{\tau}\right)}
                    + 
                \underbrace{ \left(U_{\mu^N} - U_{\alpha^N_X}\right)\left(Y^N_{\tau}\right) \ind_{\frac{\tau}{N} \leq T - \varepsilon} }_{
                    \leq \left(U_{\mu^N} - U_{\alpha^N_X}\right)\left(Y^N_{\tau}\right) \ind_{\frac{\tau}{N} < T - \varepsilon}
}
            \Big] \nonumber
\\      &\leq \sup_{\frac{\tau}{N} \leq T}
            \left[ V^{N}_{T - \varepsilon - \tau}\left(Y^N_{\tau}\right) + \left(U_{\mu^N} - U_{\alpha^N_X}\right)\left(Y^N_{\tau}\right) \ind_{\frac{\tau}{N} < T - \varepsilon} 
            \right] \label{eq:new-OSP}
\\      &\leq \sup_{\frac{\tau}{N} \leq T} \E^x \left[ \tilde{G}^{N,\varepsilon}_{T-\varepsilon} \left(Y^N_{\tau}, \frac{\tau}{N}\right) \right] 
            \label{eq:new-OSP-G}
\end{align}
It remains to show that no optimizer of \eqref{eq:new-OSP} resp. \eqref{eq:new-OSP-G} will stop after time $T-\varepsilon$. 
By definition \eqref{eq:V-extension} we have $V^{N}_{T - \varepsilon - \tau}\left(Y^N_{\tau}\right) = V^{N}_{T - \varepsilon - \tau \wedge (T-\varepsilon)}\left(Y^N_{\tau}\right)$ since we cannot do better than $V^{N}_{0}\left(Y^N_{\tau}\right) = U_{\lambda^N}\left(Y^N_{\tau}\right)$. 
Now, let $(Z_t)_{t \geq 0}$ be a martingale. 
Then we consider 
    $\left( \tilde{G}^{N,\varepsilon}_{T-\varepsilon} \left(Z_t, t\right) \right)_{t \in [T-\varepsilon, T]} = 
        \left( U_{\lambda^N}\left(Z_t\right) \right)_{t \in [T-\varepsilon, T]}$.
As $U_{\lambda^N}$ is a concave function, $\left( U_{\lambda^N}\left(Z_t\right) \right)_{t \in [T-\varepsilon, T]}$ is a supermartingale and for any stopping time $\tau$ we have
\[
    \E^x \left[ \tilde{G}^{N,\varepsilon}_{T-\varepsilon} \left(Z_{\tau \wedge (T-\varepsilon)}, \tau \wedge (T-\varepsilon) \right) \right]
    \geq \E^x \left[ \tilde{G}^{N,\varepsilon}_{T-\varepsilon} \left(Z_{\tau \wedge T}, \tau \wedge T\right) \right].
\]
Hence we see that no optimizer of \eqref{eq:new-OSP-G} will stop after time $T-\varepsilon$ which concludes \eqref{5.}.
%
Due to Theorem \ref{thm:discrete-delayed-Root} we know that 
\[
\tilde\tau^{N*}_{T-\varepsilon} := \inf\{t \in \mathbb{N} : (N(T-\varepsilon)-t, Y_t^N) \in \tilde{R}^N\} \wedge N(T-\varepsilon)
\]
is an optimizer of the optimal stopping problem 
\[
\sup_{\frac{\tau}{N} \leq T- \varepsilon } \E^x \left[ G^{N}_{T-\varepsilon} \left(Y^N_{\tau}, \frac{\tau}{N}\right) \right]
\]
which gives equality in  \eqref{6.}. 
%
The convergence \emph{(i)} together with Corollary \ref{cor:LHS-conv} gives equality in \eqref{7.} and \eqref{8.}.
%
The convergence \eqref{9.} follows by considering the barrier $R$ such that the corresponding Root stopping time $\rho$ embeds the measure $\mu$. 
Then the measure $\mu_T$ will be embedded by 
$\rho_T := \inf\{t \geq \eta :  (t, W_t) \in R \cup ([T, \infty) \times \mathbb{R}) \}$.
Obviously, we have that $R \cup ([T, \infty) \times \mathbb{R})$ converges to $R \cup ([T, \infty) \times \mathbb{R})$ in Root's barrier distance for $\varepsilon \searrow 0$ and convergence of the corresponding measures follows from Lemma \ref{lem:delayed-original-Root}. 
\end{proof}

This concludes the convergence of the discrete optimal stopping representation to the continuous optimal stopping representation, 
hence completing the proof of Theorem \ref{thm:delayed-Root}.

The convergence in the Rost case now follows analogously, choosing the function $G_T$ and $G^N_T$ as
\begin{align*}
    G_T(x,t)   &:= U_{\mu}(x) - V^{\alpha}_{T-t}(x) \,\, \text{ and }
\\  G_T^N(x,t) &:= U_{\mu^N}(x) - V^{N}_{T-t}(x),
\end{align*}
concluding the proof of Theorem \ref{thm:delayed-Rost}.

\bibliographystyle{abbrv}
\bibliography{../MBjointbib/joint_biblio}
\end{document}